\documentclass[11pt]{amsart}
\usepackage{eurosym}
\usepackage{amsfonts}
\usepackage{amssymb}
\usepackage{graphicx}
\usepackage{amsmath,amsfonts,amsmath,amssymb}
\usepackage{mathrsfs}
\usepackage[latin1]{inputenc}
\usepackage[T1]{fontenc}
\usepackage{color,xspace}

\usepackage{tikz}
\usetikzlibrary{matrix,positioning,decorations.pathreplacing}

\usepackage[colorlinks=true]{hyperref}

\usepackage[usenames,dvipsnames]{pstricks}

\newtheorem{theorem}{Theorem}[section]

\newtheorem{corollary}[theorem]{Corollary}

\newtheorem{definition}[theorem]{Definition}

\newtheorem{lemma}[theorem]{Lemma}

\newtheorem{proposition}[theorem]{Proposition}

\theoremstyle{definition}
\newtheorem{remark}[theorem]{Remark}

\newcommand{\R}{\mathbb{R}}
\newcommand{\C}{\mathbb{C}}
\newcommand{\Z}{\mathbb{Z}}
\newcommand{\N}{\mathbb{N}}

\newcommand{\zb}{\mathbf{z}}
\newcommand{\tz}{\tilde{z}}
\newcommand{\tzb}{\tilde{\zb}}

\renewcommand{\v}{\mathbf{v}}

\newcommand{\w}{\mathbf{w}}
\newcommand{\tw}{\tilde{\w}}

\newcommand{\p}{\partial}
\renewcommand{\d}{\partial}
\newcommand{\U}{\mathscr U}
\newcommand{\V}{\mathscr V}
\newcommand{\W}{\mathscr W}
\newcommand{\B}{\mathscr B}
\newcommand{\G}{\mathscr G}
\newcommand{\Norm}[1]{\left\Vert#1\right\Vert}
\newcommand{\bigNorm}[1]{\big\Vert#1\big\Vert}

\newcommand{\norm}[1]{\left| #1 \right|}
\newcommand{\bignorm}[1]{\big| #1 \big|}

\newcommand{\abs}[1]{\left|#1\right|}
\newcommand{\bigabs}[1]{\big|#1\big|}
\newcommand{\Bigabs}[1]{\Big|#1\Big|}

\newcommand{\set}[1]{\left\{#1\right\}}

\newcommand{\para}[1]{\left(#1\right)}
\newcommand{\bigpara}[1]{\big(#1\big)}

\newcommand{\cro}[1]{\left[#1\right]}

\newcommand{\y}{\varrho}
\newcommand{\hy}{\hat{\varrho}}
\newcommand{\ty}{\tilde{\varrho}}
\newcommand{\uu}{\underline{u}}
\newcommand{\vv}{\underline{v}}

\newcommand{\hchi}{\hat{\chi}}
\newcommand{\tchi}{\tilde{\chi}}
\newcommand{\htchi}{\hat{\tilde{\chi}}}

\newcommand{\ha}{\hat{a}}

\let \Re \relax
\DeclareMathOperator{\Re}{Re}
\let \Im \relax
\DeclareMathOperator{\Im}{Im}

\newcommand{\id}{\textrm{Id}}
\newcommand{\ovl}[1]{\overline{#1}}
\newcommand{\udl}[1]{\underline{#1}}

\newcommand{\imp}{\Rightarrow}

\newcommand{\Con}{\ensuremath{\mathscr C}}
\newcommand{\Cinf}{\ensuremath{\Con^\infty}}
\newcommand{\Cinfc}{\ensuremath{\Con^\infty}_{c}}
\renewcommand{\S}{\ensuremath{\mathscr S}}

\newcommand{\Rp}{\R^n_+}
\newcommand{\Rpb}{\ovl{\R}^n_+}

\newcommand{\Ssc}{S_\tau}
\newcommand{\Psisc}{\Psi_\tau}

\newcommand{\T}{{\mbox{\tiny ${\mathsf T}$}}}
\newcommand{\Ssct}{S_{\T,\tau}}
\newcommand{\Psisct}{\Psi_{\T,\tau}}

\newcommand{\Ssctn}{S_{\T,\ttau}}

\newcommand{\Sscl}{S_{\tau,\mathrm{cl}}}
\newcommand{\Psiscl}{\Psi_{\tau,\mathrm{cl}}}

\newcommand{\Scl}{S_{\mathrm{cl}}}
\newcommand{\Psicl}{\Psi_{\mathrm{cl}}}

\newcommand{\Ssctcl}{S_{{\mathsf T},\tau,\mathrm{cl}}}

\newcommand{\Psisctcl}{\Psi_{{\mathsf T},\tau,\mathrm{cl}}}

\newcommand{\lambdat}{\lambda_\T}
\newcommand{\Lambdat}{\Lambda_\T}

\newcommand{\mulr}{\mu_\lr}

\newcommand{\mut}{\mu_\T}
\newcommand{\mutlr}{\mu_{\T,\lr}}
\newcommand{\mutl}{\mu_{\T,\l}}
\newcommand{\mutr}{\mu_{\T,r}}

\newcommand{\glr}{g_\lr}

\newcommand{\gt}{g_\T}
\newcommand{\gtlr}{g_{\T,\lr}}

\newcommand{\sphbundle}{{\mathbb S}^\ast_{\T,\tau}}

\newcommand{\lr}{{\mbox{\scriptsize $^r\!\!\slash_{\!\ell}$}}}
\renewcommand{\l}{\ell}

\newcommand{\Pl}{P_\ell}
\renewcommand{\Pr}{P_r}
\newcommand{\Plr}{P_\lr}

\newcommand{\Tl}{T_\ell}
\newcommand{\Tr}{T_r}
\newcommand{\Tlr}{T_\lr}

\newcommand{\Plv}{P_{\ell,\varphi}}
\newcommand{\Prv}{P_{r,\varphi}}
\newcommand{\Plrv}{P_{\lr,\varphi}}

\newcommand{\Tlv}{T_{\ell,\varphi}}
\newcommand{\Trv}{T_{r,\varphi}}
\newcommand{\Tlrv}{T_{\lr,\varphi}}

\newcommand{\Elv}{E_{\ell,\varphi}}
\newcommand{\Erv}{E_{r,\varphi}}
\newcommand{\Elrv}{E_{\lr,\varphi}}

\newcommand{\pl}{p_\ell}
\newcommand{\pr}{p_r}
\newcommand{\plr}{p_\lr}

\newcommand{\tl}{t_\ell}

\newcommand{\tlr}{t_\lr}

\newcommand{\vl}{\varphi_\ell}
\newcommand{\vr}{\varphi_r}
\newcommand{\vlr}{\varphi_\lr}

\newcommand{\ul}{u_\ell}
\newcommand{\ur}{u_r}

\newcommand{\fl}{f_\ell}
\newcommand{\fr}{f_r}

\newcommand{\plv}{p_{\ell,\varphi}}
\newcommand{\prv}{p_{r,\varphi}}
\newcommand{\plrv}{p_{\lr,\varphi}}

\newcommand{\klv}{\kappa_{\ell,\varphi}}
\newcommand{\krv}{\kappa_{r,\varphi}}
\newcommand{\klrv}{\kappa_{\lr,\varphi}}

\newcommand{\tlv}{t_{\ell,\varphi}}
\newcommand{\trv}{t_{r,\varphi}}
\newcommand{\tlrv}{t_{\lr,\varphi}}

\newcommand{\tlp}{t_{\ell,\psi}}
\newcommand{\trp}{t_{r,\psi}}

\newcommand{\elv}{e_{\ell,\varphi}}
\newcommand{\erv}{e_{r,\varphi}}
\newcommand{\elrv}{e_{\lr,\varphi}}

\newcommand{\elp}{e_{\ell,\psi}}
\newcommand{\erp}{e_{r,\psi}}

\newcommand{\tplv}{{\tilde p}_{\ell,\varphi}}
\newcommand{\tprv}{{\tilde p}_{r,\varphi}}
\newcommand{\tplrv}{{\tilde p}_{\lr,\varphi}}

\DeclareMathOperator{\supp}{supp}

\DeclareMathOperator{\Op}{\ensuremath{Op}}

\newcommand{\br}{_{|x_n=0^+}}

\newcommand{\de}{\delta}
\newcommand{\eps}{\varepsilon}

\newcommand{\hf}{\frac{1}{2}}

\newcommand{\transp}{\ensuremath{\phantom{}^{t}}}

\DeclareMathOperator{\rank}{rank}

\newcommand{\bld}[1]{\mbox{\boldmath $#1$}}

\newcommand{\wrt}{w.r.t.\@\xspace}
\newcommand{\rhs}{r.h.s.\@\xspace}
\newcommand{\lhs}{l.h.s.\@\xspace}
\newcommand{\ie}{i.e.\@\xspace}
\newcommand{\viz}{viz.\@\xspace}

\newcommand{\eg}{e.g.\@\xspace}

\newcommand{\resp}{resp.\@\xspace}
\newcommand{\suff}{sufficiently\xspace}
\newcommand{\nhd}{neighborhood\xspace}
\newcommand{\nhds}{neighborhoods\xspace}
\newcommand{\st}{such that\xspace}

\usepackage{times}

\newcommand{\spc}{strongly pseudo-convex\xspace}
\newcommand{\spcty}{strong pseudo-convexity\xspace}

\newcommand{\ttau}{\tilde{\tau}}
\newcommand{\htau}{\hat{\tau}}

\newcommand{\ttaulr}{\tilde{\tau}_\lr}

\newcommand{\ttaul}{\tilde{\tau}_\l}

\newcommand{\ttaur}{\tilde{\tau}_r}

\newcommand{\varphilr}{\varphi_\lr}
\newcommand{\varphil}{\varphi_\l}
\newcommand{\varphir}{\varphi_r}

\newcommand{\psilr}{\psi_\lr}
\newcommand{\psil}{\psi_\l}
\newcommand{\psir}{\psi_r}

\DeclareMathSymbol{\intop}{\mathop}{symbols}{115}
\def\int{\intop}
\let \iint \relax
\DeclareMathOperator*{\iint}{\intop\!\!\intop}

\DeclareMathSymbol{\sumop}{\mathop}{largesymbols}{80}
\def\ssum{\sumop}
\let \sum \relax
\DeclareMathOperator*{\sum}{\textstyle{\ssum}}

\DeclareMathSymbol{\prodop}{\mathop}{largesymbols}{81}
\def\sprod{\prodop}
\let \prod \relax
\DeclareMathOperator*{\prod}{\textstyle{\sprod}}

\usepackage{epsfig}

\newcommand{\csp}{\gamma}

\newcommand{\mi}{\alpha}
\newcommand{\fmi}{\gamma}
\newcommand{\smi}{\beta}

\newcommand{\torder}{\beta}

\DeclareMathOperator{\trace}{tr}

\begin{document}

\title[Elliptic transmission problems]
{Carleman estimates for elliptic operators with complex
  coefficients\\Part II: transmission problems}
\author[M.~Bellassoued and J.~Le~Rousseau]
{Mourad~Bellassoued and J\'er\^ome~Le~Rousseau}

\address{M.~Bellassoued. University of Tunis El Manar, National Engineering School of Tunis, ENIT-LAMSIN, B.P. 37, 1002 Tunis, Tunisia.}
\email{mourad.bellassoued@fsb.rnu.tn}

\address{ J.~Le~Rousseau. Laboratoire de math\'ematiques - analyse, probabilit\'es, mod\'elisation - Orl\'eans, CNRS UMR 7349, F\'ed\'eration
Denis-Poisson, FR CNRS 2964, Universit\'e d'Orl\'eans, B.P. 6759, 45067 Orl\'eans cedex 2, France.
Institut universitaire de France.}
\email{jlr@univ-orleans.fr}

\date{\today}

\thanks{The second author acknowledges support from Agence Nationale de
  la Recherche (grant ANR-13-JS01-0006 - iproblems - Probl\`emes Inverses).}

\begin{abstract}
 \noindent
 We consider elliptic transmission problems with
 complex coefficients across an interface. Under proper transmission
 conditions, that extend known conditions for well-posedness, and sub-ellipticity we derive microlocal and
 local Carleman estimates near the interface. Carleman estimates are
 weighted {\em a priori} estimates of the solutions of the 
 elliptic transmission problem. The weight is of exponential form,
 $\exp(\tau \varphi)$ where $\tau$ can be taken as large as
 desired. Such estimates have numerous applications in unique
 continuation, inverse problems, and control theory. The proof relies
 on microlocal factorizations of the symbols of the conjugated
 operators in connection with the sign of the imaginary part of their
 roots. We further consider weight functions where $\varphi= \exp(\csp
 \psi)$, with  $\csp$ acting as a second large paremeter, and we derive estimates where the
 dependency upon the two parameters, $\tau$ and $\csp$, is made
 explicit. Applications to unique continuation properties are
 given. \\

 \noindent {\sc R\'esum\'e:} Nous consid\'erons des probl\`emes de
 transmission elliptiques \`a coefficients complexes. En \'etendant des conditions qui rendent ce probl\`eme
 bien pos\'e, et sous condition de  sous-ellipticit\'e nous obtenons des in\'egalit\'es
 de Carleman microlocales et locales \`a l'interface qui sont des
 in\'egalit\'es a priori \`a poids pour les solutions du
 probl\`eme. Les fonctions poids
 sont exponentielles, $\exp(\tau \varphi)$, o\`u le param\`etre
 $\tau$ peut \^etre choisi arbitrairement grand. De telles estimations
 ont de nombreuses applications comme pour les questions de
 prolongement unique, les probl\`emes inverses et le contr\^ole.
La d\'emonstration repose sur des factorisations
  microlocales du symbole des op\'erateurs conjugu\'es li\'ees aux signes
  des parties imaginaires de leurs racines. Nous consid\'erons le cas $\varphi= \exp(\csp \psi)$, o\`u
  $\csp$ peut-\^etre arbitrairement grand et nous obtenons
  des in\'egalit\'es de Carleman pour lesquelles la d\'ependence en
  les deux grands param\`etres,  $\tau$ et $\csp$, est rendue
  explicite. Des applications aux questions de prolongement unique
  sont propos\'ees.\\

 \noindent
  {\sc Keywords:} {Carleman estimate; elliptic operators; transmission 
    problem and condition; unique continuation}\\
  
  \noindent
  {\sc  AMS 2010 subject classification: 35B45; 35J30; 35J40.}
\end{abstract}

\maketitle

\tableofcontents

\section{Introduction and main result}
\label{sec: introduction}
Let $\Omega$ be an open subset of $\R^n$ with a smooth boundary and  let
$\Omega_1$ be an open  subset  of $\Omega$ such that $\Omega_1 \Subset
\Omega$ and such that  $S = \d \Omega_1$ is smooth.
We set $\Omega_2 = \Omega\setminus \Omega_1$. We
thus have $\d \Omega_2 = S \cup \d \Omega$. \medskip

Points in $\R^n$ are denoted by $x =(x_1,\dots, x_n)$ and we write $D_j =-i\p/\d x_j$ where $i=\sqrt{-1}$.
Let us consider two linear partial differential operators $P_k$,
$k=1,2$ of respective order $m_k=2\mu_k$, with $\mu_k\geq 1$,
\begin{equation}
  \label{eq: Definition P}
  P_k= \sum_{\abs{\mi}\leq m_k} a^k_\mi(x)D^\mi,\quad k=1,2,
\end{equation}
where the coefficients $a^k_\mi(x)$ are bounded measurable
complex-valued functions defined in $\ovl{\Omega}$. The higher-order
coefficients $a^k_\mi(x)$ with $\abs{\mi}= m_k$ are required to be $\Cinf$
in $\ovl{\Omega}_k$. In what follows, we assume that boths operators $P_k$, $k=1,2$ are elliptic.

In addition, we consider a system of $m_1+m_2$ linear transmission 
operators 
\begin{equation}
  \label{eq: Definition Tkj}
  T_{k}^j=\sum_{\abs{\mi}\leq  \torder^j_{k}}
  t^j_{k,\mi}(x)D^\mi, 
  \quad k=1,2,\quad j=1,\dots,m=\mu_1+\mu_2,
\end{equation}
with $0\leq \torder^j_{k}< m_k$, and where the coefficients
$t^j_{k,\mi}(x)$ are $\Cinf$ complex-valued functions defined in some
neighborhood of $S$.
Setting $\torder^j = (\torder^j_1 + \torder^j_2)/2$ we assume that 
\begin{align}
  \label{eq: assumption beta}
  m_1 - \torder_1^j = m_2 - \torder_2^j = m - \torder^j, \quad j=1, \dots, m=\mu_1+\mu_2. 
\end{align}

We consider a system of $\mu_2 = m_2/2$ linear boundary
operators of order less than $m_2$
\begin{equation}
  \label{eq: boundary operators}
  B^j= \sum_{\abs{\mi}\leq \beta_\d^j} b^j_\mi(x)D^\mi,\quad j=1,\dots,\mu_2,
\end{equation}
where the coefficients $b^j_\mi(x)$ are $\Cinf$ complex-valued
functions defined in some neighborhood of $\d \Omega$.

We can then consider the following elliptic boundary-value
transmission problem
\begin{equation*}
  \begin{cases}
    P_k u_k = f_k &  \text{in} \ \Omega_k, \ , k=1,2\\
    T^j_1u_1+T_2^ju_2=g^j, &  \text{in} \ S,\quad j=1,\dots,m.\\
    B^j u =h^j , &   \text{in} \ \d\Omega,\quad j=1,\dots,\mu_2.
  \end{cases}
\end{equation*}

The aim of the present  article is to derive  a Carleman estimate for
this transmission problem. 
Carleman estimates are weighted {\em a priori} inequalities for the solutions of
a partial differential equation (PDE), where the weight is of exponential type.  For the partial
differential operator $P$ {\em away from the boundary and from the
  interface}, say for $w \in \Cinfc(\Omega_1)$ or $\Cinfc(\Omega_2)$, it takes the form:
\begin{equation}
  \label{eq: intro Carleman}
  \Norm{e^{\tau \varphi} w}_{L^2}
  \leq C \Norm{e^{\tau \varphi} P w}_{L^2}, \qquad 
  \tau \geq \tau_0.
\end{equation}
The exponential weight involves a parameter $\tau$ that can be taken
as large as desired. The weight function $\varphi$ needs to be chosen
carefully. Additional terms in the \lhs, involving derivatives of $u$, can be
obtained depending on the order of $P$ and on the {\em joint}
properties of $P$ and $\varphi$.  For instance for a second-order
operator $P$ such an estimate can take the form
\begin{align}
  \label{eq: example Carleman}
  \tau^3 \Norm{e^{\tau \varphi} w}_{L^2}^2 
  + \tau \Norm{e^{\tau \varphi} \nabla_x w}_{L^2}^2 
  \leq C 
  \Norm{e^{\tau \varphi} P  w}_{L^2}^2, \qquad 
  \tau \geq \tau_0,\  w \in \Cinfc(\Omega_1)\ \text{or}\
  \Cinfc(\Omega_2).
\end{align}
This type of estimate was used for the first time by T.~Carleman
\cite{Carleman:39} to achieve uniqueness properties for the Cauchy
problem of an elliptic operator. Later, A.-P.~Calder\'on and
L.~H\"ormander further developed Carleman's method
\cite{Calderon:58,Hoermander:58}.  To this day, Carleman estimates
remain an essential method to prove unique continuation properties;
see for instance \cite{Zuily:83} for manifold results.  On such questions
more recent advances have been concerned with differential operators
with singular potentials, starting with the contribution of D.~Jerison
and C.~Kenig \cite{JK:85}. The reader is also referred to
\cite{Sogge:89,KT:01,KT:02}.  In more recent years, the field of
applications of Carleman estimates has gone beyond the original
domain; they are also used in the study of:
\begin{itemize}
\item Inverse problems, where Carleman estimates are used to obtain
  stability estimates for the unknown sought quantity (\eg
  coefficient, source term) with respect to norms on measurements
  performed on the solution of the PDE, see \eg
  \cite{BK:81,Isakov:98,Kubo:00,IIY:03}; Carleman estimates are also
  fundamental in the construction of complex geometrical optic solutions
  that lead to the resolution of inverse problems such as the Calder\'on
  problem with partial data \cite{KSU:07,DKSU:09}.
  
\item Control theory for PDEs; through unique continuation properties,
 Carleman estimates are used for the exact controllability of hyperbolic equations \cite{BLR:92}.  They also
yield the null controllability of linear parabolic equations
\cite{LR:95} and the null controllability of classes of semi-linear
parabolic equations \cite{FI:96,Barbu:00,FZ:00}.
\end{itemize}
For general elliptic operators, Carleman estimates away from
boundaries and interfaces can be found in \cite[Chapter
8]{Hoermander:63}. The essential condition for the derivation of such an
estimate is a compatibility property between the elliptic operator $P$
and the weight function $\varphi$, the so-called sub-ellipticity
condition which is known to be necessary and sufficient for the
estimate to hold in the case of an elliptic operator.  At the boundary
$\d \Omega$, a Lopatinskii-type compatibility condition involving $P$,
$\varphi$, and the operators $B^k$ can be put forward yielding a
Carleman estimate in conjunction with the sub-ellipticity condition
\cite{Tataru:96,BLR:13}. The main goal of the present article is the
extension of this analysis to transmission problems.

Note that Carleman estimates of the form given here are local. Yet,
they can be patched together to form global estimates. Our goal here
is to derive such an estimate in the \nhd of a point of the interface
$S$. Derivation of Carleman estimates away for the interface can be
found in the aforementionned references. Then, the patching procedure
allows one to obtain a global estimate in the whole $\Omega$,
following for instance \cite[Lemma 8.3.1]{Hoermander:63} and
\cite{LRL:12}. We do not cover this issue here. 

Here the weight function $\varphi$ will be chosen continuous and
piecewise smooth, that is,
$\varphi_k = \varphi_{|\Omega_k} \in \Cinf(\ovl{\Omega_k})$.  The
estimate we shall obtain will exhibit additional terms that account
for the transmission conditions given by the operators $T^j_k$,
$k=1,2$, $j=1,\dots,\mu_1+\mu_2$.  The key conditions for the
derivation of the present Carleman estimate are compatibility
properties between the elliptic operator $P$, the weight function
$\varphi$, and the transmission operators $T_k^j$. Those are the 
sub-ellipticity condition described above that expresses compatibility
between $P$ and $\varphi$, and in addition a condition that connects
them to $T_k^j$ at the interface; we shall refer to this latter
condition as to the {\em transmission condition}. This condition is an
extension of the condition presented in \cite{Schechter:60} in the
case of a conjugated operator. There, it was introduced towards to
understanding of the well-posedness of the elliptic transmission
problem. The condition we use is very close in its formulation to the
Lopatinskii type boundary condition used in the first part of this
work~\cite{BLR:13}. In \cite{Tataru:96,BLR:13} the derivation of
Carleman estimates at a boundary is based on the study of interior and
boundary differential quadratic forms, an approach that originates in
the work of \cite{Hoermander:63} for estimates away from boundaries
and in \cite{Sakamoto:70,Sakamoto:82,Miyatake:75} for the treatment of
boundaries.  This approach is here extended to interface tranmission problems.
By proper (tangential)
  microlocalizations at the interface we show the precise action of
  our transmission condition.  These
  microlocalizations are important as the transmission condition is
  function of the sign of the imaginary parts of the roots
  of\footnote{Here to simplify we consider the case $S = \{ x_n
    =0\}$. Then $\xi_n$ corresponds the (co)normal direction at the
    interface. In the main text we shall use
    change of variables to reach this configuration locally.  }
  $p_{k,\varphi}(x,\xi',\tau,\xi_n) = p_k(x,\xi + i \tau
  \varphi'(x))$, $k=1,2$, 
  viewed as a polynomial in $\xi_n$. Of course the configuration of
  the roots changes as the other parameters $(x',\xi',\tau)$ are modified. Roots can
  for instance cross the real axis. Each configuration needs to be
  addressed separately through a microlocalization procedure. For the
  Laplace operator at a  {\em boundary} this was exploited to obtain a
  Carleman estimate in \cite{LR:97} for the purpose of proving a
  stabilization result for the wave equation.
  This approach was used for the study of an interface problem in
  \cite{Bellassoued:03,LRR:10,LR-L:13} in the case of second-order
  elliptic operators. The present article
  provides a generalization of these earlier works, both with respect
  to the order of the operators and with respect to the generality of
  the transmission operators used.

   The Carleman estimate we prove here is of the form, with $u_k = u_{|\Omega_k}$,
  \begin{multline*}
    \sum_{k=1,2}
    \Norm{e^{\tau\varphi_k}u_k}^2
    +\sum_{k=1,2}|e^{\tau\varphi}\trace(u_k)|^2\\
    \leq C\Big(\sum_{k=1,2}\Norm{e^{\tau\varphi_k}P(x,D)u_k}^2+
      \sum_{j=1}^{\mu_1+ \mu_2}|e^{\tau\varphi_{|S}} (T_1^j(x,D){u_1}_{|S} + T_2^j(x,D){u_2}_{|S} )|^2\Big),
  \end{multline*}
  for $u$ supported near a point at the interface, 
  where $\trace(u_k)$ stands for the trace of $(u_k,D_\nu u_k, \dots,
  D_\nu^{m-1} u_k)$, the successive normal derivatives of $u_k$, at
  the interface $S$. In this form, the estimate is
  incorrect as norms needs to be made precise. For a correct statement
  please refer to Theorem~\ref{theorem: Carleman} below.

For Carleman estimates, one is often inclined to choose a weight
  function of the form $\varphi = \exp(\csp \psi)$, with the parameter
  $\csp>0$ chosen large. Several authors have derived Carleman
  estimates for some operators in which the dependency upon the second
  parameters $\csp$ is kept explicit. See for instance
  \cite{FI:96}. Such results can be very useful to address systems of
  PDEs, in particular for the purpose of solving inverse problems. On
  such questions see for instance \cite{Eller:00,EI:00,IK:08,BY:10}.
  
  Compatibility conditions need to be introduced between the operator
  $P$ and the weight $\psi$. Those are the so-called \spcty conditions
  introduced by L.~H\"ormander \cite{Hoermander:63,Hoermander:V4}.
  With the weight function $\varphi$ of the form $\varphi = \exp(\csp
  \psi)$, the parameter $\csp$ can be viewed as a convexification
  parameter.  As shown in Proposition~28.3.3 in \cite{Hoermander:V4}
  the \spcty of the function $\psi$ with respect to $P$
  implies the sub-ellipticity condition for $\varphi$ mentioned
  above\footnote{The terminology for the \spcty condition and the
    sub-ellipticity condition are often confused by authors. Here we
    make a clear distinction of the two notions.} for $\csp$ chosen
  \suff large. Away from any
  boundary and interface, for a second-order estimate the resulting Carleman
  estimate can take the form (compare with \eqref{eq: example Carleman}):
  \begin{align}
    \label{eq: example Carleman -2p}
    (\csp \tau)^3 \bigNorm{\varphi^{3/2} e^{\tau \varphi} u }_{L^2}^2 
    + \csp \tau \bigNorm{\varphi^{1/2} e^{\tau
        \varphi} \nabla_x u }_{L^2}^2 
    \lesssim \Norm{e^{\tau \varphi} P  u }_{L^2}^2, 
    \qquad \tau \geq \tau_0,\ \csp \geq \csp_0,\  u \in
    \Cinfc(\Omega_1) \ \text{or}\ \Cinfc(\Omega_2).
  \end{align}
  We aim to extend such estimate in the \nhd of the interface $S$. We
  then assume that the transmission condition holds for the operators
  $P$, $T_k^j$, $k=1,2$, $j=1, \dots, \mu_1+ \mu_2$, and the weight
  $\psi$. The work \cite{LeRousseau:12} provides a general framework
  for the analysis and the derivation of Carleman estimates with two
  large parameters away from boundaries. For that purpose it
  introduces a pseudo-differential calculus of the Weyl-H\"ormander
  type that resembles the semi-classical calculus and takes into
  account the two large parameters $\tau$ and $\csp$ as well as the
  weight function $\varphi = \exp(\csp \psi)$. Here, following the
  first part of this article \cite{BLR:13}, the analysis of
  \cite{LeRousseau:12} is adapted to the case of an estimate at the
  interface.  Estimates with the two large parameters $\tau$ and $\csp$
  are derived in the case of general elliptic operators.

  If we strengthen \spcty condition of $\psi$ and $P$, assuming the
  so-called simple characteristic property, sharper estimates can be
  obtained \cite{LeRousseau:12}. We also derive such estimates at the
  interface.

  \medskip
  With the different Carleman estimate that we obtain here she shall
  be able to achieve unique continuation properties at an interface
  across some hypersurface for some classes of elliptic operators and
  some products of such operators. 

\subsection{Setting} 
Now, we give the precise setting of the problem we consider.  For
$x=(x_1,\dots,x_n) \in \R^n$, we denote by $\xi=(\xi_1,\dots,\xi_n)$
the corresponding Fourier variables. Moreover, for every $\xi\in\R^n$
and $\mi\in\N^n$ we define
$\xi^\mi=\xi_1^{\mi_1}\cdots\xi_n^{\mi_n}$. We denote by
\begin{equation*}
  p_k(x,\xi)=\sum_{\abs{\mi}=m_k}a^k_\mi(x)\xi^\mi
\end{equation*}
the principal symbol of the operator $P_k$ given in \eqref{eq:
  Definition P}, $k=1,2$.
The operators $P_k$ are assumed to be elliptic, \viz,
\begin{equation*}
p_k(x,\xi)\neq 0,\quad \forall x\in \ovl{\Omega}_k,\,\forall \xi\in\R^n\backslash\set{0}.
\end{equation*}

With $m = \mu_1 + \mu_2 = (m_1 + m_2)/2$, we denote by
\begin{equation*}
  t_{k}^j(x,\xi)=\sum_{\abs{\mi}=\torder^j_k} t^j_{k,\mi}(x)\xi^\mi,\quad
  k=1,2\,\, j=1,\dots,m, 
\end{equation*}
the principal symbol of the transmission operator $T_k^j$ defined in
\eqref{eq: Definition Tkj}.  Each set $\set{T_{1}^j}_{1\leq j\leq m}$
and $\set{T_{2}^j}_{1\leq j\leq m}$ is assumed normal, that is 
$$
0\leq \torder^1_k \leq \torder^2_k 
\leq \cdots \leq \torder^{m}_k<m_k, 
$$
and for all $x \in S$ the conormal vector $\nu(x)$ is non characteristic, \ie, $t_{k}^j(x,\nu(x)) \neq 0$.

We recall that we assume 
\begin{align*}
  m_1 - \torder^j_1 = m_2 - \torder_2^j = m - \torder^j, \quad j=1, \dots, m=\mu_1+\mu_2. 
\end{align*}
We now review the definition of two inportant notions that will
be used in what follows:
\begin{itemize}
\item the sub-ellipticity condition between the operators $P_k$ and
  the weight function $\varphi$;
\item the transmission condition stating the compatibility between the transmission operators
  $T_k^j$, the operators $P_k$, and the weight function $\varphi$ at a
  point of the interface.
\end{itemize}

\subsection{Sub-ellipticity condition}
For any two functions $f(x,\xi)$ and $g(x,\xi)$ in
$\Cinf(\Omega_k\times\R^n)$ we denote their Poisson bracket in phase-space by
\begin{equation*}
  \set{f,g}=\sum_{j=0}^n\bigpara{\frac{\d f}{\d\xi_j}\frac{\d g}{\d x_j}
    -\frac{\d f}{\d\xi_j}\frac{\d g}{\d x_j}}.
\end{equation*}
It is to be connected with the commutator of two (pseudo-)differential
operators. 
In fact, if $f$ and $g$ are polynomials in $\xi$, then the principal
symbol of the commutator $[f(x,D), g(x,D)]$ is precisely $- i \{f,g
\}(x,\xi)$.   

\medskip
The sub-ellipticity condition connecting the symbol $p_k$ and a weight
function $\varphi$ is the following (See \cite[Chapter
8]{Hoermander:63} and \cite[Sections 28.2--3]{Hoermander:V4}).
\begin{definition}
  \label{def: sub-ellipticity}
  Let $k\in \{1,2\}$ and let $U$ be an open subset of $\Omega_k$ and
  set $\varphi_k = \varphi_{|\Omega_k}$. The
pair $\{P_k,\varphi_k\}$ satisfies the sub-ellipticity condition on
$\ovl{U}$ if $\varphi_k'(x):=\nabla\varphi(x)\neq 0$ at every point
in $\ovl{U}$ and if
  \begin{equation*}
    p_k(x, \xi+i\tau\varphi'(x)) = 0 \quad \imp \quad 
    \frac{1}{2 i}\set{\ovl{p}_k (x,\xi-i\tau\varphi'(x)),
      p_k(x,\xi+i\tau\varphi'(x))}>0,
  \end{equation*}
  for all $x \in \ovl{U}$ and all non-zero $\xi\in\R^n$, $\tau>0$.
\end{definition}
For an elliptic operator the sub-ellipticity condition is
{\em necessary and sufficient} for a Carleman
estimate of the form of \eqref{eq: intro Carleman} to hold away from
the boundary
\cite[Section 28.2]{Hoermander:V4}. 
For a simple exposition of the derivation of Carleman estimates for
second-order elliptic operators under the sub-ellipticity condition we
refer to \cite{LRL:12}.

Note also that the sub-ellipticity condition is invariant under
changes of coordinates. This is an important fact here as we shall work in
local coordinates in what follows.

\begin{remark}
  \label{remark: sub-ellipticity tau =0}
 Note that here, as the operator $P_k$ are elliptic, we have
  $p_k(x,\xi)\neq 0$ for each $\xi\in\R^n$, $\xi\neq 0$. The
  sub-ellipticity condition thus holds naturally at $\tau=0$.
\end{remark}
\begin{remark}
  Setting $p_{k,\varphi}(x,\xi, \tau) = p_k(x,\xi+ i \tau \varphi_k')$  and
  writing $p_{k,\varphi} = a + i b$ with $a$ and $b$ real, we have 
  $$
  \frac{1}{2 i}\set{\ovl{p_k} (x,\xi-i\tau\varphi_k'),
      p_k(x,\xi+i\tau\varphi_k')}
    = \frac{1}{2 i}\{ \ovl{p_{k,\varphi}}, p_{k,\varphi}\}(x,\xi,\tau)
    = \{ a, b\} (x,\xi,\tau).
  $$
  Below, we shall use the sub-ellipticity condition in the form
  \begin{equation*}
    p_k(x, \xi+i\tau\varphi_k') = 0 \quad \imp \quad 
    \{ a, b\} (x,\xi,\tau)>0,
  \end{equation*}
  for all $x \in \ovl{U}$ and all non-zero $\xi\in\R^n$, $\tau>0$.

  In connection with the symbol interpretation of the Poisson bracket
  given above, we see that the sub-ellipticity condition guarantees
  some positivity for the operator $i [a(x,D,\tau), b(x,D,\tau)]$ on
  the characteristic set of $p_k(x, D+i\tau\varphi_k') = a(x,D,\tau) +i
  b(x,D,\tau)$.  A proper combination of $ a(x,D,\tau)^\ast
  a(x,D,\tau) + b(x,D,\tau)^\ast b(x,D,\tau)$ and $i [a(x,D,\tau),
  b(x,D,\tau)]$ thus leads to a positive operator. This is the heart of
  the proof of Carleman estimates.
\end{remark}

\subsection{Transmission conditions}
\label{sec: intro transmission}

\medskip We consider a \nhd $X$ of a point of the interface $S$,
chosen sufficiently small, so that there exists a smooth function $\theta(x)$
such that $ d \theta (x) \neq 0\ \text{in}\ X$ and
\begin{equation}
  \label{eq: theta}
 \{x\in X; \  \theta(x) < 0 \} = \Omega_1 \cap X, \quad 
 \{x\in X; \  \theta(x) = 0\} =  S \cap X, \quad 
  \{x\in X; \  \theta(x) > 0\} = \Omega_2 \cap X. 
\end{equation}

For $x \in S$,  we denote by $N_x^*(S)$ the
conormal space above $x$ given by
$$
N_x^*(S)=\set{\nu \in T_x^*(\Omega);
  \,\nu(Z)=0,\,\forall Z\in T_x(S)}.
$$
The conormal bundle of $S$ is given by
$$
N^*(S)=\set{(x,\nu)\in T^*(\Omega);\ x\in S,\, \nu\in N_x^*(S)}.
$$
In fact, if $x \in X \cap S$ and $(x,\nu) \in N^*(S)$ then $\nu = t\
d \theta(x)$ for some $t\in \R$.

\medskip
By an interface quadruple $\omega=(x,Y,\nu,\tau)$ we shall mean
\begin{equation*}
  x\in S \cap X, 
  \quad Y\in T^*_x(S),
  \quad \nu = t d \theta (x) \in N_x^\ast(S) \ \text{with}\ t>0, 
  \quad \text{and}\ \tau\geq 0.
\end{equation*}
In particular $\nu$ ``points'' from $\Omega_1$ into $\Omega_2$. 
For an interface quadruple $\omega$  and $\lambda\in\C$, we set, for
$k=1,2$, 
\begin{align}
  \label{eq: polynomial transmission}
  \tilde{p}_{k,\varphi}(\omega,\lambda):=p_k (x,Y+\lambda \nu_k+i\tau
    d\varphi_k(x)), \quad \text{with}\ \varphi_k =
  \varphi_{|\Omega_k}, \text{and}\  \nu_k = (-1)^k \nu \in N_x^*(S).
\end{align}
Note that for $\Omega_k$ the covector $\nu_k$ points inward and we have
$\nu_1 = -\nu_2= - \nu$.

For a fixed interface quadruple $\omega_0=(x_0,Y_0,\nu_0,\tau_0)$, we
denote by $\sigma_k^j$, $k=1,2$, the roots of $\tilde{p}_{k,\varphi}(\omega_0,\lambda)
$ with multiplicity $\mu_k^j$, viewed as a polynomial of degree $m$ in
$\lambda$, with leading-order coefficient $c_{k,0}$. We can then factorize
this polynomial as follows:
\begin{equation*}
  \tilde{p}_{k,\varphi}(\omega_0,\lambda)
  = c_{k,0} \tilde{p}_{k,\varphi}^+(\omega_0,\lambda)
  \tilde{p}_{k,\varphi}^-(\omega_0,\lambda)
  \tilde{p}_{k,\varphi}^0(\omega_0,\lambda),
\end{equation*}
  with
  \begin{equation*}
    \tilde{p}_{k,\varphi}^\pm(\omega_0,\lambda) = \prod_{\pm\Im \sigma_k^j >0} (\lambda - \sigma_k^j)^{\mu_k^j},
    \quad
     \tilde{p}_{k,\varphi}^0(\omega_0,\lambda) = \prod_{\Im \sigma_k^j =0} (\lambda - \sigma_k^j)^{\mu_k^j}.
  \end{equation*}
\medskip

We define the polynomial $\kappa_{k,\varphi}(\omega_0,\lambda)$ by
\begin{equation}
  \label{eq: def kappa lopatinskii}
  \kappa_{k,\varphi}(\omega_0,\lambda)
  =\tilde{p}_{k,\varphi}^+(\omega_0,\lambda)
  \tilde{p}_{k,\varphi}^0(\omega_0,\lambda).
\end{equation}

\bigskip Similarly, for the set of transmission operators,
$\big\{T_k^j\big\}_{k=1,2;j=1,\dots,m}$, $m = \mu_1 + \mu_2$, and their
principal symbols, $t^j_k(x,\xi)$, for an interface quadruple
$\omega=(x,Y,\nu, \tau)$ we set
\begin{equation}
  \label{eq: transmission lopatinskii}
  \tilde{t}^j_{k,\varphi}(\omega,\lambda)=t^j_k(x,Y+\lambda \nu_k+i\tau
  d\varphi_k(x)), 
\end{equation}
with $\nu_k = (-1)^k\nu$ as above.
\begin{definition}
  \label{def: transmission Lopatinskii}
  We say that $\big\{P_k,T^j_k,\varphi,\ k=1,2;j=1,\dots,m\big\}$
  satisfies the transmission condition at an interface quadruple
  $\omega_0=(x_0,Y_0,\nu_0,\tau_0)$ if for all pairs of polynomials,
  $q_k(\lambda)$, $k=1,2$, there exist $U_k$, $k=1,2$, polynomials and
  $c_j\in\C$, $j=1,\dots,m$, such that:
  \begin{equation*} 
    q_1(\lambda)=\sum_{j=1}^{m}c_j\tilde{t}^j_{1,\varphi}(\omega_0,\lambda)
  +U_1(\lambda)\kappa_{1,\varphi}(\omega_0,\lambda)
 \quad \text{and} \ \
  q_2(\lambda)=
  \sum_{j=1}^{m}c_j\tilde{t}^j_{2,\varphi}(\omega_0,\lambda)
  +U_2(\lambda) \kappa_{2,\varphi}(\omega_0,\lambda),
  \end{equation*}
  where the polynomials $\kappa_{k,\varphi}(\omega_0,\lambda)$ are
  those defined by \eqref{eq: def kappa lopatinskii}.

  Additionally, for $x_0 \in S$, we say that
  $\big\{P_k,T^j_k,\varphi,\ k=1,2;j=1,\dots,m\big\}$
   satisfies the transmission condition at $x_0$ if the above property
   holds for all interace quadruples $\omega = (x_0,
   Y_0,\nu_0,\tau_0)$ with $Y \in T^*_{x_0}(S)$, $\nu =  t\
d \theta(x_0)$ with $t>0$, and $\tau\geq 0$. 
\end{definition}
It should be noted that the same coefficients $c_j$ are used in both
decompositions.

\begin{remark}
\begin{enumerate}
\item There is a strong similarity in the form between
  Definition~\ref{def: transmission Lopatinskii} and the strong
  Lopatinskii condition that is used in the derivation of a Carleman
  estimate at the boundary.  The latter condition connects  the elliptic operator, the weight
  functions and the boundary operators given in \eqref{eq: boundary
    operators} in \cite{Tataru:96,BLR:13}.
\item Note that we did not choose any  particular co-normal vectors $\nu_k$ connected with
  the function $\theta$ that locally defines $S$, apart from their orientation. In fact, for any $t>0$
  replacing $\nu_k$ by $t \nu_k$ does not affect the transmission
  condition of Definition~\ref{def: transmission Lopatinskii}. We
  could for instance use normalized conormal vectors, yet keeping the
  directions of $\nu_k$.
\end{enumerate}
\end{remark}

\subsubsection{Alternative formulation}
\label{sec: transmission alternative formulation}
Setting $n_k = d^\circ \kappa_{k,\varphi}$  we have $n_k=m_k -
m_k^-$ with $m_k^- = d^\circ \tilde{p}_{k,\varphi}^-$. 
We have $n_k \leq
m_k$. Hence, it is sufficient to consider the polynomials $q_k$,
$k=1,2$, to be of degree less than $m_k-1$ respectively. Then both
polynomials $U_k$ are of degree less than or equal to $m_k - n_k-1 =
m_k^- -1$, recalling that
$\tilde{t}^j_{k,\varphi}$ is of degree $\torder_k^j < m_k$. 

We then set
\begin{align*}
  \tilde{e}_{1}^j &= \begin{cases} 
     \tilde{t}^j_{1,\varphi} & \text{if}\  j=1, \dots, m, \\
     \lambda^{j-(m+1)} \kappa_{1,\varphi}
    & \text{if}\  j=m+1, \dots, m+ m_1^-, \\
\end{cases}\\
  \tilde{e}_{2}^j &= \begin{cases} 
     \tilde{t}^j_{2,\varphi} & \text{if}\  j=1, \dots, m, \\
    \lambda^{j-(m+1)} \kappa_{2,\varphi}
    & \text{if}\  j=m +1, \dots, m+ m_2^-, 
\end{cases}
\end{align*}
and the linear map 
\begin{align}
  \label{eq: polynomial map}
  \Phi: \C^{m} \times \C^{m_1^-} \times \C^{m_2^-} &\to \C_{m_1-1}[\lambda]
  \times  \C_{m_2-1}[\lambda], \\
  (c, \gamma_1, \gamma_2) 
  & \mapsto \Big( \sum_{j=1}^m c_j \tilde{e}_{1}^j +
  \sum_{j=1}^{m_1^-}  \gamma_{1,j} \tilde{e}_{1}^{j+m} , 
  \sum_{j=1}^m c_j \tilde{e}_{2}^j +
  \sum_{j=1}^{m_2^-}  \gamma_{2,j} \tilde{e}_{2}^{j+m}\Big).
  \notag
\end{align}
The transmission condition of Definition~\ref{def: transmission
  Lopatinskii} means precisely that the map $\Phi$ is {\em surjective}. 
In particular this implies that $m'= m + m_1^- + m_2^- \geq m_1 + m_2
= 2m$.

\subsubsection{Transmission condition and well-posedness}
Transmission conditions across an interface for elliptic problems can
be found in \cite{Schechter:60} to prove well-posedness of the
transmission elliptic problem. In \cite{Schechter:60} it corresponds to the case $\varphi=0$,
as no conjugation with a weight function is performed and, there, the
transmission condition reads as follows: let $U_k$, $k=1,2$,
polynomials and $c_j\in\C$, $j=1,\dots,m$, be such that
  \begin{equation*} 
   0=\sum_{j=1}^{m}c_j\tilde{t}^j_{1,\varphi=0}(\omega_0,\lambda)
  +U_1(\lambda)\kappa_{1,\varphi=0}(\omega_0,\lambda)
 \quad \text{and} \ \
  0=
  \sum_{j=1}^{m}c_j\tilde{t}^j_{2,\varphi=0}(\omega_0,\lambda)
  +U_2(\lambda) \kappa_{2,\varphi=0}(\omega_0,\lambda),
  \end{equation*}
  then $U_k \equiv 0$, $k=1,2$, and $c_j=0$, $j=1,\dots,m$.
  This condition precisely means that the map $\Phi$ introduced in \eqref{eq: polynomial map} is {\em injective} in the case $\varphi=0$. 
  Above we gave a surjectivity condition that suits the purpose of the present article. We now explain how the two conditions coincide in the cases studied in \cite{Schechter:60}.
  
  In fact if $\varphi=0$ none of the roots can be real as the operator is elliptic.
  If  $\#\{ \Im
  \sigma_k^j>0 \} = \#\{ \Im \sigma_k^j <0\}$, such an operator is called properly elliptic by Schechter \cite{Schechter:60} and he only studies this type of elliptic operators. 
  In fact, if the dimension $n\geq 3$ then every elliptic operator is properly elliptic (see \eg the proof of Proposition 1.1 in \cite[Chapter 2, Section 1]{LM:68}).  
  In such a case we have $m_k^- = m_k /2 = \mu_k$. Hence for
  the map $\Phi$ defined in \eqref{eq: polynomial map} we have
  \begin{equation*}
  	\dim \C^{m} \times \C^{m_1^-} \times \C^{m_2^-} 
	= \dim \C_{m_1-1}[\lambda] \times  \C_{m_2-1}[\lambda], 
  \end{equation*}
  as $m +
  m_1^- + m_2^- = 2m = m_1 + m_2$, meaning that in this case the surjectivity of
  $\Phi$ is equivalent to its injectivity. In the case of a properly elliptic operator and $\varphi=0$ the  transmission condition of \cite{Schechter:60} thus coincides with
  Definition~\ref{def: transmission Lopatinskii} in the case $\varphi=0$.

  Note that the injectivity of the map $\Phi$ may be lost for $\varphi \neq 0$
  and $\tau > 0$, as the conjugated operator $P_\varphi$ may not be elliptic, yet
  the transmission condition we give here precisely states that $\Phi$
  is surjective.

\subsubsection{Invariance by change of local coordinates}
\label{sec: invariation by change of coordinates}
We finish the presentation of the transmission condition by observing
that this definition is of geometrical nature, independent of the
choice of coordinates. This fact is important as we shall make use of
local coordinates at the interface $S$
in what follows.

In fact, for a point $x \in S$ we consider an open  \nhd $X
\subset \Omega$ of $x$, chosen \suff small so that  there exists $\theta$ as in \eqref{eq: theta}.

We consider two coordinate systems $(X^{(i)}, \psi^{(i)})$, $i=1,2$,
that is $\psi^{(i)}: X \to X^{(i)}$ is a diffeomorphism and $X^{(i)}$
is an open set in $\R^n$.  We set $x^{(i)} = \psi^{(i)}(x)$.
We moreover set 
\begin{equation*}
  \theta^{(i)} = \theta \circ (\psi^{(i)})^{-1}, \quad 
  X_k^{(i)} = \{x \in X^{(i)}; \ (-1)^k \theta^{(i)}(x) > 0\} 
  = \psi^{(i)} (X \cap \Omega_k).
\end{equation*}

We then introduce the diffeomorphism $\kappa: X^{(1)} \to X^{(2)}$
given by $\kappa = \psi^{(2)} \circ (\psi^{(1)})^{-1}$ and we have
$\kappa (x^{(1)}) = x^{(2)}$, $\theta^{(1)} = \theta^{(2)} \circ \kappa$,
yielding $\kappa (X_k^{(1)}) = X_k^{(2)}$.

We also define $\varphi^{(i)} = \varphi \circ \psi^{(i)}$, $i=1,2$,
the local versions of the weight function in the coordinate patches
and we set $\varphi^{(i)}_k = {\varphi^{(i)}}_{|X_k^{(i)}}$. 

Let $Y^{(i)}, \nu_k^{(i)}$, $k=1,2$, $i=1,2$, be the local versions of $Y$
and $\nu_k$ in the two coordinate systems. With standard differential geometry arguments we have the following relations:
\begin{align*}
  Y^{(1)} =  \transp  \kappa '(x^{(1)}) Y^{(2)}, 
  \quad \nu_k^{(1)} =  \transp  \kappa '(x^{(1)}) \nu_k^{(2)},
  \quad d \varphi^{(1)}_k (x^{(1)}) 
  = \transp  \kappa '(x^{(1)}) d \varphi^{(2)}_k (x^{(2)}),
\end{align*}

Similarly let $p^{(i)}_k$ and $t^{j(i)}_k$, $k =1,2$, $j =1, \dots,
m$, $i=1,2$, be the local versions of the principal symbols of the
differential operators $P_k$ and $T_k^j$.  We have
\begin{align*}
  p^{(1)}_k(x,\xi) = p^{(2)}_k (\kappa(x), \transp  \kappa '(x)^{-1} \xi),
  \quad 
  t_k^{j(1)}(x,\xi) =  t_k^{j(2)}(\kappa(x), \transp  \kappa '(x)^{-1} \xi).
\end{align*}

If we set $f^{(i)}_k(\lambda) = p^{(i)}_k (x^{(i)}, Y^{(i)} + \lambda
\nu^{(i)}_k+ i \tau d \varphi^{(i)}_k(x^{(i)}) )$, $i=1,2$, we find
\begin{align*}
  f_k^{(1)} (\lambda) 
  &= p^{(1)}_k (x^{(1)}, Y^{(1)} + \lambda \nu^{(1)}_k
  + i \tau d \varphi^{(1)}_k(x^{(1)})) \\
  &= p^{(2)}_k \big(\kappa(x^{(2)}), \transp  \kappa '(x^{(1)})^{-1} 
  (Y^{(1)} + \lambda \nu^{(1)}_k+ i \tau d \varphi_k^{(1)}(x^{(1)}))\big)\\
  &= p^{(2)}_k  (x^{(2)},  Y^{(2)} + \lambda \nu_k^{(2)}
  + i \tau d \varphi_k^{(2)}(x^{(2)}))\\
  &= f_k^{(2)} (\lambda),
\end{align*}
which simply means that the polynomial function
$\tilde{p}_{k,\varphi}$ defined in \eqref{eq: polynomial transmission}
does not depend on the coordinate system chosen. The same holds for
the polynomial function $\tilde{t}_{k,\varphi}^j$ defined in
\eqref{eq: transmission lopatinskii}, which allows one to conclude
that the transmission condition of Definition~\ref{def: transmission
  Lopatinskii} can be stated (and checked) in any coordinate system.

\subsection{Sobolev norms with a parameter}
\label{sec: sobolev norms}
For non-negative integer $m$ and a real number $\tau\geq 0$, we
introduce the Sobolev spaces $H^m_\tau(\Omega_k)$ and
$H^m_\tau(S)$ defined by the following norms respectively:
\begin{equation}
  \label{eq: def norms}
  \Norm{u}^2_{m,\tau}=\sum_{k=0}^m
  \tau^{2(m-k)}\Norm{u}^2_{H^k(\Omega_k)}
  \quad 
  \text{and} 
  \quad 
  \norm{u}^2_{m,\tau}=\sum_{k=0}^m \tau^{2(m-k)}\norm{u}^2_{H^k(S)},
\end{equation}
where we denote the usual Sobolev norms on $\Omega_k$ and $S$ by
$\Norm{.}_{H^s(\Omega_k)}$ and $\norm{.}_{H^s(S)}$. The $L^2$
inner-products on $\Omega_k$ and $S$ will be denoted by $\para{.,.}$
and $\para{.,.}_\d$ respectively.  Observe that for the
norm  $\Norm{.}_{m,\tau}$ on $H^m_\tau(\Omega_k)$ we do not specify
explicitly the integer $k$ that refers to which side of the interface we
consider, $\Omega_1$ or $\Omega_2$. In the main text there should
never be any ambiguity as the norm will be used for functions that are
clearly defined on one of the open sets.
\\

For $m \in \N$ and $s\in \R$ we introduce the following interface space
$$
H^{m,s}_\tau(S)=\prod_{j=0}^m H^{m-j+s}_\tau(S),
$$
equipped with the norm
\begin{equation}
  \label{eq: norm Sobolev m s}
  \norm{\mathbf{u}}^2_{m,s,\tau}=\sum_{j=0}^m
  \norm{u_j}_{m-j+s,\tau}^2,
  \quad \mathbf{u}=(u_0,\dots,u_m).
\end{equation}

If $u\in \Cinf(\ovl{\Omega_k})$ we set
$\trace^m(u)=(\trace_0(u),\dots,\trace_m(u))$ where
$\trace_j(u)=(\frac{1}{i}\d_\nu)^ju_{|S}$, with $\nu$ conormal to $S$,
is the sectional trace of $u$ of order $j$ and, in accordance with~\eqref{eq: norm Sobolev m s},  we define
$$
\norm{\trace^m(u)}^2_{m,s,\tau}
=\sum_{j=0}^m \norm{\trace_j(u)}_{m-j+s,\tau}^2.
$$
In what follows we shall write $\trace(u)$ in place of $\trace^m(u)$
for concision. We shall also write norms of the form
$|e^{\tau\varphi}\trace(u)|^2_{m,s,\tau}$
actually meaning
\begin{equation*}
\norm{e^{\tau\varphi} \trace^m(u)}^2_{m,s,\tau}
=\sum_{j=0}^m  \norm{e^{\tau\varphi} \trace_j(u)}_{m-j+s,\tau}^2.
\end{equation*}

\subsection{Statement of the main result}
We can now state the local Carleman estimate that we prove in the
\nhd of a point of the interface, with the sub-ellipticity and transmission conditions.
\begin{theorem}
  \label{theorem: Carleman}
  Let $x_0 \in  S$ and let $\varphi\in \Con^0(\Omega)$ be \st
  $\varphi_k = \varphi_{|\Omega_k} \in \Cinf(\Omega_k)$ for $k=1,2$
  and \st 
  the pairs $\{P_k,\varphi_k\}$ have the sub-ellipticity property of
  Definition~\ref{def: sub-ellipticity} in a \nhd of $x_0$ in
  $\ovl{\Omega_k}$. Moreover, assume that $\big\{P_k,\varphi, T_k^j,\
    k=1,2, \ j=1,\dots,m\big\}$ satisfies the transmission condition at
  $x_0$.  Then there exist a \nhd $W$ of $x_0$ in $\R^n$ and two
  constants $C$ and $\tau_\ast>0$ \st
  \begin{multline}
    \label{eq: Carleman main result}
    \sum_{k=1,2} \big( \tau^{-1}\Norm{e^{\tau\varphi_k}u_k}^{2}_{m_k,\tau}
    +\norm{e^{\tau\varphi_{|S}}\trace(u_k)}_{m_k-1,1/2,\tau}^2 \big)\\
    \leq C \Big( \sum_{k=1,2}\Norm{e^{\tau\varphi_k}
        P_k(x,D)u_k}_{L^2(\Omega_k)}^2
      + \sum_{j=1}^m \big|e^{\tau\varphi_{|S}} 
      \big(T_1^j(x,D)u_1 +T_2^j(x,D)u_2\big)_{|S} \big|^2_{m-1/2-\torder^j,\tau}\Big),
  \end{multline}
  for all $u_k = {w_k}_{|\Omega_k}$ with $w_k\in \Cinfc(W)$ and $\tau\geq \tau_\ast$.
\end{theorem}
First, this results will be established microlocally: at an interface 
point $x_0$ we shall assume that the transmission condition holds for
some interface quadruple $(x_0, Y_0, \nu_0, \tau_0)$ and we
shall prove that a Carleman estimate of the form above holds in a
conic \nhd of this interface quadruple in phase-space; localization
in phase-space will be done by means of cut-off functions and
associated pseudo-differential operators. We refer the reader to
Section~\ref{sec: microlocal Carleman}. Second, we will deduce
Theorem~\ref{theorem: Carleman} from such microlocal estimates.

\medskip
Estimate \eqref{eq: Carleman main result} concerns function located
near the interface and vanishing near the boundary $\d_\Omega$. Hence
this estimate involves the transmission operators $T_k^j$ and not the
boundary operators $B^j$. 

Estimates of the form of \eqref{eq: Carleman main result} are
local. Yet, such estimates and their counterpart estimates at the
boundary proven in \cite{BLR:13} can be patched together to form
global estimates.  We do not cover such details here and we refer to
\cite{LRR:11} where this is done in the case of a transmission
problem.

In Section~\ref{sec: two parameters} we shall prove Carleman estimates
with a weight function of the form $\varphi(x) = \exp(\csp \psi(x))$ as
is usually done in practice with the parameter $\csp$ chosen as large
as desired. We shall provide the precise dependency of the Carleman estimate
with respect to this second large parameter.

Examples of elliptic transmission problem and weight function for
which the above result applies, and other for which it does not, will
be given in Section~\ref{sec: examples} below. In fact, as the
condition for the Carleman estimate to hold are geometrical, that is,
coordinate invariant, we shall postpone the exposition of example
after we introduce local variables that ease the writting of the
transmission condition.

\subsection{Local reduction of the problem near the interface}
\label{sec: local setting}

All the different aspects of the problem we consider --operators,
sub-ellipticity condition, transmission condition-- are coordinate
invariant as we saw above. We shall thus work locally, in the \nhd of
a point of the interface and choose coordinates that allow use to ease
the subsequent analysis and derivation of the Carleman estimate.  

\subsubsection{Choice of local coordinates}
\label{sec: choice of coordinates}
Let $x_0 \in S$. There exists a \nhd $V$ of $x_0$ and a local
system of coordinates $x=(x_1,\dots, x_n)$ where $V \cap \Omega_1 \subset \{
x_n>0\}$, $V \cap \Omega_2 \subset \{
x_n<0\}$ and $x'=(x_1,\dots, x_{n-1})$ parametrizes the interface$V
\cap S \subset \{ x_n =0\}$.  We assume that $V \cap \d \Omega
=\emptyset$, that is, we focus our analysis on the interface and
remain away
from the boudnary.

We denote by $\R^n_\pm$ the half space $\{\pm x_n > 0\}$ and $V_\pm = V \cap \R^n_\pm$.
For our purpose here, without any loss of generality, we may assume that
$V_\pm$ is bounded.

\medskip

In such local coordinates, in $V_\pm$, the differential
operator\footnote{By abuse of notation, in the new local coordinates, we
keep the notation $P_k$ and $T_k^j$, $j=1, \dots, m$, $k=1,2$, for the operators
introduced in the beginning of Section~\ref{sec: introduction}.}  $P_k$ of
order $m$ with complex coefficients takes the form
\begin{equation*}
  P_k = P_k(x,D)=\sum_{j=1}^{m_k} P_{k,j}(x,D') D_n^j,\qquad D_n=\frac{1}{i}\d_n,\quad k=1,2,
\end{equation*}
where $P_{k,j}(x,D')$, $j=1, \dots, m_k$, $k=1,2$, are tangential
differential operators with complex coefficients of order $m_k-j$.  
Similarly the transmission operators take the form
\begin{equation*}
  T^j_k = T_k^j(x,D)=\sum_{i=0}^{\torder_k^j} T^j_{k,i}(x,D') D_n^i,\quad 1\leq j\leq m,\,\,k=1,2
\end{equation*}
where $T^j_{k,i}(x,D')$, $i=0, \dots, \torder_k^j$, are tangential differential
operators of order $\torder_k^j-i$. The local transmission problem we
study thus takes the form
\begin{align}
  \label{eq: local transmission problem}
  \begin{cases}
    P_1 u_1 = f_1 &  \text{in} \  V_1 = \{ x_n <0\},\\
    P_2 u_2 = f_2 &  \text{in} \  V_2 = \{ x_n >0\},\\
    T^j_1u_1+T_2^ju_2=g^j, &  \text{in} \ S,\quad j=1,\dots,m.
  \end{cases}
\end{align}
We have $P_{k,m} = P_{k,m}(x) \neq 0$. Upon dividing the functions $f_k$ by $P_{k,m} (x)$ 
we may assume that $P_{k,m}=1$.

\bigskip
Calling $(\xi', \xi_n )$ the Fourier variables corresponding to $(x',
x_n)$ we have, for the principal symbol of $P_k$,
\begin{equation*}
p_k(x,\xi)=\sum_{j=0}^{m_k}p_{k,j}(x,\xi')\xi_n^j,
\end{equation*}
which is a polynomial homogeneous  of degree $m_k$ in the $n$ variables $(\xi',\xi_n)$.
\smallskip

We introduce
$$p_{k,\varphi}(x,\xi,\tau):=p_k (x,\xi+i\tau\varphi'_k(x)),$$ which
is the principal symbol of the operator $e^{\tau \varphi_k} P_k
e^{-\tau \varphi_k}$ viewed in the class of (pseudo-)differential operators with a large 
parameter presented in Section~\ref{sec: pseudo}.

 Setting
$\y'=(x,\xi',\tau)$ and $\y=(\y',\xi_n)$, for simplicity we shall
write $p_{k,\varphi}(\y)$ in place of $p_{k,\varphi}(x,\xi,\tau)$ and often
$p_{k,\varphi}(\y',\xi_n)$ to emphasize that the symbol is polynomial in
$\xi_n$. 
Similarly  we introduce $t_{k,\varphi}^j (x,\xi,\tau):=t_k^j
(x,\xi+i\tau\varphi'_k(x)) = t_{k,\varphi}^j (\y) = t_{k,\varphi}^j (\y',\xi_n)$.

In the present local coordinate,  we have $\theta = x_n$ and thus $\nu_k = (-1)^k d x_n$ (see
Section~\ref{sec: intro transmission}).
For a fixed point $\y'=(x,\xi',\tau)$ with $x \in S$,  setting $Y=
\sum_{j=1}^{n-1} \xi_{j} d x_j$   and $\omega = (x, Y, \tau)$
this gives 
\begin{align}
  \label{eq: connexion pkvarphi tildepk}
  \tilde{p}_{k,\varphi}(\omega,(-1)^k \xi_n)
  &= p_k (x,Y+(-1)^k \xi_n \nu_k+i\tau d\varphi_k(x))\\
  &=  p_k (x, \xi +i\tau d\varphi_k(x))\notag\\
  &= p_{k,\varphi}(x,\xi,\tau),\notag
\end{align}
with $\xi =(\xi',\xi_n)$ and $\tilde{p}_{k,\varphi}$ as defined in \eqref{eq: polynomial transmission}.
Similarly, we have 
\begin{align}
  \label{eq: connexion tkvarphi tildetk}
  \tilde{t}^j_{k,\varphi} (\omega,(-1)^k \xi_n)
  = t^j_{k,\varphi}(x, \xi),
\end{align}
with $\tilde{t}^j_{k,\varphi}$ as defined in \eqref{eq: transmission
  lopatinskii}.

\subsubsection{A system formulation}
\label{sec: system formulation}

To avoid the $(-1)^{k}$ terms in the two previous equations and to simplify
a large part of the analysis that will follow we shall write the local
transmission problem as a system of equation in $V_+$.

Without any loss of generality we can choose the open \nhd $V$ to be of the
form $V' \times (-\eps, \eps)$ with $V'$ an open set of $S$ and $\eps
>0$. 

In $V_- = V'\times (\eps,0)$ we apply the change of variables
$(x',x_n) \mapsto \sigma(x',x_n) = (x',- x_n)$. We denote by $\Pl$ and
$\Tl^j$ the operators obtained from $P_1$ and $T_1^j$ through this
change of variable. For the principal symbols we have
\begin{align*}
  \pl(x,\xi) = p_1 \big(\sigma(x); \sigma(\xi)\big), \ \ 
  \tl(x,\xi) = t_1 \big(\sigma(x); \sigma(\xi)\big), 
  \quad \text{for} \ x_n >0,
\end{align*}
using that $\transp \sigma' (x)^{-1} \xi = (\xi',-\xi_n)$. 
We also define $\vl  = \varphi_1 \circ \sigma $ for $x_n >0$. 
In $V_+$ we do not apply any change of variable and for the benefit of
readibility we set $\Pr = P_2$, $\pr = p_2$, and $\vr = \varphi_2$. 
The subscripts $\l$ and $r$ are chosen to keep in mind that part of
the system we shall write comes from the left-hand side of the
interface and the second part from the right-hand side. 

The transmission problems now reads as the following system
\begin{align}
  \label{eq: local transmission system}
  \begin{cases}
    \Pl \ul = \fl, \ \ \Pr \ur = \fr  &  \text{in} \  V_+ = \{ x_n >0\},\\
    {\Tl^j \ul}  + {\Tr^j \ur} = g^j, &  \text{in} \ S = \{ x_n =0\},\quad j=1,\dots,m,
  \end{cases}
\end{align}
where $\ul = u_1 \circ \sigma$, $\fl = f_1 \circ \sigma$, $\ur =
u_2$, $\fr = f_2$.

We set  $\Plrv = e^{\tau \vlr} \Plr e^{-\tau \vlr}$ and $\Tlrv^j =
e^{\tau \vlr} \Tlr^j e^{-\tau \vlr}$. They have for respective principal symbols (in
the calculus with a large parameter of Section~\ref{sec: pseudo})
\begin{align*}
  \plrv(x,\xi,\tau)   = \plr (x, \xi + i \tau d \vlr (x)), 
  \quad 
  \tlrv^j (x,\xi,\tau)   = \tlr^j (x, \xi + i \tau d\vlr (x)).
\end{align*}
We have 
\begin{align*}
  \plv(x,\xi,\tau) 
  = p_1 (\sigma(x), \sigma(\xi + i  \tau d \vl (x)))
  = p_1 \big(\sigma(x), \sigma(\xi)  + i  \tau d \varphi_1 (\sigma(x)))\big)
\end{align*}

Now, as $\transp \sigma' (x)^{-1} \nu_1 = \nu_2= \nu$, for a fixed point
$\y'=(x,\xi',\tau)$ with $x \in S$, \ie, $x_n=0$ giving $\sigma(x) =
x$, if we  set $Y =
\sum_{j=1}^{n-1} \xi_{j} d x_j$ and $\omega = (x, Y, \tau)$
we have, for $\tilde{p}_{1,\varphi}$ as defined in \eqref{eq: polynomial transmission}, 
\begin{align}
  \label{eq: tilde p1}
  \tilde{p}_{1,\varphi}(\omega, \xi_n) 
  &= p_1 (x,Y -  \xi_n \nu +i\tau d\varphi_1(x))\\
  &= \pl (x,Y + \xi_n \nu + i \tau d \vl (x))\notag\\
  & = \plv(x, \xi,\tau), \notag
\end{align}
with $\xi = (\xi',\xi_n)$.
Similarly we have 
\begin{align}
  \label{eq: tilde t1}
  \tilde{t}_{1,\varphi}^j(\omega, \xi_n) 
  & = \tlv^j(x, \xi,\tau).
\end{align}
Naturally by \eqref{eq: connexion pkvarphi tildepk}--\eqref{eq: connexion tkvarphi tildetk} we have 
\begin{align*}
  \tilde{p}_{2,\varphi}(\omega, \xi_n) 
  = \prv(x, \xi,\tau), \qquad
  \tilde{t}_{2,\varphi}^j(\omega, \xi_n) 
  = \trv^j(x, \xi,\tau).
\end{align*}

\subsubsection{Symbol factorizations}
\label{sec: symbol factorization}
For a fixed point $\y_0'=(x_0,\xi'_0,\tau_0) \in \sphbundle(V)$  (see the
  definition below in Section~\ref{sec: notation}) with
$ x_0\in S$
we denote the roots of $\plrv(\y_0',\xi_n)$, viewed as a polynomial in
$\xi_n$, by $\alpha_{\lr,1}, \dots,
\alpha_{\lr, n_\lr}$, with respective multiplicities $\mu_{\lr,1}, \dots,
\mu_{\lr, n_\lr}$ satisfying $\mu_{\lr,1} + \cdots + \mu_{\lr,
  n_\lr}=m_\lr$, with $m_\l = m_1$ and $m_r = m_2$. By
\cite[Lemma~A.2]{BLR:13}, there exists a conic open \nhd $\U$ of
$\y_0'$ \st
  \begin{equation}
    \label{eq: symbol factorization}
    \plrv(\y',\xi_n)=\plrv^+(\y',\xi_n)\, \plrv^-(\y',\xi_n)\, \plrv^0(\y',\xi_n),
    \quad \y' \in \U, \ \xi_n \in \R,
\end{equation}
with $\plrv^\pm$ and $\plrv^0$, polynomials in $\xi_n$ of
constant degrees in $\U$, smooth and homogeneous; in $\U$ the
imaginary parts of the roots of $\plrv^+(\y',\xi_n)$ (\resp
$\plrv^-(\y',\xi_n)$) are all positive (\resp negative) and we
have
  \begin{align*}
    \plrv^\pm(\y'_0,\xi_n) = \prod_{\pm \Im \alpha_{\lr,j} >0} (\xi_n - \alpha_{\lr,j})^{\mu_{\lr,j}},
    \quad
    \plrv^0(\y'_0,\xi_n) = \prod_{\Im \alpha_{\lr,j}  =0} (\xi_n - \alpha_{\lr,j} )^{\mu_{\lr,j}}.
  \end{align*}
The polynomials $\plrv$ are thus decomposed into three factors in the
\nhd $\U$ of $\y'_0$. For $\plrv^\pm$ the sign of the imaginary part of
their roots remain constant equal to $\pm$ respectively; for
$\plrv^0$ this sign may change and the roots are precisely real at
$\y' = \y'_0$.  

We then define the polynomial $\klrv(\y',\xi_n)$ by
\begin{equation}\label{poly r}
  \klrv(\y',\xi_n)=\plrv^+(\y',\xi_n)\,  \plrv^0 (\y',\xi_n).
\end{equation}

As above, for the principal symbols of the conjugated transmission
operators $\Tlrv^j$, $j=1, \dots, m$,  we write $\tlrv^j(x,\xi,\tau) = \tlrv^j(\y',\xi_n)$ where
$\y'=(x,\xi',\tau)$ to emphasize that the symbol is polynomial in
$\xi_n$.

\begin{remark}
  \label{rem: factorization stability}
 Observe that the factorizations in \eqref{eq: symbol factorization} depends quite significantly on the
  point $\y'_0$. They may actually be different even for point $\y'$ in
  the \nhd $\U$ introduced above.
  We should rather write something like
  \begin{equation*}
   \plrv(\y',\xi_n)=p_{\lr, \varphi, \y'_0}^+(\y',\xi_n)\, 
    p_{\lr, \varphi,\y'_0}^-(\y',\xi_n)\, 
    p_{\lr,\varphi,\y'_0}^0(\y',\xi_n),
    \quad \y' \in \U, \ \xi_n \in \R,
  \end{equation*}
  in place of \eqref{eq: symbol factorization} and 
  set 
  \begin{equation*}
    \kappa_{\lr,\varphi,\y'_0} (\y',\xi_n) = p_{\lr, \varphi,
      \y'_0}^+(\y',\xi_n)\, p_{\lr, \varphi,\y'_0}^0(\y',\xi_n).
 \end{equation*}
 For $\y'_1 \in \U$ we
  may very well have 
  \begin{equation*}
    p_{\lr, \varphi, \y'_0}^+(\y',\xi_n) 
    \neq p_{\lr, \varphi, \y'_1}^+(\y',\xi_n), 
    \ \ \text{or} \ \
    p_{\lr, \varphi,\y'_0}^-(\y',\xi_n)  
    \neq p_{\lr, \varphi, \y'_1}^-(\y',\xi_n),
    \end{equation*}
    \begin{equation*} 
      \text{or} \ \
    p_{\lr, \varphi,\y'_0}^0(\y',\xi_n)
    \neq p_{\lr, \varphi, \y'_1}^0(\y',\xi_n).
  \end{equation*}
  Yet, we shall see below that the notation in \eqref{eq: symbol
    factorization}  is sufficiently clear for our purpose.

  Still, if we denote by $M_\lr^\pm(\y')$ the number of roots (counted with their
  multiplicities) with postive (\resp negative) imaginary parts of
  $\plrv(\y',\xi_n)$ for $\y' \in
  \U$ we may have $M_\lr^\pm(\y'_0) \neq
  M_\lr^\pm(\y')$ for some $\y' \in \U$. Note that in such case we have
  $M_\lr^\pm(\y'_0) \leq  M_\lr^\pm(\y')$ from the construction of the \nhd $\U$ given in
  \cite[Lemma~A.2]{BLR:13}. Arguing as in the proof of
   \cite[Lemma~A.2]{BLR:13}, using the continuity of the roots
  \wrt $\y'$  we can in
  fact prove that for $\y_1' \in \U$ there exists a conic \nhd $\U'
  \subset \U$ of $\y'_1$ such that 
  \begin{align}
    \label{eq: instability kappa}
    \kappa_{\lr, \varphi, \y'_0}(\y',\xi_n) = h_\lr (\y', \xi_n) 
    \kappa_{\lr, \varphi, \y'_1}(\y',\xi_n), \qquad \y' \in \U', 
  \end{align}
  where $h_\lr (\y', \xi_n)$ is  polynomial in $\xi_n$ with coefficients
  that are smooth \wrt $\y' \in \U'$. 
  \end{remark}
\subsubsection{The transmission conditions in the local coordinates}
\label{sec: transmission condition in local coordinates}

In the present coordinate system  $(x',x_n)$ in $V_+$, a conormal vector
$\nu$ pointing from $\Omega_1$ to $\Omega_2$ is given by $(0,\dots,
0,\nu_n)$ with $\nu_n >0$. For the statement of the transmission condition we may
choose $\nu = (0,\dots,0,1)$. A boundary quadruple
$\omega = (x, Y,\nu, \tau)$, with $Y = (\xi',0)$  can thus be identified with
$\y' = (x,\xi',\tau)$. 

The transmission condition of Definition~\ref{def: transmission
  Lopatinskii} being invariant under change of variables as seen in Section~\ref{sec: invariation by change of
  coordinates}, 
because of \eqref{eq: tilde p1} and \eqref{eq: tilde t1}, 
we may use the polynomials $\plrv$ and $\tlrv^j$
to state locally this condition for
$\big\{P_\lr,T^j_\lr,\varphi_\lr,\ j=1,\dots,m\big\}$ at $\y'_0=(x_0,
\xi'_0,\tau_0)$.  It reads:\\
\indent {\em
For all pairs of polynomials,
  $q_\lr(\xi_n)$, there exist $U_\lr$,  polynomials, and
  $c_j\in\C$, $j=1,\dots,m$, such that
  \begin{align}
    \label{eq: local reformulation transmission 1}
    q_\l(\xi_n)=\sum_{j=1}^{m}c_j\tlv^j(\y',\xi_n)
  +U_\l(\xi_n)\klv(\y',\xi_n),
 \intertext{and}
 \label{eq: local reformulation transmission 2}
  q_r(\xi_n)=\sum_{j=1}^{m}c_j\trv^j(\y',\xi_n)
  +U_r(\xi_n)\krv(\y',\xi_n), 
  \end{align}
for $\y' = \y'_0$. }

\bigskip
We set $m_\lr^- = d^\circ \big(\plrv^-(\y',.)\big)$, that is
{\em independent} of $\y' \in \U$, with the open conic \nhd $\U$ as
introduced above, and we let
$\klrv(\y',\xi_n)$ be the polynomial function given in \eqref{poly r}. It
takes the form 
\begin{equation*}
  \klrv(\y',\xi_n)=\sum_{i=0}^{m_{\lr}-m_{\lr}^-}\kappa_{\lr,i}(\y')\xi_n^i,
  \quad \y'\in \U,\ \xi_n \in \R,
\end{equation*}
where $\kappa_{\lr,i}$ is homogeneous of degree $m_\lr-m^-_\lr-i$ \wrt $(\xi',\tau)$.
Similarly we write
\begin{equation*}
  \tlrv^j(\y',\xi_n)=\sum_{i=0}^{\torder_\lr^j} t_{\lr,i}^j(\y')\xi_n^i,
  \quad \y'\in \U,\ \xi_n \in \R,\quad j=1,2,
\end{equation*}
with $\torder_\l^j = \torder_1^j$ and $\torder_r^j=\torder_2^j$, and where
$t_{\lr,i}^j$ is homogeneous of degree $\torder_\lr^j - i$ \wrt
$(\xi',\tau)$. We recall that we have (see \eqref{eq: assumption beta})
\begin{align}
  \label{eq: assumption beta lr}
  m_\l - \torder^j_\l = m_r - \torder_r^j = m - \torder^j, \quad j=1, \dots, m. 
\end{align}

Now, we introduce two famillies of polynomial functions,
$e_\lr^j(\y', .)$, of degree less than or equal to $m_\lr-1$, for
$j=1, \dots,m_\lr'$, with $m_\lr'= m+m_\lr^-$. We recall that $m =(m_\l+m_r)/2$.
We set 
\begin{align*}
  e_\lr^j (\y', \xi_n) = 
  \begin{cases}
    \tlrv^j (\y', \xi_n) & \text{for}\ 1\leq j \leq m, \\
    \xi_n^{j-(m+1)} \klrv (\y', \xi_n) & \text{for}\ m+1\leq j \leq
    m_\lr', \quad \text{if} \ m_\lr' > m.
  \end{cases}
\end{align*}
Observe that $m_\lr' > m$ if $m_\lr^- >0$.
If we write 
\begin{equation*}
  e_\lr^j(\y',\xi_n)=\sum_{i=0}^{m_\lr-1} e_{\lr,i}^j(\y')\xi_n^i,
\end{equation*}
we thus obtain
\begin{align*}
  &\bullet \text{for}\ 1\leq j \leq m, \ \ e_{\lr,i}^j = \begin{cases}
    t_{\lr,i}^j & \text{if}\  0\leq i \leq \torder_\lr^j, \\
    0          & \text{otherwise}.
   \end{cases}
   \\
   & \bullet \text{for}\ m+1\leq j \leq m_\lr', \ \ e_{\lr,i}^j = \begin{cases}
     \kappa_{\lr,i-j+m+1}  &\text{if}\  j-(m+1) \leq  i \leq j-(m+1)+m_\lr-m_\lr^-, \\
    0          & \text{otherwise}, 
   \end{cases}\\
   & \quad \text{if} \ m_\lr' > m.
\end{align*}

For $1\leq j \leq m$, $e_{\lr,i}^j$ is homogeneous of degree
$\torder_\lr^j-i$  \wrt $(\xi',\tau)$. 
If $m_\lr' > m$, setting 
\begin{equation}
  \label{eq: extended transmission order}
  \torder_\lr^j=j+m_\lr-(m+m_\lr^-+1), \quad j = m+1, \dots,  m_\lr', 
\end{equation}
 we see that, for $m+1\leq j \leq m_\lr'$, 
the tangential symbol $e^j_{\lr,i}$ is  homogeneous of degree $\torder_\lr^j-i$  \wrt
$(\xi',\tau)$ and the symbol $e^j_{\lr}$ is  homogeneous of degree
$\torder_\lr^j$ \wrt $(\xi,\tau)$. 

Introducing the matrices
\begin{align*}
  \mathcal{T}_\lr^1(\y')= \begin{pmatrix} e^j_{\lr,i-1} 
\end{pmatrix}_{{1\leq i\leq m_\lr} \atop {1\leq j\leq m} }, 
\qquad 
\mathcal{T}_\lr^2(\y')= \begin{pmatrix} e^{j+m}_{\lr,i-1} 
\end{pmatrix}_{{1\leq i\leq m_\lr} \atop {1\leq j\leq m_\lr^-} }, 
\end{align*}
by Section~\ref{sec: transmission alternative formulation}, we see
that the transmission condition of Definition~\ref{def: transmission
  Lopatinskii} for $\big\{P_\lr,T^j_\lr,\varphi_\lr,\
j=1,\dots,m\big\}$ at $\y'_0 = (x_0, \xi'_0, \tau_0)\in \sphbundle
(V)$, with $x_0 \in S$, also stated in \eqref{eq: local reformulation
  transmission 1}--\eqref{eq: local reformulation transmission 2} with
the local setting introduced here, reads as follows
  \begin{align}
    \label{eq: rank formulation transmission condition}
    \rank \mathcal{T}  (\y')  = m_\l + m_r = 2m, 
    \qquad \text{with}\ \
    \mathcal{T}  (\y') =\begin{pmatrix}
      \mathcal{T}_\l^1(\y') & \mathcal{T}_\l^2(\y')  & 0 \\ 
      \mathcal{T}_r^1(\y') & 0 & \mathcal{T}_r^2(\y') 
    \end{pmatrix}, 
  \end{align}
  for $\y' = \y'_0$. Note that $2m$ is the number of rows in $\mathcal{T}  (\y')$. 

\medskip
We find again that $m' = m + m_\l^- + m_r^-\geq 2m$. Moreover, there
exists a $2m\times 2m$ sub-matrix $\mathcal{T}_0(\y'_0)$ \st $\det
\mathcal{T}_0(\y'_0)\neq 0$. 
As the coefficients of $\mathcal{T}_\lr^1(\y')$
and $\mathcal{T}_\lr^2(\y')$
are continuous and homogeneous of degree $\torder^j_\lr-i+1$ and $\torder^{j+m}_\lr-i+1$ \wrt $(\xi',\tau)$ respectively, 
where $j$ is  the column number and $i$ is the line number, we then have
 $\det \mathcal{T}_0(\y')\neq 0$ homogenous \wrt
$(\xi',\tau)$.  It follows that $\det \mathcal{T}_0(\y')\neq 0$ for $\y'$ in a
small conic neighborhood $\V \subset \U$ of $\y'_0$. Note that the
homogeneity of the coefficients is important for $\V$ to be chosen
conic since $\det \mathcal{T}_0(\y')$ is itself homogeneous \wrt
$(\xi',\tau)$.  The rank of $\mathcal{T}(\y')$ thus remains equal to
$2m$ in $\V$, meaning that condition \eqref{eq: rank formulation
  transmission condition} is valid for $\y'$ in the whole $\V$.

\medskip
We have thus reached the following result. 
\begin{proposition}
  \label{prop: stability transmission}
  Let the transmission condition $\big\{P_\lr,T^j_\lr,\varphi_\lr,\
  j=1,\dots,m\big\}$  hold at
  $\y'_0=(x_0,\xi'_0,\tau_0)$. Then we have $m_\l^- + m_r^- \geq
  m$. Moreover there exists a conic \nhd $\V$ of $\y'_0$ \st 
\eqref{eq: rank formulation
  transmission condition} is valid for $\y' \in \V$.
\end{proposition}

\begin{remark}
 Observe that the result of Proposition~\ref{prop: stability
    transmission} implies the condition \eqref{eq: local reformulation
    transmission 1}--\eqref{eq: local reformulation transmission 2}
  holds for $\y' \in \V \subset \U$, yet with $\kappa_{\lr, \varphi}$
  defined by the symbol factorizations at $\y'_0$, that is
  $\kappa_{\lr, \varphi} = \kappa_{\lr, \varphi,\y'_0}$ using the
  notation of Remark~\ref{rem: factorization stability}.
Now using \eqref{eq: instability kappa} we see that this implies that 
the transmission condition also holds at $\y'_1$.  We thus see that
the transmission condition remains valid in a conic \nhd of
$\y'_0$. However, we shall not use this aspect here. The importance
aspect we shall use is the local persistence of condition \eqref{eq: local reformulation
    transmission 1}--\eqref{eq: local reformulation transmission 2}
  for $\y'$ in a conic \nhd of $\y'_0$ as
stated in Proposition~\ref{prop: stability transmission} (of course
the two are very related). This explains why we do not use the ``more
precise'' notation of Remark~\ref{rem: factorization stability} throughout the article.
\end{remark}

With the locally presistant decomposition of the conjugated opeators 
$\plrv=\plrv^-  \klrv$ in the \nhd $\U$ of $\y'_0$,  the following states roughly the proof
strategy we shall adopt: 
\begin{enumerate}
  \item The factors $\plrv^-$ associated with roots with negative
    imaginary parts yields two perfect elliptic estimates at the
    interface and no transmission condition is needed.
  \item Each  factors $\klrv$ yields an estimate at the interface
    that involves trace  terms. These terms will be  estimated via the
    actions of the transmission  operators $\Tlrv^j$ by means of to the
   transmission condition.
\end{enumerate}
Note that  we have 
$2 m - m_\l^- - m_r^- \leq m$
by Proposition~\ref{prop: stability transmission}.  The number of
trace relations available at the interface in \eqref{eq: local
  transmission system} is precisely $m$.  This indicates that we shall
have at hand a sufficiently large number of transmission relation to
control the terms originating from the estimate with the factors
$\klrv$ that are of degree $m_\lr-m_\lr^-$ whose sum is $ m_\l-m_\l^- +
m_r-m_r^- = 2m - m_\l^- - m_r^- $.

\subsection{Examples}
\label{sec: examples}
We provide several examples that illustrate the generality of the
elliptic transmission problems that can be addressed through the
results of the present article. 

\subsubsection{Second-order elliptic operators}
\label{Sec: examples Second-order elliptic operators}
We start by providing fairly classical cases of transmission problems, where the
operators are of second order on both sides of the interface.
Consider the operators 
\begin{align*}
  P_{k} = \sum_{1\leq i,j \leq n} D_i a_{ij}^{(k)}(x) D_{j},  \quad k=1,2, 
\end{align*}
with real coefficients, 
assumed to be ellitpic, that is, $a_{ij}^{(k)} = a_{ji}^{(k)}$ and $(a_{ij}^{(k)}(x) ) \geq C>0$
uniformly for $x \in \Omega_k$.  
Let $x_0 \in S$. As above, local coordinates  are chosen so that
$\Omega_1 = \{ x_n < 0\}$ and $\Omega_2 = \{ x_n >0 \}$  in a \nhd of $x_0$.

\medskip
\paragraph{\bfseries A. A natural transmission problem for second-order elliptic
  operators}
A natural and classical  transmission problem can be stated with the following interface 
operators 
\begin{align*}
  T_{k}^1 = (-1)^k, \quad T_k^2 = (-1)^k \sum_{1\leq j \leq n} 
  a_{nj}^{(k)}(x) D_j, \qquad k=1,2.
\end{align*}
The transmission problem thus reads
\begin{align*}
  P_1 u_1 = f_1\ \text{in} \ \Omega_1, \quad P_2 u_2 = f_2  \ \text{in} \ \Omega_2, 
\end{align*}
and
\begin{align*}
 {u_1}_{|S} = {u_2}_{|S}  \quad \text{and}  \quad \sum_{1\leq j \leq n} 
  a_{nj}^{(1)}(x) {D_j u_1}_{|S} =  \sum_{1\leq j \leq n} 
  a_{nj}^{(2)}(x) {D_j u_2}_{|S}, 
\end{align*}
that is,  we impose,  at the interface, the  continuity 
of the solution, as well as that of the normal flux (in the sense of
the anisotropic diffusion matrices $(a_{ij}^{(1)})$ and
$(a_{ij}^{(2)})$) . This is physically
very natural, and, mathematically, it implies that piecewise smooth functions satisfying 
these two conditions are in the domain of the self-adjoint operator
$P = \nabla \cdot A \nabla$, where $A= (a_{ij}(x))$ with $a_{ij}(x) =
a^{(k)}_{ij}(x)$ if $x \in \Omega_k$, $k=1,2$.

This
configuration was treated in \cite{LR-L:13} following some works
on some particular ``conformal'' cases
\cite{Bellassoued:03,LRR:10}. Since this example has been extensively
studied, it is natural to question if the results of the present
article generalize those provided in these references.

We choose a weight function $\varphi(x)$ that is smooth on both sides
of $S$ and continuous across $S$. Then for an interface quadruple
$\omega = (x_0, Y, \nu, \tau)$, with here $\nu = (0, \dots, 0,1)$,
$Y=(\xi',0) = (\xi_1, \dots, \xi_{n-1}, 0)$, we have, for $k=1,2$,  
\begin{align*}
  \tilde{t}_{k,\varphi}^1(\omega, \lambda) = (-1)^k,
\end{align*}
and 
\begin{align*}
  \tilde{t}_{k,\varphi}^2(\omega, \lambda) 
  =(-1)^k a_{nn}^{(k)}(x_0) \big((-1)^k\lambda + i  \tau\d_{x_n} \varphi_k(x_0)\big) 
  +  (-1)^k \sum_{1\leq j \leq
    n-1} a_{nj}^{(k)}(x_0) \big(\xi_j + i \tau \d_{x_j}\varphi_k(x_0)\big).
\end{align*}
As $a_{nn}^{(k)}\neq 0$ because of the ellipticity of $P_k$ we see
that $\tilde{t}_{k,\varphi}^2(\omega, \lambda) $ are exactly of degree
1 in $\lambda$.

Observe that we can write the principal symbol of $P_k$, $k=1,2$, as
  \begin{align*}
    p_k(x,\xi) = a_{nn}^{(k)} \Big( 
    \big( \xi_n + \sum_{j=1}^{n-1} a_{nj}^{(k)}/ a_{nn}^{(k)}  \xi_j  \big)^2 + b_k(x,\xi')
    \Big), 
    \end{align*}
    with the quadratic form
    \begin{align*}
    b_k(x,\xi') = \big( a_{nn}^{(k)} \big)^{-2}\sum_{i,j=1}^{n-1}  \big(a_{ij}^{(k)}a_{nn}^{(k)} -
    a_{ni}^{(k)} a_{nj}^{(k)} \big) \xi_i \xi_j,
  \end{align*}
  which is positive definite in $\xi'$. We thus find
  \begin{align*}
    \tilde{p}_{k,\varphi}(\omega, \lambda) &= a_{nn}^{(k)}(x_0)  \Big( 
    \big( (-1)^k \lambda + i \tau \d_{x_n} \varphi_k (x_0)  + \sum_{j=1}^{n-1} a_{nj}^{(k)}(x_0) / a_{nn}^{(k)}(x_0) 
    (\xi_j + i \tau \d_{x_j}\varphi_k (x_0) )  \big)^2\\
    &\qquad \qquad \qquad \qquad+ b_k(x_0 ,\xi' + i\tau d_{x'}\varphi_k (x_0) )
    \Big)
\end{align*}
In fact, we may write $b_k(x,\xi' + i\tau d_{x'}\varphi_k) =(A_k - i B_k)^2$, 
  where $A_k$ and $B_k$ are functions of $x$, $\xi'$  and $\tau$,
  homogeneous of degree one in $(\xi',\tau)$, with
  $A_k \geq 0$, \and we
  thus find
  \begin{align*}
    \tilde{p}_{k,\varphi}(\omega,
    \lambda) 
    &= a_{nn}^{(k)} \prod_{k'=1,2} \Big( (-1)^k \lambda 
      + i  \tau \d_{x_n} \varphi_k  (x_0)
      + \sum_{j=1}^{n-1} a_{nj}^{(k)}/ a_{nn}^{(k)}(x_0)
    (\xi_j + i \tau \d_{x_j}\varphi_k (x_0)) \\
    & \qquad \qquad  \qquad \qquad+  (-1)^{k'} ( B_k + i
      A_k) \Big).
  \end{align*}

Writing $\tilde{p}_{k,\varphi}(\omega,
    \lambda) = a_{nn}^{(k)} (\lambda - \sigma_k^1(\omega)) (\lambda -
    \sigma_k^2(\omega))$, 
several cases can occur depending on the signs of the
imaginary part of the roots $\sigma_k^j$, $k=1,2$, $j=1,2$. 

\medskip
\paragraph{{\bfseries Case 1. Either $\bld{\tilde{p}_{1,\varphi}(\omega, \lambda)}$ or
    $\bld{\tilde{p}_{2,\varphi}(\omega, \lambda)}$ has two roots in $\bld{\{ \Im z
    <0\}}$.}} 
Assume that for instance $\bld{\tilde{p}_{1,\varphi}(\omega,
  \lambda)}$ has its two roots in $\{ \Im z<0\}$. Then
$\kappa_{1,\varphi} =1$ while $\kappa_{2,\varphi}(\omega, \lambda) $ is of degree 0, 1,
or 2. Let then $q_1(\lambda),   q_2(\lambda)$ be two polynomial functions.

As $\tilde{t}_{2,\varphi}^1(\omega, \lambda)= 1$, and $\tilde{t}_{2,\varphi}^2(\omega, \lambda)$ is
exaclty of degree 1, we
may write 
\begin{align*}
  q_2(\lambda) = c_1 \tilde{t}_{2,\varphi}^1(\omega, \lambda) + c_2 \tilde{t}_{2,\varphi}^2(\omega, \lambda) +
  U_2(\lambda) \kappa_{2,\varphi}(\omega,\lambda).
\end{align*}
by means of a Euclidean division by $\kappa_{2,\varphi}(\omega,
\lambda)$, writing the remainder polynomial as a linear combination of $\tilde{t}_{2,\varphi}^1(\omega, \lambda)$
and $\tilde{t}_{2,\varphi}^2(\omega, \lambda)$.
We then have 
\begin{align*}
  q_1(\lambda) = c_1 \tilde{t}_{1,\varphi}^1(\omega, \lambda) + c_2 \tilde{t}_{1,\varphi}^2(\omega, \lambda) +
  U_1(\lambda) \kappa_{1,\varphi}(\omega,\lambda), 
\end{align*}
by simply choosing $U_1(\lambda)  = q_1(\lambda) - c_1
\tilde{t}_{1,\varphi}^1(\omega, \lambda) - c_2
\tilde{t}_{1,\varphi}^2(\omega, \lambda)$.
The transmission condition of Definition~\ref{def: transmission Lopatinskii} thus holds in this case.

\medskip
\paragraph{{\bfseries Case 2. Each symbol 
    $\bld{\tilde{p}_{k,\varphi}(\omega, \lambda)}$, $\bld{k=1,2}$ has only one roots in $\bld{\{ \Im z
    <0\}}$.}} 

In this case, both polynomials $\kappa_{1,\varphi}(\omega,\lambda)$ and
$\kappa_{2,\varphi}(\omega,\lambda)$ are of degree 1 in $\lambda$.
As we chose $A_k \geq 0$, the root of $\tilde{p}_{k,\varphi}(\omega,
    \lambda)$ with non-negative imaginary
part is given by 
$\sigma_k^+ = - \zeta_k / a_{nn}^{(k)}(x_0)  
  + B_k+i   A_k$ with
\begin{align*}
 \zeta_k &= (-1)^k \big(\sum_{j=1}^{n-1} a_{nj}^{(k)}(x_0)(\xi_j + i
            \tau \d_{x_j} \varphi_j(x_0)) + i \tau
            a_{nn}^{(k)}(x_0) \d_{x_n} \varphi_k(x_0) \big) .
\end{align*}

To understand whether the transmission condition holds or not, it is
handy to use the matrix $\mathcal{T}$ introduced in \eqref{eq: rank
  formulation transmission condition}. It is a $4 \times 4$ matrix
here.  Using \eqref{eq: tilde p1} and \eqref{eq: tilde t1}, we find
\begin{align*}
    \mathcal{T}  =\begin{pmatrix}
      -1 & \zeta_1 &  - \sigma_1^+& 0 \\
      0 & a_{nn}^{(1)} & 1 & 0\\
      1 & \zeta_2 & 0 & -\sigma_2^+\\
      0 & a_{nn}^{(2)} & 0 & 1
    \end{pmatrix}.
  \end{align*}
Observe that the rank of $\mathcal{T}$ is
4 if and only if $\zeta_1 + \zeta_2 \neq - a_{nn}^{(1)} \sigma_1^+ -
a_{nn}^{(2)} \sigma_2^+$. This yields the condition 
\begin{align*}
  a_{nn}^{(1)} ( B_1 + i A_1) + a_{nn}^{(2)} ( B_2 + i A_2) \neq 0. 
\end{align*}
As we have $A_1 \geq 0$ and $A_2 \geq 0$ this condition holds if $A_1>0$
or $A_2>0$. In fact if $A_k =0$ this mean that $\Im \sigma_k^1 = \Im
\sigma_k^2$ which is exluded here. Hence we have $A_1>0$
and $A_2>0$ and the transmission condition holds in this case. 

\medskip
\paragraph{{\bfseries Case 3. One symbol $\bld{\tilde{p}_{k,\varphi}(\omega, \lambda)}$ has two roots in $\bld{\{ \Im z
    \geq 0\}}$ and the second one has at most one root in   $\bld{\{ \Im z
    < 0\}}$.}}
Assume, for instance, that $\tilde{p}_{1,\varphi}(\omega, \lambda)$
has two roots in $\{ \Im z \geq 0\}$. Then, $\kappa_{1,\varphi}$ is of
degree 2. With the assumption on $\tilde{p}_{2,\varphi}(\omega,
\lambda)$, then, $\kappa_{1,\varphi}$ is at least of degree 1.  In
this case, we find that the matrix $\mathcal{T} (\y') $ has 4 lines
and, at most, 3 columns. It cannot be of rank 4. Hence, the
transmission condition cannot hold in this case.

  \bigskip
  From the three exhaustive cases studied above we thus conclude that the derivation of a Carleman estimate in a \nhd
  of the point $x_0$ can be achieved, according to Theorem~\ref{theorem: Carleman},  if one chooses the weight function
  $\varphi$ so that Case~3 does not occur.

  \bigskip 
  Using the computation made above we write 
  \begin{align*}
    \tilde{p}_{k,\varphi}(\omega,
    \lambda) 
    &= a_{nn}^{(k)} \prod_{k'=1,2} \Big( (-1)^k \lambda     
      + \sum_{j=1}^{n-1} a_{nj}^{(k)}/ a_{nn}^{(k)}(x_0)\xi_j 
      +  i \tau \gamma_k + (-1)^{k'} ( B_k + i A_k) \Big).
  \end{align*}
  where 
  \begin{align}
    \label{eq: example def gamma}
    \gamma_k =  \d_{x_n} \varphi_k  (x_0) + \sum_{j=1}^{n-1} a_{nj}^{(k)}/ a_{nn}^{(k)}(x_0) \d_{x_j}\varphi_k (x_0)
  \end{align}
  The values of $\gamma_k$ are fixed by the
  choice of the weight function $\varphi$. 
  The imaginary parts of the roots are given by 
  \begin{align}
    \label{eq: imaginary part roots}
    \Im \alpha_{k}^{k'} = - (-1)^k 
    \big(\tau   \gamma_k + (-1)^{k'}A_k \big).
  \end{align}

  If both   $b_k(x,\xi' + i\tau d_{x'}\varphi_k)$, $k=1,2$,  are  nonpositive
  real numbers, that is $A_k=0$,  in particular if $\xi'=0$ and $\tau>0$,
  then, the imaginary parts of the roots of
  $\tilde{p}_{k,\varphi}(\omega, \lambda)$ coincide:  $\Im
  \alpha_{k}^{k'}= - (-1)^k \tau \gamma_k$, $k'=1,2$.
  In particular, this requires 
  \begin{align}
    \label{eq: condition slope}
    \gamma_2> 0 \ \ \text{if} \ \  \gamma_1 \geq 0,\qquad \text{and}
    \qquad \gamma_1< 0 \ \ \text{if} \ \  \gamma_2 \leq 0,
  \end{align}
  as, otherwise, we face
  the occurence of Case~3. The case $\gamma_1 \geq 0$ and $\gamma_2
  \leq 0$ is thus excluded.
 
  Let us assume that $\gamma_1 \geq 0$ and $\gamma_2 >0$. 
  On the one hand, as we have assumed $A_k\geq 0$, the root $\alpha_1^2$ remains in $\{ \Im
  z \geq 0\}$ and the root $\alpha_2^2$ remains in $\{ \Im
  z <0\}$. On the other hand, we have
  \begin{align*}
    \Im \alpha_1^1 = \tau \gamma_1 - A_1, \qquad \Im \alpha_2^1 = - \tau \gamma_2 + A_2.
  \end{align*}
  Hence, Case~3 does not occur if and only if $\Im \alpha_2^1 \geq 0 \ \imp \
  \Im \alpha_1^1 < 0 $,
  that is, 
  \begin{align*}
    0< \tau \gamma_2  \leq  A_2\ \imp \ 0 \leq \tau \gamma_1 < A_1.
  \end{align*}
 A sufficient condition is then
  \begin{align}
    \label{eq: slope condition}
    \frac{\gamma_1}{\gamma_2} 
   <  \frac{A_1}{A_2} \quad \text{if} \ A_2 \neq 0.
 \end{align}
Observe that in the case where $\varphi_k$ are only  functions of $x_n$,
then $A_1$ and $A_2$ are independent of $\tau$ and \eqref{eq: slope
  condition} becomes a {\em necessary and sufficient} condition for the
transmission condition to hold.

\medskip
 The case $\gamma_1 < 0$ and $\gamma_2 \leq 0$ leads similarly to the
sufficient condition
 \begin{align*}
    \frac{\gamma_2}{\gamma_1} 
   <  \frac{A_2}{A_1} \quad \text{if} \ A_1 \neq 0.
 \end{align*}

Finally, we consider the case $\gamma_1< 0$ and $\gamma_2>0$. 
Because of \eqref{eq: imaginary part roots} we then find that Case~3
cannot occur with this choice of $\gamma_1$ and $\gamma_2$. 
This particular choice, is however not very interesting, as it somehow corresponds to an observation\footnote{This
   interpretation makes sense  in the case $\varphi =
  \varphi(x_n)$. Then $\gamma_k = \d_{x_n} \varphi_k$. In Carleman
  estimates, ``observation'' region are associated with regions where
  the weight function is the largest.} of the transmission
  problem both from $\Omega_2$ and $\Omega_1$. 
  The choice $\gamma_1\geq 0$ and $\gamma_2>0$ correspond to an
  observation of the transmission problem from $\Omega_2$ only. This
  is relevant for pratical applications, for instance, in unique
  continuation problems, as one may want to find uniqueness across an
  interface, having information on the solution on one side of the
  interface only. Similarly, the choice $\gamma_1< 0$ and $\gamma_2\leq 0$
  corresponds to an observation  from $\Omega_1$ only. 
 
\medskip
\paragraph{\bfseries Remark and open question.}
Note that with \eqref{eq: slope condition} we recover the condition
stated in \cite{LR-L:13}. There, in the case
$\varphi_k = \varphi_k(x_n)$ it is proven to be sharp for a Carleman
estimate to hold. This raises the following question: is the
transmission condition presented here necessary and sufficient for the
Carleman estimate to hold? Sufficiency is the subject of the present
article. Necessity is clear in particular cases as shown in \cite{LR-L:13} but it
is not clear in general. Second-order transmission problems, in the case where $\varphi_k$
depend on $x_n$ and also $x'$ would be a natural field of investigation, but
the question extends to higher order transmission problems.

\medskip
\paragraph{\bfseries  B. Two ``non communicating'' Dirichlet problems}
If we consider 
\begin{align*}
T^1_1 = T_2^1 = 1, \qquad T^2_1 = -T_2^2=1, 
\end{align*}
observe, then, that the transmission problem
\begin{align*}
  P_1 u_1 = f_1 \ \text{in} \ \Omega_1, 
  \quad  P_2 u_2 = f_2\ \text{in} \ \Omega_2, 
\end{align*}
and
\begin{align*}
  {T^1_1 u_1}_{|S} +   {T^1_2 u_2}_{|S}= g_1, \quad 
  {T^2_1 u_1}_{|S} +   {T^2_2 u_2}_{|S}= g_2,
\end{align*}
corresponds two  having the following two problems
\begin{align*}
  P_1 u_1 = f_1,\quad  {u_1}_{|S} = (g_1 + g_2)/2,  
\end{align*}
and
\begin{align*}
  P_2 u_2 = f_2,\quad  {u_2}_{|S} = (g_1 - g_2)/2,  
\end{align*}
that is,  two well-posed elliptic equations with {\em independent}  Dirichlet
boundary conditions.

Using the notation and symbol computations of the previous example,
let us 
assume that one of the symbols $\tilde{p}_{k,\varphi}(\omega,
\lambda)$,  say for $k=1$, has two roots in $\{ \Im z \geq 0\}$.  Then,
the matrix  $\mathcal{T}$ introduced in \eqref{eq: rank
  formulation transmission condition} takes the form
\begin{center}
\begin{tikzpicture}[
style1/.style={
  matrix of math nodes,
  every node/.append style={text width=#1,align=center,minimum height=3.2ex},
  nodes in empty cells,
  left delimiter=(,
  right delimiter=),
  }
]
\matrix[style1=0.2cm] (1mat)
{
  & & &  & \\
  &  & &  & \\
  & & &   &\\
  & & &  &\\
};
\node[font=\normalsize] 
  at (1mat-1-1.west) {1};
  \node[font=\normalsize] 
  at (1mat-2-1.west) {0};
 \node[font=\normalsize] 
  at (1mat-3-1.west) {1};
   \node[font=\normalsize] 
  at (1mat-4-1.west) {0};
  
  \node[font=\normalsize] 
  at (1mat-1-2) {1};
  \node[font=\normalsize] 
  at (1mat-2-2) {0};
 \node[font=\normalsize] 
  at (1mat-3-2) {-1};
   \node[font=\normalsize] 
  at (1mat-4-2) {0};

   \node[font=\huge] 
  at (1mat-2-4.north east) {0};
  
   \node[font=\huge] 
  at (1mat-3-4.south east) {$\otimes$};
  
  \node[font=\normalsize] 
  at (1mat-2-1.south west) {\hspace*{-2cm}$\mathcal T=$};

  \draw[dashed]
  (1mat-2-3.south) -- (1mat-2-5.south east);
   \draw[dashed]
  (1mat-2-3.south)-- (1mat-1-3.north);
   \draw[dashed]
  (1mat-2-3.south) -- (1mat-4-3.south);
\end{tikzpicture}
\end{center}
and thus cannot be of rank 4. Hence, to derive a Carleman
estimate, according to Theorem~\ref{theorem: Carleman}, we need to avoid an
occurance of two roots associated with the same symbol $\tilde{p}_{k,\varphi}(\omega,
\lambda)$ in $\{ \Im z \geq 0\}$. With $\gamma_k$, as
defined in \eqref{eq: example def gamma}, and the imaginary parts of
the roots given in \eqref{eq: imaginary part roots}, we note that we
then need to impose $\gamma_1 <0$ and $\gamma_2 >0$. As explained in
the previous example, this corresponds to observing the transmission
problem from both sides of the interface $S$. This is totally sensible
here, as the example we consider represents two totally decoupled
ellitpic problems: observing from one side of the interface cannot
yield any information about the system on the other side. As
Theorem~\ref{theorem: Carleman} implies unique continuation properties
(see Section~\ref{sec: unique continuation}) we see that it is very natural that the Carleman
estimate  cannot be derived, unless observations are made on both sides.

\subsubsection{Higher-order elliptic operators}

Here, we consider an example that involves both a second- and a
fourth-order elliptic operator. In $\R^2$, we consider the operators
$P_k(x,D)$, $k=1,2$, such that in the local coordinates as above, the
principal symbols are given by
\begin{equation*}
p_1(x,\xi_1,\xi_2)=\xi_2^2+b_1(x,\xi_1),\qquad 
p_2(x,\xi_1,\xi_2)=\xi_2^4+b^2_2(x,\xi_1),
\end{equation*}
where $b_k(x,.)$, $k=1,2$, are two positive definite quadratic
forms. We assume that the principal symbols of the transmission operators are given by
\begin{equation*}
t_1^1(x,\xi_1,\xi_2)=-1,\quad t_2^1(x,\xi_1,\xi_2)=1,
\end{equation*}
\begin{equation*}
t_1^2(x,\xi_1,\xi_2)=-\xi_2,\quad t_2^2(x,\xi_1,\xi_2)=\xi_2^3,
\end{equation*}
\begin{equation*}
t_1^3(x,\xi_1,\xi_2)=0,\quad t_2^3(x,\xi_1,\xi_2)=\xi_2^2.
\end{equation*}
We choose a weight function $\varphi(x)=\varphi(x_2)$ that is smooth on both sides of $S$ and continous across $S$. Then for an interface quadruple $\omega=(x_0,Y,\nu,\tau)$, with $\nu=(0,1)$, $Y=(\xi_1,0)$, we have for $\mu_k=-i\tau\varphi'_k$, 
\begin{equation*}
\tilde{t}_{1,\varphi}^1(\omega,\lambda)=-1,
\quad \tilde{t}_{2,\varphi}^1(\omega,\lambda)=1,
\end{equation*}
\begin{equation*}
\tilde{t}_{1,\varphi}^2(\omega,\lambda)=(\lambda+\mu_1),
\quad\tilde{t}_{2,\varphi}^2(\omega,\lambda)=(\lambda-\mu_2)^3,
\end{equation*}
\begin{equation*}
\tilde{t}_{1,\varphi}^3(\omega,\lambda)=0,
\quad \tilde{t}_{2,\varphi}^3(\omega,\lambda)=(\lambda-\mu_2)^2.
\end{equation*}
For smooth function $\varphi_k$, $k=1,2$,  such that
$\varphi_k$ is only a function of $x_2$, we assume that $\d_{x_2}\varphi_1(x_0)>0$ and $\d_{x_2}\varphi_2(x_0)>0$. We obtain
\begin{equation*}
\tilde{p}_{1,\varphi}(\omega,\lambda)=(\lambda-\alpha_1)(\lambda-\alpha_2),
\end{equation*}
where
\begin{equation*}
\alpha_1=i\tau \d_{x_2} \varphi_{1}(x_0)+i\sqrt{b_1(x_0,\xi_1)},\quad \alpha_2=i\tau \d_{x_2}\varphi_{1}(x_0)-i\sqrt{b_1(x_0,\xi_1)}.
\end{equation*}
We thus have $\Im \alpha_1>0$. The sign of $\Im \alpha_2$ may however vary.
We also have 
\begin{equation*}
\tilde{p}_{2,\varphi}(\omega,\lambda)=\prod_{j=1}^4(\lambda-\beta_j)
\end{equation*}
where
\begin{equation*}
\beta_j=-i\tau \d_{x_2}\varphi_2(x_0)-e^{i\pi(2j-1)/4}\sqrt{b_2(x_0,\xi_1)},\quad j=1,2,3,4.
\end{equation*}
As  $\d_{x_2}\varphi_2(x_0)>0$, this forbid all the roots to be in the upper
complex half plane. Then $\Im \beta_1<0$ and $\textrm{Im}\beta_2<0$. 
Yet, the signs of $\Im \beta_3$ and $\Im \beta_4$ are equal and may vary.
We have 
\begin{align*}
  \Im \beta_3 = \Im \beta_4= - \tau \d_{x_2} \varphi_2(x_0) + \sqrt{b_2(x_0,\xi_1)/2}.
\end{align*}

According to Section~\ref{sec: transmission condition in local
  coordinates}, using \eqref{eq: tilde p1} and \eqref{eq: tilde t1},
we have,
\begin{equation*}
\mathcal{T}^1_\l=\begin{pmatrix}
-1 & \mu_1 & 0 \\ 
0 & 1 & 0
\end{pmatrix},
\quad 
\mathcal{T}^1_r=\begin{pmatrix}
1 & -\mu_2^3 & \mu_2^2 \\ 
0 & 3\mu_2^2 & -2\mu_2\\
0 & -3\mu_2 & 1\\
0 & 1 & 0
\end{pmatrix}.
\end{equation*}

\medskip
\paragraph{{\bfseries Case 1. $\bld{\Im \alpha_2 <0}$ and $\bld{ \Im \beta_3
    = \Im \beta_4 <0}$.}}
In this case, we have $\tilde{\kappa}_{1,\varphi}(\omega,\lambda)=(\lambda-\alpha_1)$
and 
$\tilde{\kappa}_{2,\varphi}(\omega,\lambda) =1$.
We then have 
\begin{equation*}
\mathcal{T}_\ell^2=\begin{pmatrix}
-\alpha_1 \\ 
1
\end{pmatrix}
\end{equation*}
and $\mathcal{T}^2_r= \id_4$.
Recalling the form of $\mathcal{T}$ in \eqref{eq: rank formulation
  transmission condition}, we have
\begin{center}
\begin{tikzpicture}[
style1/.style={
  matrix of math nodes,
  every node/.append style={text width=#1,align=center,minimum height=3.2ex},
  nodes in empty cells,
  left delimiter=(,
  right delimiter=),
  }
]
\matrix[style1=0.2cm] (1mat)
{
  & & &  & & &\\
  &  & &  & & &\\
  & & &   & & &\\
  & & &  & & &\\
};
\node[font=\Large] 
  at (1mat-1-1.south east) {$\mathcal T_\ell^1$};
  \node[font=\Large] 
  at (1mat-3-1.south east) {$\mathcal T_r^1$};
   \node[font=\huge] 
  at (1mat-2-6.north east) {0};  
   \node[font=\Large] 
  at (1mat-3-6.south east) {$\id_4$}; 
  \node[font=\normalsize] 
  at (1mat-2-1.south west) {\hspace*{-2cm}$\mathcal T=$};
  \node[font=\normalsize] 
  at (1mat-1-4) {$-\alpha_1$};
  \node[font=\normalsize] 
  at (1mat-2-4) {$1$};
   \node[font=\Large] 
  at (1mat-3-4.south) {0};
  \draw[dashed]
  (1mat-2-1.south west) -- (1mat-2-7.south east);
   \draw[dashed]
  (1mat-2-3.south)-- (1mat-1-3.north);
   \draw[dashed]
  (1mat-2-3.south) -- (1mat-4-3.south);
   \draw[dashed]
  (1mat-2-5.south)-- (1mat-1-5.north);
  \draw[dashed]
  (1mat-2-5.south) -- (1mat-4-5.south); 
\end{tikzpicture}
\end{center}
whose rank is 6 as $\textrm{rank} \mathcal{T}_\ell^1=2$. Hence, the transmission condition holds in
this case, by its formulation given in \eqref{eq: rank formulation
  transmission condition}.

\paragraph{{\bfseries Case 2. $\bld{\Im \alpha_2 \geq 0}$ and $\bld{ \Im \beta_3
    = \Im \beta_4 <0}$.}}
In such case, we have 
$\tilde{\kappa}_{1,\varphi}(\omega,\lambda)=(\lambda-\alpha_1)(\lambda-\alpha_2)$
and $\tilde{\kappa}_{2,\varphi}(\omega,\lambda) =1$. As $m_\ell^-=0$
then no matrix $\mathcal{T}_\ell^2$ enters in the composition of
$\mathcal{T}$. Still, this matrix has the same  form as in Case 1
with the fourth column removed and, in this case, the rank is $6$
implying that the transmission condition holds in
this case.

\paragraph{{\bfseries Case 3. $\bld{\Im \alpha_2 <0}$ and $\bld{ \Im \beta_3
    = \Im \beta_4 \geq 0}$.}}

In such case, we have 
$\tilde{\kappa}_{1,\varphi}(\omega,\lambda)=(\lambda-\alpha_1)$
and $\tilde{\kappa}_{2,\varphi}(\omega,\lambda) =(\lambda-\beta_3) (\lambda-\beta_4)$.
We thus have 
\begin{equation*}
\mathcal{T}_\ell^2=\begin{pmatrix}
-\alpha_1 \\ 
1
\end{pmatrix},
\quad 
\mathcal{T}^2_r=\begin{pmatrix}
\beta_3\beta_4 & 0\\ 
-(\beta_3+\beta_4) & \beta_3\beta_4\\
1 & -(\beta_3+\beta_4) \\
0 & 1
\end{pmatrix}.
\end{equation*}
Thus, the matrix $\mathcal T$ reads
\begin{equation*}
\mathcal{T}=\begin{pmatrix}
-1 & \mu_1 & 0 & -\alpha_1 & 0 & 0\\ 
0 & 1 & 0 & 1& 0& 0\\
1 & -\mu_2^3 & \mu_2^2 & 0 &\beta_3\beta_4 & 0\\ 
0 & 3\mu_2^2 & -2\mu_2 & 0 &-(\beta_3+\beta_4) & \beta_3\beta_4\\
0 & -3\mu_2 & 1 & 0 &1 & -(\beta_3+\beta_4)\\
0 & 1 & 0& 0 & 0 & 1
\end{pmatrix}.
\end{equation*}
Computing its determinant we find $\det (\mathcal{T}) =
b_2(x_0,\xi_1)^2$ that does not vanish as here $\xi_1\neq 0$. In fact,
$\xi_1=0$ yields $\Im \alpha_2 = \tau \d_{x_2} \varphi_1(x_0)>0$ in contradiction
with the assumption $\Im \alpha_2<0$ made here.

\medskip
\paragraph{{\bfseries Case 4. $\bld{\Im \alpha_2 \geq 0}$ and $ \bld{\Im \beta_3
    = \Im \beta_4 \geq 0}$.}}
In this case we have
$\tilde{\kappa}_{1,\varphi}(\omega,\lambda)=(\lambda-\alpha_1) (\lambda-\alpha_2)$
and $\tilde{\kappa}_{2,\varphi}(\omega,\lambda) =(\lambda-\alpha_3) (\lambda-\alpha_4)$.
In such case the matrix $\mathcal{T}$ is a $6 \times 5$ matrix. Its
rank cannot be $6$. The transmission condition cannot hold in this
case.

\bigskip The four exhaustive cases studied above reveal that the weight
function $\varphi$ needs to be chosen so that  Case 4 does not
occur. Hence, the following condition needs to be fulfilled:
\begin{align*}
  \Im \beta_3 = \Im \beta_4 \geq 0 \quad \Rightarrow \quad \Im
  \alpha_2 <0.
\end{align*}
Recalling the forms of the roots derived above this reads
\begin{align*}
  0< \tau \varphi_1' (x_0)\leq \sqrt{b_2(x_0,\xi_1)/2} \quad \Rightarrow \quad
  0 < \tau \varphi_2' (x_0)< \sqrt{b_1(x_0,\xi_1)}.
\end{align*}
A necessary and sufficient condition is then 
\begin{align*}
  \frac{\varphi_1'(x_0)}{\varphi_2'(x_0)} < \sqrt{\frac{2 b_1(x_0,\xi_1) }{b_2(x_0,\xi_1)}}.
\end{align*}
Since $b_1(x,\xi_1) / b_2(x,\xi_1)$ is bounded from below, for any
$\xi_1$, locally in $x$, we see that this yields a precise condition
on the weight function $\varphi$. The condition prescribes a minimal
relative jump of the normal derivative of the weight function across
the interface (going from $\{x_2<0\}$ to $\{x_2>0\}$). Note that one can also provide a sufficient condition in
the case $\varphi$ depends also on the $x_1$ variable, as in
Example~\ref{Sec: examples Second-order elliptic operators}-A.

\subsection{Notation}
\label{sec: notation}

If $V \subset \Rpb$ we denote the semi-classical unit half cosphere
bundle over
$V$ (in the cotangential direction $\xi'$) by 
$$
\sphbundle(V) = \{ (x,\xi',\tau); \ x \in V,
\ \xi' \in \R^{n-1},\ \tau \in \R_+, \ |\xi'|^2 + \tau^2 =1
\}.
$$
The canonical inner product in $\C^m$ is denoted by
$\para{\zb,\zb'}_{\C^m} = \sum_{j=0}^{m-1} z_j \ovl{z'}_j$, for $\zb=(z_0,\dots,z_{m-1}), \zb'=(z'_0,\dots,z'_{m-1})\in \C^m$.
The associated norm will be denoted $\abs{\zb}_{\C^m}^2=\sum_{j=0}^{m-1}|z_j|^2$.

\medskip
We shall use some spaces of smooth functions in the closed half space.
We set
\begin{equation*}
 \S(\Rpb) = \{ u_{|\Rpb}; \ u \in \S(\R^n)\}.
\end{equation*}

For two $u, v \in \S(\Rpb)$ we set
\begin{equation*}
  \para{u,v}_+ = \para{u,v}_{L^2(\Rp)}
  \qquad
  \para{u\br,v\br}_\d =  \para{u\br,v\br}_{L^2(\R^{n-1})}.
\end{equation*}
We also set 
\begin{equation*}
  \Norm{u}_+ = \Norm{u}_{L^2(\Rp)}
  \qquad
  \norm{u\br}_\d =  \norm{u\br}_{L^2(\R^{n-1})}.
\end{equation*}

In this article, when the constant $C$ is used, it refers to a
constant that is independent of the large parameter $\tau$.  Its
value may however change from one line to another. If we want to keep
track of the value of a constant we shall use another letter.

In what follows, for concision, we shall sometimes use the notation
$\lesssim$ for $\leq C$, with a constant $C>0$.
We shall write $a \asymp b$ to denote $a \lesssim b \lesssim a$.

\section{Pseudo-differential operators with a large parameter}
\label{sec: pseudo}
\setcounter{equation}{0}

Parameter-dependent pseudo-differential operators have proven to be
important tools for the derivation of Carleman estimates. The general
aim is to obtain a pseudo-differential calculus with a large
parameter, and then to derive estimates with constants
that are independent of the parameter. Often such a
pseudo-differential calculus is referred to as a semi-classical calculus.

\subsection{Classes of symbols}
\label{sec: symbol classes}
We first introduce symbols that depend on a parameter.
\begin{definition}
  \label{def: semi-classical symbols}
  Let $a(\y) \in \Cinf(\R^n\times\R^n)$, $\y =(x,\xi,\tau)$, with $\tau$ as a parameter
  in $ [\tau_{\min},+\infty)$, $\tau_{\min}>0$, and $m \in \R$, be \st for all multi-indices
  $\mi, \smi \in \N^n$ we have
  \begin{equation}
      \label{eq: semi-classical symbols}
    \abs{\d_x^\mi \d_\xi^\smi a(\y)}
    \leq C_{\mi,\smi} \lambda^{m-\abs{\smi}},
      \quad x\in \R^n,\ \xi\in\R^n,\ \tau \in  [\tau_{\min},+\infty),
  \end{equation}
  where $\lambda = |(\xi,\tau)| = \big( |\xi|^2 + \tau^2 \big)^\hf$. Thus differentiation with respect to $\xi$ improves the decay in $\xi$ and $\tau$ simultaneously. We write $a \in \Ssc^m(\R^n\times\R^n)$ or simply $\Ssc^m$.
For $a \in \Ssc^m$ we denote by $\sigma(a)$ its principal part, that is,
its equivalence class in $\Ssc^m / \Ssc^{m-1}$.
\smallskip

We also introduce tangential symbols.  Let $a(\y') \in
\Cinf(\Rpb\times\R^{n-1})$, $\y'= (x,\xi',\tau)$, with $\tau$ as a
parameter in $ [\tau_{\min},+\infty)$, $\tau_{\min}>0$, and $m \in
\R$, be \st for all multi-indices $\mi\in \N^n$, $\smi\in \N^{n-1}$ we
have
  \begin{equation*}
    \abs{\d_x^\mi \d_{\xi'}^\smi a(\y')}
    \leq C_{\mi,\smi} \lambdat^{m-\abs{\smi}},
      \quad x\in \Rpb,\ \xi' \in\R^{n-1},\ \tau \in  [\tau_{\min},+\infty),
  \end{equation*}
  where $\lambdat=|(\xi',\tau)| = \big( |\xi'|^2 + \tau^2 \big)^\hf$.
We write $a \in \Ssct^m(\Rpb\times\R^{n-1})$ or simply $\Ssct^m$.
For $a \in \Ssct^m$ we denote by $\sigma(a)$ its principal part,
that is, its equivalence class in $\Ssct^m / \Ssct^{m-1}$.
\smallskip

We also introduce symbol classes that behave
polynomially in the $\xi_n$ variable.
Let $a(\y) \in \Cinf(\Rpb\times\R^n)$, with $\tau$ as a parameter
  in $ [\tau_{\min},+\infty)$,  $\tau_{\min}>0$, and $m\in \N$ and $r \in \R$, be \st
 \begin{equation*}
   a(\y)
   =\sum_{j=0}^m a_j(\y') \xi_n^j,
   \quad a_j\in \Ssct^{m-j+r}, \ \ \y = (\y',\xi_n), \
   \y'=(x,\xi',\tau), 
\end{equation*}
with $x \in \Rpb$, $\xi \in \R^n$, $\tau\geq
   \tau_{\min}$, and $\xi_n \in \R$.
We write $a(\y)\in \Ssc^{m,r}(\Rpb\times\R^n)$ or simply $\Ssc^{m,r}$.
\end{definition}
Note that we have $\Ssc^{m,r} \subset \Ssc^{m+m',r-m'}$, if $m,m'\in
\N$ and $r\in \R$. We shall call the principal symbol of $a$ the symbol
$$
\sigma(a)(\y) = \sum_{j=0}^m \sigma(a_j)(\y') \xi_n^j,
$$
which is a representative of the class of $a$ in $\Ssc^{m,r}/
\Ssc^{m,r-1}$.
\smallskip

Note that $\Ssc^{m,r} \not\subset \Ssc^{m+r}$. For example consider
$a(x,\xi,\tau)= |(\xi',\tau)| \xi_n$ for  $|(\xi',\tau)| \geq 1$.
We have $a \in \Ssc^{2,0} \cap  \Ssc^{1,1}$ and yet $a \notin
\Ssc^2$. In fact observe  that
differentiating with respect to $\xi'$ yields
$$
|\d_{\xi'}^\mi a(x,\xi,\tau) | \leq C_{\mi} |(\xi',\tau)|^{1-\abs{\mi}} |\xi_n|.
$$
An estimate of the form of \eqref{eq: semi-classical symbols} is however
not achieved for $|\mi|\geq 2$. A microlocalization is required to
repair this flaw and to use the two different symbol classes in a
pseudo-differential calculus (See \cite[Theorem~18.1.35]{Hoermander:V3}).

Finally, we define the corresponding spaces of poly-homogeneous
symbols. Such symbols are often referred to as classical symbols; they
are characterized by an asymptotic expansion where each term is positively
homogeneous with respect to $(\xi,\tau)$ (\resp $(\xi',\tau)$):
\begin{definition}
 \label{def: classical symbols}
 We shall say $a\in \Sscl^m(\R^n\times \R^n)$ or simply $\Sscl^m$
  (\resp $\Ssctcl^m(\Rpb\times \R^{n-1})$ or simply $\Ssctcl^m$) if
  there exists $a^{(j)} \in \Ssc^{m-j}$ (\resp $\Ssct^{m-j}$), homogeneous
  of degree $m-j$ in $(\xi,\tau)$ for $\abs{(\xi,\tau)}\geq r_0$, (\resp
  $(\xi',\tau)$ for $\abs{(\xi',\tau)}\geq r_0$), with $r_0\geq 0$, \st
\begin{equation}
  a\sim \sum_{j\geq 0} a^{(j)},\quad  \text{in the sense
    that}\text\quad
  a - \sum_{j=0}^N a^{(j)} \in \Ssc^{m-N-1}\
  (\text{\resp}\ \Ssct^{m-N-1}).
\end{equation}
A representative of the principal part is then given by the first term
in the expansion.
\smallskip

Finally for $m \in \N$ and $r \in \R$,  we shall say that $a(\y)\in
\Sscl^{m,r}(\Rpb\times\R^n)$ or simply $\Sscl^{m,r}$,
if
\begin{equation*}
   a(\y)
   =\sum_{j=0}^m a_j(\y') \xi_n^j,
   \quad \text{with} \ a_j\in \Ssctcl^{m-j+r}, \ \ \y = (\y',\xi_n).
\end{equation*}
The principal part is given by $\sum_{j=0}^m \sigma(a_j)(\y') \xi_n^j$
and is homogeneous of degree $m$ in $(\xi,\tau)$.
\end{definition}

\subsection{Classes of semi-classical pseudo-differential operators}
\label{sec: PsiDO}
For $a \in \Ssc^m(\R^n \times \R^n)$ (\resp $\Sscl^m(\R^n \times \R^n)$) we define the following pseudo-differential
operator in $\R^n$:
\begin{align}
  \label{eq: pseudo}
\displaystyle
  a(x,D,\tau) u(x)= \Op(a) u(x)
  &= (2 \pi)^{-n} \int_{\R^n} e^{i \para{x,\xi}} a(x,\xi,\tau)
  \hat{u}(\xi) \ d \xi, \qquad u \in \S(\R^n),
\end{align}
where $\hat{u}$ is the Fourier transform of $u$.
In the sense of oscillatory integrals we have
\begin{align*}
 a(x,D,\tau) u(x)= \Op(a) u(x) &= (2 \pi)^{-n} \iint_{\R^{2n}} e^{i \para{x-y,\xi}} a(x,\xi,\tau)
  u(y) \ d \xi\ d y.
\end{align*}
We write $\Op(a) \in \Psisc^m(\R^n)$ or simply $\Psiscl^m$ (\resp
$\Psisc^m(\R^n)$ or simply $\Psiscl^m$).  Here $D$ denotes $D_x$. The
principal symbol of $\Op(a)$ is $\sigma(\Op(a)) = \sigma(a) \in
\Ssc^m/\Ssc^{m-1}$ (\resp $\Sscl^m/\Sscl^{m-1}$).  \smallskip

Tangential operators are defined similarly. For $a \in
\Ssct^m(\Rpb\times \R^n)$ (\resp $\Ssctcl^m(\Rpb\times \R^n)$) we set
\begin{align}
  \label{eq: tangential pseudo}
 a(x,D',\tau) u(x)= \Op(a) u(x) &= (2 \pi)^{-(n-1)} \iint_{\R^{2n-2}} e^{i \para{x'-y',\xi'}} a(x,\xi',\tau)
  u(y',x_n) \ d \xi'\ d y', 
\end{align}
for $u \in \S(\Rpb)$,
where $x \in \Rpb$.
Here $D'$ denotes $D_{x'}$. We write $A=\Op(a) \in \Psisct^m(\Rpb)$ or
simply $\Psisct^m$ (\resp $\Psisctcl^m(\Rpb)$ or
simply $\Psisctcl^m$). The principal symbol of $A=\Op(a)$ is
$\sigma(A) = \sigma(a) \in  \Ssct^m/\Ssct^{m-1}$ (\resp $\Ssctcl^m/\Ssctcl^{m-1}$).
\smallskip

Finally for $m \in \N$, $r \in \R$, and $a \in
\Ssc^{m,r}$ (\resp $\Sscl^{m,r}$) with
 \begin{equation*}
   a(\y)
   =\sum_{j=0}^m a_j(\y') \xi_n^j,
   \quad a_j\in \Ssct^{m-j+r} (\text{\resp}\ \Ssctcl^{m-j+r}),\ \ \y = (\y',\xi_n),
\end{equation*}
we set
\begin{align*}
  a(x,D,\tau) = \Op(a) = \sum_{j=0}^m a_j (x,D',\tau) D_n^j,
\end{align*}
and we write $A=\Op(a) \in \Psisc^{m,r}(\Rp)$ or simply $\Psisc^{m,r}$
(\resp $\Psiscl^{m,r}(\Rp)$ or simply $\Psiscl^{m,r}$).
The principal symbol of $A$ is $\sigma(A) (\y)= \sigma(a) (\y) =
\sum_{j=0}^m \sigma(a_j)(\y') \xi_n^j$ in $\Ssc^{m,r}/\Ssc^{m,r-1}$
(\resp $\Sscl^{m,r}/\Sscl^{m,r-1}$).
\smallskip

We provide some basic calculus rules in the case of tangential operators.
\begin{proposition}[composition]
  \label{prop: compositon}
  Let $a \in \Ssct^m$ (\resp $\Ssctcl^m$)  and $b\in \Ssct^{m'}$(\resp
  $\Ssctcl^{m'}$) be two tangential symbols. Then $\Op(a) \Op(b) =
  \Op(c) \in \Psisct^{m+m'}$ (\resp $\Psisctcl^{m+m'}$) with $ c \in
  \Ssct^{m+m'}$ (\resp $\Ssctcl^{m+m'}$) defined by the (oscillatory)
  integral:
  \begin{align*}
    c(\y') = (a\, \#\, b) (\y') & =
  (2\pi )^{-(n-1)} \iint e^{-i \para{y',\eta'}}
  a(x,\xi'+\eta',\tau)\, b(x'+y',x_n,\xi',\tau)\ d y' \ d \eta' \\
  & =\sum_{|\mi| < N} \frac{(-i )^{|\mi|}}{\mi!}
  \d_{\xi'}^\mi a(\y')\ \d_{x'}^\mi b(\y')
  + r_N ,
\end{align*}
where $r_N \in \Ssct^{m +m'-N}$ (\resp $\Ssctcl^{m +m'-N}$) is given by
\begin{align*}
  r_N  =  \frac{(-i)^{N}}{(2\pi)^{(n-1)}}\!\!\sum_{|\mi|=N} \int\limits_0^1
  \frac{N(1-s)^{N-1}}{\mi!}  \iint\! e^{-i \para{y',\eta'}}
  \d_{\xi'}^\mi  a(x,\xi'+ \eta',\tau)
  \d_{x'}^\mi b(x'+s y',x_n,\xi',\tau)\, d y'  d \eta' d s.
\end{align*}
\end{proposition}
\begin{proposition}[formal adjoint]
  \label{prop: adjoint}
  Let $a \in \Ssct^m$ (\resp $\Ssctcl^m$). There exists $a^\ast \in
  \Ssct^m$ (\resp  $\Ssctcl^m$) \st
  \begin{align*}
    \para{\Op(a) u,v}_+ = \para{u,\Op(a^\ast)v}_+, \qquad u,v \in \S(\Rpb).
  \end{align*}
  and $a^\ast$ is given be the following asymptotic expansion
\begin{align*}
    a^\ast (\y') & =
    (2\pi )^{-(n-1)} \iint e^{-i \para{y',\eta'}}
    \ovl{a}(x'+y', x_n ,\xi'+\eta',\tau)\ d y' \ d \eta' \\
  & =\sum_{|\mi| < N} \frac{(-i )^{|\mi|}}{\mi!}
  \d_{\xi'}^\mi \d_{x'}^\mi \ovl{a}(\y')
  + r_N , \quad r_N \in \Ssct^{m -N} (\text{\resp} \  \Ssctcl^{m -N}) ,
\end{align*}
where
\begin{align*}
  r_N  =  \frac{(-i)^{N}}{(2\pi)^{(n-1)}}\!\!\sum_{|\mi|=N} \int\limits_0^1
  \frac{N(1-s)^{N-1}}{\mi!}  \iint\! e^{-i \para{y',\eta'}}
  \d_{\xi'}^\mi
  \d_{x'}^\mi \ovl{a}(x'+s y',x_n,\xi'+ \eta',\tau)\, d y'  d \eta' d s.
\end{align*}
  We denote $\Op(a)^\ast = \Op(a^\ast)$.
  We refer to $\Op(a)^\ast$ as to the formal adjoint of $\Op(a)$.
\end{proposition}
A consequence of the previous calculus results is the following proposition.
\begin{proposition}
  \label{prop: calculus consequences}
  Let $a(\y') \in \Ssct^{m}$ (\resp $\Ssctcl^{m}$) and  $b(\y')\in
  \Ssct^{m'}$ (\resp $\Ssctcl^{m'}$), with $m,m' \in \R$.
  Define $h(\y') = D_{x'}(b\d_{\xi'}\ovl{a})(\y')
  \in S_\tau^{m+m'-1}$. Then we have
  $$
  \Op(a)^\ast \Op(b) - \Op(\ovl{a} b + h) \in \Psisct^{m+m'-2} \
  (\text{\resp}\ \Psisctcl^{m+m'-2}), 
  $$
  or equivalently
  $a^\ast \# b - \ovl{a} b -h \in \Ssct^{m+m'-2}$ (\resp $\Ssctcl^{m+m'-2}$).
\end{proposition}

For semi-classical operators in the half space with symbols that are polynomial in $\xi_n$ we
also provide a notion of formal adjoint.
\begin{definition}
  \label{def: formal ajoint}
   Let $b \in \Ssc^{m,r}$ (\resp $\Sscl^{m,r}$), with
  \begin{align*}
    b(x,D,\tau)  = \sum_{j=0}^m b_j (x,D',\tau) D_n^j, \qquad b_j \in
    \Ssct^{m+r-j} \ (\text{\resp}\ \Ssctcl^{m+r-j} ).
  \end{align*}
  We set
  \begin{align*}
    b(x,D,\tau)^\ast  = \sum_{j=0}^m D_n^j b_j (x,D',\tau)^\ast.
  \end{align*}
\end{definition}
In other words, in this definition we ignore the possible occurrence of
boundary terms when performing the operator transposition.

Note that for $a \in \Ssctcl^m$ we have $\cro{D_n, \Op(a)} = \Op(D_n a)
\in \Psisctcl^m$ and more generally, for $j \geq 1$,  we have
\begin{align*}
  \cro{D_n^j, \Op(a)} = \sum_{k=0}^{j-1} \Op(\mi_k) D_n^k, \qquad
  \mi_k \in \Ssctcl^{m},
\end{align*}
where the symbols $\mi_k$ involve various derivatives of $a$ in the
$x_n$-direction.
As an application we see that if we consider $a_j \in \Ssctcl^{m-j+r}$
then we have
\begin{align*}
  \sum_{j=0}^m D_n^j a_j (x,D',\tau) = \sum_{j=0}^m \tilde{a}_j (x,D',\tau) D_n^j,
\end{align*}
where $\tilde{a}_j \in \Ssctcl^{m-j+r}$ and its principal part
satisfies $\sigma(\tilde{a}_j) \equiv
a_j$  in $\Ssct^{m-j+r} / \Ssct^{m-j+r-1}$. Hence
\begin{align*}
  \sigma\Big(\sum_{j=0}^m D_n^j a_j (x,D',\tau)\Big) = \sum_{j=0}^m
  a_j (x,\xi',\tau) \xi_n^j \mod \Ssc^{m,r-1}.
\end{align*}
From the calculus rules given above for the tangential operators and
the above observation we have the following results on the principal
symbols.
\begin{proposition}
  \label{prop: adjoint + composition S m,r}
  Let $a \in \Ssc^{m,r}$ (\resp  $\Sscl^{m,r}$) and $b \in
  \Ssc^{m',r'}$ (\resp  $\Sscl^{m',r'}$) with
  $$
  a(\y)=\sum_{j=0}^m a_j(\y')\xi_n^j,\quad
  b(\y)=\sum_{j=0}^{m'}b_j(\y')\xi_n^j, \quad
  \y = (\y',\xi_n), \ \y' = (x,\xi',\tau).
  $$
  \begin{enumerate}
    \item
  We have $a(x,D,\tau)^\ast \in \Psisc^{m,r}$ (\resp $\Psiscl^{m,r}$) and
  \begin{align*}
    \sigma\big(a(x,D,\tau)^\ast\big) \equiv \sum_{j=0}^m \ovl{a}_j
    (\y')\xi_n^j \in \Ssc^{m,r} /\Ssc^{m,r-1} \ (\text{\resp} \ \Sscl^{m,r} /\Sscl^{m,r-1}).
  \end{align*}
  Moreover we have $\Op(a)^\ast  - \Op(\ovl{a}) \in \Psisc^{m,r-1}$
  (\resp $\Psiscl^{m,r-1}$).
\item
    $a(x,D,\tau) b (x,D,\tau) \in \Psisc^{m+m',r+r'}$ (\resp $\Psiscl^{m+m',r+r'}$) and
  \begin{multline*}
    \sigma\big( a(x,D,\tau) b (x,D,\tau)\big) \equiv 
    \sum_{{0\leq j\leq m} \atop {0 \leq k\leq m'}}
    a_j (\y') b_k (\y') \xi_{n}^{j+k} \in \Ssc^{m +m', r + r'} /
    \Ssc^{m+m',r+r'-1} \\
    (\text{\resp}\ \Sscl^{m +m', r + r'} /
    \Sscl^{m+m',r+r'-1} ). 
  \end{multline*}
  We have
  $\Op(a) \Op(b) u - \Op(ab) u \in \Psisc^{m+m',r+r'-1}$ (\resp $\Psiscl^{m+m',r+r'-1}$). 
  \end{enumerate}
\end{proposition}

\subsection{Sobolev continuity results}

Here we state continuity results for the operators defined above using
the Sobolev  norms with parameters introduced in Section~\ref{sec: sobolev norms}. Such
results can be obtained from their standard counterparts.
\medskip


Let $\lambdat(\xi',\tau)=\para{\tau^2+\abs{\xi'}^2}^{1/2}$ and
$\Lambdat:=\Op (\lambdat)$. For a given real number $s$, the boundary
norm given by \eqref{eq: norm Sobolev m s} is equivalent to the
following norms (see \eqref{eq: def norms} for the definition of $|.|_{p,\tau}$):
\begin{equation}
  \norm{\mathbf{u}}^2_{m,s,\tau}
  =\sum_{k=0}^m |\Lambdat^su_k|_{m-k,\tau}^2,
  \quad \mathbf{u}=(u_0,\dots,u_m) \in \big( \S(\R^{n-1})\big)^{m+1}.
\end{equation}
Moreover, we define the following semi-classical interior norm
\begin{equation}
\Norm{u}^2_{m,s,\tau}=\Norm{\Lambdat^su}^2_{m,\tau},\quad
u \in \S(\Rpb).
\end{equation}
\begin{proposition}
  \label{prop: sobolev regularity}
  If $a(\y)\in \Ssc^{m,r}$, with $m \in \N$ and $r \in
  \R$,  then for $m' \in \N$ and $r' \in
  \R$ there exists $C>0$ \st
  \begin{align*}
    \Norm{\Op(a) u }_{m',r',\tau}
    \leq C
    \Norm{u}_{m+m',r+r',\tau},  \qquad u \in \S(\Rpb).
  \end{align*}
\end{proposition}
A consequence of this results and Proposition~\ref{prop: adjoint +
  composition S m,r} is the following property.
\begin{corollary}
  \label{cor: regularity adjoint - operator}
  Let $a \in \Ssc^{m,r}$ and $m' \in \N$ and $s \in \R$. We have
  \begin{align*}
  \Norm{a(x,D,\tau)^*u-\ovl{a}(x,D,\tau)u}_{m',s,\tau}
  \leq C \Norm{u}_{m+m',r+s-1,\tau},  \qquad u \in \S(\Rpb).
  \end{align*}
\end{corollary}

The following simple inequality will be used implicitly at many places in
what follows when we invoke the parameter $\tau$ to be chosen \suff
large. This will then allow us to absorb semi-classical norms
of lower order.
\begin{corollary}
  \label{cor: semi-classical argument}
  Let $m\in \N$ and $s \in \R$ and $\ell \geq 0$. For some $C>0$, we have
  \begin{align*}
  \Norm{u}_{m,s,\tau}
  \leq C \tau^{-\ell}\Norm{u}_{m,s+\ell,\tau},  \qquad u \in \S(\Rpb).
  \end{align*}
\end{corollary}
This implies that $\Norm{u}_{m,s,\tau}\ll \Norm{u}_{m,s+\ell,\tau}$
for $\tau$ \suff large.

\section{Interface quadratic forms}
\label{sec: interface forms}
\setcounter{equation}{0}

For $a(\y) \in \Sscl^{p,\sigma}(\Rpb\times\R^n)$, we have 
\begin{align*}
  a(\y) = \sum_{j=0}^p a_j(\y') \xi_n^j, \quad\text{with} \ a_j \in \Ssctcl^{p-j +
    \sigma} (\Rpb\times\R^{n-1}),
\end{align*}
and for $\zb = (z_0, \dots, z_{p})\in \C^{p+1}$ we set 
\begin{align}
  \label{eq: sigma a}
  \Sigma_a(\y',\zb) = \sum_{j=0}^p a_j(\y') z_j.
\end{align}

We let $m_\l$ and $m_r$ be two integers. For applications of the
results of this section we shall use the values of $m_\lr$ that come
with the elliptic transmission problem we consider in the present article. 
\begin{definition}[interface quadratic forms]
  \label{def: interface quadratic form}
Let $w=(w_\l,w_r)\in \big(\S (\Rpb)\big)^2$. We say that
\begin{equation*}
  \G (w)= \sum_{s=1}^N\para{A_\l ^s {w_\l}\br +A_r^s {w_r}\br, 
    B_\l^s {w_\l}\br+B_r^s{w_r}\br}_\d,
\end{equation*}
with  $A_\lr^s=a_\lr^s(x,D,\tau)$ and $B_\lr^s=b_\lr^s(x,D,\tau)$,
is an interface quadratic form of type $\para{m_\l-1,m_r-1,\sigma}$
with $\Cinf$ coefficients, if for each $s=1,\dots, N$, we have
$a_\lr^s(\y), b_\lr^s(\y)\in \Sscl^{m_\lr-1,\sigma_\lr}(\Rpb\times\R^n)$, with
$\sigma_\l+\sigma_r=2\sigma$, $\y=(\y',\xi_n)$ with
$\y'=(x,\xi',\tau)$.  

For $\w=(\zb^\l,\zb^r), \tw=(\tzb^\l,\tzb^r) \in\C^{m_\l}\times\C^{m_r}$,
$\zb^\lr=(z^{\lr}_0,\dots,z^{\lr}_{m_\lr-1}), \tzb^\lr=(\tz^{\lr}_0,\dots,\tz^\lr_{m_\lr-1})\in \C^{m_\lr}$
with the interface quadratic form $\G$ we associate the following
bilinear symbol
\begin{align*}
  \Sigma_{\G} (\y', \w,\tw)=\sum_{s=1}^N 
  \big(\Sigma_{a_\l^{s}}(\y',\zb^\l)+\Sigma_{a_r^s}(\y',\zb^r)\big)
  \ovl{\big(\Sigma_{b_\l^s}(\y', \tzb^\l)+\Sigma_{b_r^s}(\y',\tzb^r)\big)}.
\end{align*}
with $\Sigma_{a_\lr^{s}}$ and $\Sigma_{b_\lr^s}$ defined as in \eqref{eq: sigma a}.
\end{definition}
\begin{definition}
   Let $\W$ be an open conic set in $\R^{n-1} \times \R^{n-1} \times \R_+$
  and let $\G$ be an interface quadratic form of type
  $(m_\l-1,m_r-1,\sigma)$ associated with the bilinear symbol
  $\Sigma_{\G}(\y',\w,\tw)$. We say that $\G$ is positive
  definite in $\W$ if there exists $C>0$ and $R>0$ such that
  \begin{align*}
    \Re \Sigma_{\G}(\y'',x_n=0^+,\w,\w) 
    \geq C \Big(
      \sum_{j=0}^{m_\l-1}\lambdat^{2(m_\l-1-j+\sigma_\l)}\bigabs{z^{\l}_j}^2
      +\sum_{j=0}^{m_r-1}\lambdat^{2(m_r-1-j+\sigma_r)}\bigabs{z^{r}_j }^2
      \Big),
  \end{align*}
  for any $\w=(\zb^\l,\zb^r)$,
  $\zb^\lr=(z^{\lr}_{0},\dots,z^{\lr}_{m_\lr-1}) \in \C^{m_\lr}$, and
  $\y'' = (x',\xi',\tau) \in\W$, $\tau\geq 0$, such that $\lambdat= |(\xi',\tau)| \geq R$.
\end{definition}

Then we have the following Lemma
\begin{lemma}
  \label{lemma: Gaarding for interface forms}
  Let $\W$ be an open conic set in $\R^{n-1}\times \R^{n-1} \times
  \R_+$ and let $\G$ be an interface quadratic form of type
  $(m_\l-1,m_r-1,\sigma)$ that is positive definite in $\W$.
  Let $\chi \in \Ssct^0$ be homogeneous of degree 0, with
  $\supp(\chi\br)\subset \W$ and  let $N\in \N$. Then there exist $\tau_\ast\geq
  1$, $C>0$, $C_N >0$ \st  
  \begin{align*}
    \Re \G (\Op(\chi) u) 
    &\geq  C \big(
    \norm{\trace(\Op(\chi) u_\l)}^2_{m_\l-1,\sigma_\l,\tau}
    +\norm{\trace(\Op(\chi) u_r)}^2_{m_r-1,\sigma_r,\tau}
    \big)\\
    &\quad- C_N \big(
    \norm{\trace(u_\l)}^2_{m_\l-1,\sigma_\l-N,\tau}
    +\norm{\trace(u_r)}^2_{m_r-1,\sigma_r-N,\tau}
    \big)
    \end{align*}
    for $u=(u_\l,u_r)\in \big(\S(\Rpb)\big)^2$ and $\tau\geq\tau_\ast$.
\end{lemma}
\begin{proof}
The interface quadratic form can be written as
\begin{align*}
  \G(u)&=\sum_{j,k=0}^{m_\l-1}
  \big(
  G^{\l\l}_{kj}\Lambdat^{m_\l-1-j+\sigma_\l} {D^j_nu_\l} \br,
  \Lambdat^{m_\l-1-k+\sigma_\l} {D^k_n u_\l}\br \big)_\d\\
  &\quad 
  +\sum_{j,k=0}^{m_r-1}\big( 
  G^{rr}_{kj}\Lambdat^{m_r-1-j+\sigma_r} {D^j_n u_r}\br,
  \Lambdat^{m_r-1-k+\sigma_r} {D^k_n u_r}\br \big)_\d\\
  & \quad +\sum_{j=0}^{m_\l-1}\sum_{k=0}^{m_r-1}\big( 
  G^{r\l}_{kj}\Lambdat^{m_\l-1-j+\sigma_\l} {D^j_n u_\l}\br, 
  \Lambdat^{m_r-1-k+\sigma_r} {D^k_n u_r}\br \big)_\d\\
  & \quad +\sum_{j=0}^{m_r-1}\sum_{k=0}^{m_\l-1}\big( 
  G^{\l r}_{kj}\Lambdat^{m_r-1-j+\sigma_r} {D^j_n u_r}\br,
  \Lambdat^{m_\l-1-k+\sigma_\l} {D^k_n u_\l}\br \big)_\d,
\end{align*}
where $G^{i i'}_{jk} = \Op(g^{i i'}_{jk})\in \Psisctcl^{0}$, with $i, i'
= \lr$. 

We set the $2m \times 2m$-matrix tangential symbol
\begin{equation*}
  \mathbf{g} (\y')
  =\begin{pmatrix}
    g^{\l\l} & g^{\l r} \\
    g^{r\l} & g^{rr}
  \end{pmatrix}(\y'), \qquad 
  g^{i i'} (\y')
  = \big(g^{i i'}_{jk} (\y')\big)_{{0\leq j\leq m_i-1 }\atop{0\leq
      k\leq m_{i'}-1 }}, \quad i,i' = \lr.
\end{equation*}
We introduce $\tchi \in \Ssct^0$ that has the same properties
as $\chi$ with moreover $0\leq \tchi\leq 1$ and $\tchi =
1$ in a \nhd of $\supp \chi$.
We then set 
\begin{align*}
  \tilde{\mathbf{g}}= \tchi \mathbf{g} + (1- \tchi) I_{2m}, 
\end{align*}
where $I_{2m}$ is the $2m \times 2m$  identity matrix.

As $\G$  is positive
definite in $\W$ we have, for some $C>0$, 
\begin{align*}
  \Re \big( \mathbf{g}(\y'',x_n=0^+) \w, \w\big) 
  \geq C \abs{\w}_{\C^{2m}}^2, \quad \y''\in \W, \ \w \in \C^{2m}.
\end{align*}
Therefore we have, for some $C'>0$,
\begin{align}
  \label{eq: positivity tilde g}
  \Re \big( \tilde{\mathbf{g}}(\y'',x_n=0^+)\w, \w \big) 
  \geq C' \abs{\w}_{\C^{2m}}^2, 
  \quad \y'' \in \R^{n-1} \times \R^{n-1} \times \R_+, \ 
  \w \in \C^{2m}.
\end{align}

\medskip For a function $v$ we define the $m_\lr$-tuple functions
$V_\lr=(v_{\lr,0},\dots,v_{\lr,m_\lr-1})$ by
\begin{equation*}
  v_{\lr,k}=\Lambdat^{m_\lr+\sigma_\lr-1-k}D_n^k {v_\lr}\br,\quad k=0,\dots,m_\lr-1.
\end{equation*}
We then have, for $N \in \Z$,
\begin{align}
  \label{eq: equiv norms V}
  \norm{V_\lr}_{N,\tau}^2
  &=\sum_{k=0}^{m_\lr-1}\norm{v_{\lr,k}}^2_{N,\tau}
  =\sum_{k=0}^{m_\lr-1}\bignorm{\Lambdat^{m_\lr+\sigma_\lr-1-k}D_n^k {v_\lr}\br}_{N,\tau}^2\\
  &\asymp \sum_{k=0}^{m_\lr-1} \bignorm{\Lambdat^{\sigma_\lr+N} D_n^k v_\lr}_{m_\lr-1-k,\tau}^2
  =\norm{\trace(v_\lr)}^2_{m_\lr-1,\sigma_\lr+N,\tau}.\notag
\end{align}

We set $\uu_\lr = \Op(\chi) u_\lr$ and introduce  $U_\lr = (u_{\lr,0}, \dots, u_{\lr,m_\lr-1})$
and $\udl{U}_\lr =   (\uu_{\lr,0}, \dots, \uu_{\lr,m-1})$ as above:
\begin{equation*}
  u_{\lr,k}=\Lambdat^{m_\lr+\sigma_\lr-1-k}D_n^k {u_\lr}\br,\quad 
  \uu_{\lr,k} = \Lambdat^{m_\lr+\sigma_\lr-1-k}D_n^k {\uu_\lr}\br, \quad k=0,\dots,m_\lr-1.
\end{equation*}

Setting $\transp \udl{U} = (\udl{U}_\l, \udl{U}_r)$ we obtain
\begin{align*}
  \G (\uu) = \big( \Op(\mathbf{g}\br) \udl{U},\udl{U}\big)_\d.
\end{align*}
Writing $\mathbf{g} = \tilde{\mathbf{g}}  + \mathbf{r}$ with $\mathbf{r} = (\mathbf{g} - I_{2m}) (1 -
\tchi)$ we find 
\begin{align*}
  \G (\uu) = \big(\Op(\tilde{\mathbf{g}}\br) \udl{U},\udl{U}\big)_\d 
  + \big(\Op(\mathbf{r}\br) \udl{U},\udl{U}\big)_\d 
\end{align*}
As the supports of $1- \tchi$ and $\chi$ are disjoint, with the
pseudo-differential calculus, for any $N\in \N$  we have for some $C_N>0$ 
\begin{align}
  \label{eq: garding interface form 1}
  \big|\big(\Op(\mathbf{r}\br) \udl{U},\udl{U}\big)_\d \big| \leq C_N \norm{U}^2_{-N,\tau}.
\end{align}
Next, from \eqref{eq: positivity tilde g} with the G{\aa}rding inequality in the tangential
direction we deduce that 
for  some $C>0$ we have 
\begin{align}
  \label{eq: garding interface form 2}
  \Re \big(\Op(\mathbf{\tilde{g}}\br) \udl{U},\udl{U}\big)_\d \geq  C \norm{\udl{U}}_{\d}^2, 
\end{align}
for $\tau$ sufficiently large.  Combining \eqref{eq: garding interface
  form 1}--\eqref{eq: garding interface form 2} with~\eqref{eq: equiv
  norms V} yields the conclusion.
\end{proof}

\begin{proposition}
  \label{prop: interface symbol positivity}
  Assume that the transmission condition of Definition~\ref{def:
    transmission Lopatinskii} holds at $\y'_0=(x_0,\xi_0',\tau_0) \in \sphbundle(V)$
  with $x_0 \in S$ (see also \eqref{eq: local reformulation
    transmission 1}--\eqref{eq: local reformulation transmission 2}
  and \eqref{eq: rank formulation transmission condition} for a formulation in the local
  setting).  Then there exists $\U_1$ a conic open \nhd of
  $\y'_0$ in $\ovl{V_+} \times \R^{n-1} \times \R_+$ \st 
\begin{multline*} 
  \sum_{j=1}^{m} \lambdat^{2(m - 1/2- \torder^j)} \bigabs{ \Sigma_{\tlv^j}(\y', \zb^\l) 
    + \Sigma_{\trv^j}(\y',\zb^r) }^2\\
  +  \sum_{j=m+1}^{m_\l'} \lambdat^{2(m_\l -1/2 - \torder_\l^j)} 
  \bigabs{ \Sigma_{\elv^j}(\y',\zb^{\l})}^2
  + \sum_{j=m+1}^{m_r'} \lambdat^{2(m_r -1/2 - \torder_r^j)} 
  \bigabs{ \Sigma_{\erv^j}(\y',\zb^{r})}^2\\
  \geq C \Big(\sum_{j=0}^{m_\l-1} \lambdat^{2(m_\l  -1/2-j)} | z^\l_j|^2 
  + \sum_{j=0}^{m_r-1} \lambdat^{2(m_r  -1/2 -j)}  |z^r_j|^2\Big), 
\end{multline*}
for $\y \in \U_1$ and $\zb^\lr=(z^{\lr}_0,\dots,z^{\lr}_{m_\lr-1})\in \C^{m_\lr}$.
\end{proposition}
We recall that $\torder^j = (\torder^j_\l + \torder^j_r)/2$ for $j=1, \dots, m$, and that we have
\eqref{eq: assumption beta lr}.
\begin{proof}
  With the transmission condition holding at $\y'_0$, by Proposition~\ref{prop: stability
    transmission} there exists a conic open
  set $\U_1$ \nhd of $\y'_0$ where condition~\eqref{eq: rank formulation
  transmission condition} is valid. Observe that $\mathcal K = \ovl{\U_1} \cap
  \ovl{\sphbundle(V)}$ is compact, recalling that $V_+$ is bounded.

Let $\y'_1\in \mathcal K$ and $\mathcal{T}(\y_1')$ be as in \eqref{eq: rank formulation transmission condition}. We have  
\begin{align*}
  \rank \mathcal{T}(\y_1')
  = \rank \ovl{\mathcal{T}(\y_1')}\  \transp \mathcal{T}(\y_1')= 2m,
\end{align*}
For $\w=(\zb^\l,\zb^r) \in \C^{2m}$,
$\zb^\lr=(z^{\lr}_0,\dots,z^{\lr}_{m_\lr-1})\in \C^{m_\lr}$, recalling
that $2m = m_\l +m_r$, we thus have 
$\big( \ovl{\mathcal{T}(\y_1')}\ 
   \transp \mathcal{T}(\y_1')\w, \w\big)\geq C |\w|_{\C^{2m}}^2$, for some
   $C>0$. 

Observe that we have 
\begin{align*}
  \ovl{\mathcal{T}(\y_1')}\  \transp \mathcal{T}(\y_1')
  = \begin{pmatrix}
    \ovl{\mathcal{T}_\l^1(\y_1')} \ \transp  \mathcal{T}_\l^1(\y_1') 
    & \ovl{\mathcal{T}_\l^1(\y_1')} \ \transp  \mathcal{T}_r^1(\y_1')\\
    \ovl{\mathcal{T}_r^1(\y_1')} \ \transp  \mathcal{T}_\l^1(\y_1')
    & \ovl{\mathcal{T}_r^1(\y_1')} \ \transp  \mathcal{T}_r^1(\y_1')
  \end{pmatrix}
  + \begin{pmatrix}
    \ovl{\mathcal{T}_\l^2(\y_1')} \ \transp  \mathcal{T}_\l^2(\y_1') 
    & 0\\
    0
    & \ovl{\mathcal{T}_r^2(\y_1')} \ \transp  \mathcal{T}_r^2(\y_1')
  \end{pmatrix}.
\end{align*}
For $\w=(\zb^\l,\zb^r)$, $\zb^\lr=(z^{\lr}_0,\dots,z^{\lr}_{m_\lr-1}) \in \C^{m_\lr}$,  we have 
 \begin{align*} 
   \big( \ovl{\mathcal{T}(\y_1')}\ 
   \transp \mathcal{T}(\y_1')\w, \w\big)
   & =  \bigabs{\transp  \mathcal{T}_\l^1(\y_1') \zb^\l 
     + \transp \mathcal{T}_r^1(\y_1') \zb^r  }_{\C^m}^2
   + \bigabs{\transp  \mathcal{T}_\l^2(\y_1') \zb^\l }_{\C^{m_\l^-}}^2
     + \bigabs{\transp \mathcal{T}_r^2(\y_1') \zb^r  }_{\C^{m_r^-}}^2
   \\
   &=\sum_{j=1}^{m} \Bigabs{
     \sum_{i=0}^{m_\l-1} t_{\l,i}^j(\y_1')z^{\l}_{i}
     +\sum_{i=0}^{m_r-1}  t_{r,i}^j(\y_1')z^{r}_{i}}^2\\
   &\quad + \sum_{j=1}^{m_\l^-} 
    \bigabs{ \sum_{i=0}^{m_\l-1}  e_{\l,i}^{j+m}(\y_1')z^{\l}_{i}}^2
    + \sum_{j=1}^{m_r^-} 
    \bigabs{ \sum_{i=0}^{m_r-1}  e_{r,i}^{j+m}(\y_1')z^{r}_{i}}^2
   \\
   &= \sum_{j=1}^{m} \abs{ \Sigma_{\tlv^j}(\y_1', \zb^\l) 
     + \Sigma_{\trv^j}(\y_1',\zb^r) }^2
   +  \sum_{j=1}^{m_\l^-} 
    \bigabs{  \Sigma_{\elv^{j+m}}(\y_1',\zb^{\l})}^2
    + \sum_{j=1}^{m_r'^-} 
    \bigabs{ \Sigma_{\erv^{j+m}}(\y_1',\zb^{r})}^2.
\end{align*} 
We thus obtain
\begin{align*} 
\sum_{j=1}^{m} \abs{ \Sigma_{\tlv^j}(\y_1', \zb^\l) 
     + \Sigma_{\trv^j}(\y_1',\zb^r) }^2
   +  \sum_{j=1}^{m_\l^-} 
    \bigabs{  \Sigma_{\elv^{j+m}}(\y_1',\zb^{\l})}^2
    + \sum_{j=1}^{m_r^-} 
    \bigabs{ \Sigma_{\erv^{j+m}}(\y_1',\zb^{r})}^2\gtrsim |(\zb^\l,\zb^r)|^2.
\end{align*}
By continuity this inequality remains true in a small \nhd
of $\y'_1$ in $\mathcal K$. Using the compactness of $\mathcal K$ we thus find 
\begin{align*} 
  \sum_{j=1}^{m} \abs{ \Sigma_{\tlv^j}(\y', \zb^\l) 
     + \Sigma_{\trv^j}(\y',\zb^r) }^2
   +  \sum_{j=1}^{m_\l^-} 
    \bigabs{  \Sigma_{\elv^{j+m}}(\y',\zb^{\l})}^2
    + \sum_{j=1}^{m_r^-} 
    \bigabs{ \Sigma_{\erv^{j+m}}(\y',\zb^{r})}^2
     \gtrsim |(\zb^\l,\zb^r)|^2, 
\end{align*}
for $\y' \in {\mathcal K}$ and $\zb^\lr=(z^{\lr}_0,\dots,z^{\lr}_{m_\lr-1}) \in \C^{m_\lr}$.
Introducing   the map 
\begin{align*}
  M_t \y'  = (x, t \eta),
  \qquad  \y' = (x,\eta) \in \Rpb \times \R^{n-1} \times \R_+, \quad t>0,
\end{align*}
as we have  $\ovl{\U_1} = \{ M_t \y'; \  t >0, \  \y' \in \mathcal
K\}$, we find
\begin{align*} 
\sum_{j=1}^{m} \abs{ \Sigma_{\tlv^j}(M_t  \y', \tzb^\l) 
     + \Sigma_{\trv^j}(M_t \y',\tzb^r) }^2
   +  \sum_{j=1}^{m_\l^-} 
    \bigabs{  \Sigma_{\elv^{j+m}}(M_t \y',\tzb^{\l})}^2
    + \sum_{j=1}^{m_r^-} 
    \bigabs{ \Sigma_{\erv^{j+m}}(M_t \y',\tzb^{r})}^2\gtrsim |(\tzb^\l,\tzb^r)|^2, 
\end{align*}
where $t = \lambdat^{-1}=|(\xi',\tau)|^{-1}$ and
$\tzb^\lr=(\tz^{\lr}_0,\dots,\tz^\lr_{m_\lr-1})\in
\C^{m_\lr}$, 
with $\tz^\lr_j = t^{-m_\lr + 1/2 +j} z^\lr_j$, yielding
\begin{multline*} 
  \sum_{j=1}^{m} \bigabs{ \lambdat^{-(\torder_\l^j - m_\l +1/2)} \Sigma_{\tlv^j}(\y', \zb^\l) 
    + \lambdat^{-(\torder_r^j - m_r +1/2)} \Sigma_{\trv^j}(\y',\zb^r) }^2\\
  +  \sum_{j=1}^{m_\l^-} 
  \bigabs{ \lambdat^{-(\torder_\l^j - m_\l +1/2)} \Sigma_{\elv^{j+m}}(\y',\zb^{\l})}^2
  + \sum_{j=1}^{m_r^-} 
  \bigabs{ \lambdat^{-(\torder_r^j - m_r +1/2)}
    \Sigma_{\erv^{j+m}}(\y',\zb^{r})}^2\\
  \gtrsim \sum_{j=0}^{m_\l-1} |\lambdat^{m_\l  -1/2-j} z^\l_j|^2 
  + \sum_{j=0}^{m_r-1} |\lambdat^{m_r  -1/2 -j} z^r_j|^2.
\end{multline*}
With~\eqref{eq: assumption beta lr} we obtain the sought result.
\end{proof}

\section{Proof of the Carleman estimate}
\label{sec: proof Carleman}
\setcounter{equation}{0}
As is usual in the proof of Carleman estimates we consider the
following conjugated operators
 $$P_{\lr,\varphi} = e^{\tau \varphi_\lr} P_\lr e^{-\tau  \varphi_\lr}.$$
 As $e^{\tau \varphi_\lr} D_j e^{-\tau \varphi_\lr} = D_j + i \tau
\d_j \varphi_\lr \in \Psiscl^{1,0}$, we see that $P_{\lr,\varphi} \in \Psiscl^{2m,0}$. Their principal
symbols are given by $p_{\lr,\varphi}(\y) = p_\lr(x,\xi+i \tau
\varphi_\lr'(x)) \in \Sscl^{2m,0}$.
 
Similarly we recall that we set 
$$
\Tlrv^j  = e^{\tau \varphi_\lr} \Tlr^j e^{-\tau  \varphi_\lr} \in
\Psiscl^{\torder_\lr^j,0}, \quad j=1, \dots, m, 
$$
with principal symbols $\tlrv(\y) = \tlr^j (x,\xi+
i\tau\varphi_\lr'(x)) \in \Sscl^{\torder^j_\lr,0}$.

We start the proof of the main theorem with a microlocal 
estimate that exploits the transmission condition.
\subsection{Estimate with the transmission condition}
\label{sec: estimate with transmission condition}

Let us first consider a polynomial function with roots with negative imaginary parts in a
microlocal region. Then, we have the following perfect microlocal
elliptic estimate. We refer to Lemma~4.1 in Part~I \cite{BLR:13} for a proof.
\begin{lemma}
   \label{lemma: elliptic estimate}
   Let $h(\y',\xi_n)\in \Ssc^{k,0}$, $\y' = (x,\xi',\tau)$, with $k\geq 1$, be polynomial
   in $\xi_n$ with homogeneous coefficients in $(\xi',\tau)$ and $H = h(x,D,\tau)$.  When viewed as a
   polynomial in $\xi_n$ the leading coefficient is $1$.  Let $\U$ be
   a conic open subset of $\ovl{V_+} \times \R^{n-1}\times \R_+$.  We assume
   that all the roots of $h(\y',\xi_n)=0$ have negative imaginary part
   for $\y'=(x,\xi',\tau) \in \U$ . Letting $\chi(\y') \in \Ssct^0$ be
   homogeneous of degree $0$ and \st $\supp(\chi) \subset \U$, and $N \in \N$, there exist
   $C>0$, $C_N>0$, and $\tau_\ast>0$ \st
  \begin{equation*}
    \Norm{\Op(\chi)w}^2_{k,\tau}
    +\abs{\trace(\Op(\chi)w)}^2_{k-1,1/2,\tau}
    \leq C
    \Norm{H\Op(\chi)w}_{+}^2
    +C_N\big( \Norm{w}^2_{k,-N,\tau}
    + \norm{\trace(w)}^2_{k-1,-N,\tau}\big),
\end{equation*}
for $w\in \S(\Rpb)$ and $\tau\geq \tau_\ast$.
\end{lemma}

\medskip
Now, we consider a point in the cotangent bundle, at the interface
where the transmission condition holds. We then obtain an estimate of
an interface norm.
\begin{proposition}
  \label{prop: transmission step}
  Assume that the transmission condition of Definition~\ref{def:
    transmission Lopatinskii} is satisfied at
  $\y'_0=(x_0,\xi_0',\tau_0) \in \sphbundle(V)$ with $x_0\in S
  \cap V$.  Then there exists $\U$, a conic open \nhd of
  $\y'_0$ in $\ovl{V_+} \times \R^{n-1} \times \R_+$, \st for $\chi \in
  \Ssct^0$, homogeneous of degree 0,  with $\supp(\chi) \subset \U$, there exist
  $C>0$ and $\tau_\ast>0$ \st
  \begin{align*}
    &C \big(\norm{\trace(\Op(\chi) v_\l)}^2_{m_\l-1,1/2,\tau}
    + \norm{\trace(\Op(\chi) v_r)}^2_{m_r-1,1/2,\tau}\big)
    \\
    &\qquad \qquad 
    \leq\sum_{j=1}^m \bignorm{\Tlv^j {v_\l}\br + \Trv^j {v_r}\br}_{m-1/2-\torder^j,\tau}^2
    + \Norm{\Plv v_\l}_{+}^2 + \Norm{\Prv v_r}_{+}^2\\
     &\qquad \qquad \quad + \Norm{v_\l}_{m_\l,-1,\tau}^2 + \Norm{v_r}_{m_r,-1,\tau}^2
    + \norm{\trace(v_\l)}^2_{m_\l-1,-1/2,\tau}
    + \norm{\trace(v_r)}^2_{m_r-1,-1/2,\tau},
\end{align*}
for $\tau\geq\tau_\ast$, $v_\l, v_r\in \S(\Rpb)$.
\end{proposition}
\begin{proof}
As the transmission condition holds at $\y'_0$ the local smooth symbol
factorizations of Section~\ref{sec: symbol factorization}, 
\begin{align*}
    \plrv(\y',\xi_n) = \plrv^-(\y',\xi_n)\, \klrv (\y',\xi_n),
\end{align*}
is such that condition~\eqref{eq: rank formulation
  transmission condition}  is valid for $\y' \in \U_0$ with $U_0$ a
conic \nhd of $\y'_0$ in $\ovl{V_+} \times \R^{n-1}
  \times \R_+$. Moreover, by Proposition~\ref{prop: interface symbol positivity} there exists
  $\U_1 \subset \U_0$, a conic open \nhd of $\y'_0$,  \st
\begin{multline*} 
  \sum_{j=m+1}^{m_\l'} \lambdat^{2(m_\l -1/2 - \torder_\l^j)} 
  \bigabs{ \Sigma_{\elv^j}(\y',\zb^{\l})}^2
  + \sum_{j=m+1}^{m_r'} \lambdat^{2(m_r -1/2 - \torder_r^j)} 
  \bigabs{ \Sigma_{\erv^j}(\y',\zb^{r})}^2\\
  +\sum_{j=1}^{m} \lambdat^{2(m - 1/2- \torder^j)} \bigabs{ \Sigma_{\tlv^j}(\y', \zb^\l) 
    + \Sigma_{\trv^j}(\y',\zb^r) }^2
  \gtrsim \sum_{j=0}^{m_\l-1} \lambdat^{2(m_\l  -1/2-j)} | z^\l_j|^2 
  + \sum_{j=0}^{m_r-1} \lambdat^{2(m_r  -1/2 -j)}  |z^r_j|^2, 
\end{multline*}
for $\y' \in \U_1$ and $\zb^\lr\in \C^{m_\lr}$.
We now choose $\U$ a conic open subset, \nhd of $\y'_0$, \st $\ovl{\U} \subset
\U_1$.  We let  $\chi$ be as in the statement and  
we also choose $\tchi \in \Ssct^0$ homogeneous of degree $0$ with $\supp(\tchi)
\subset \U_1$  and $\tchi =
1$ in a \nhd of  $\ovl{\U}$. Then,
\begin{multline}
  \label{eq: ineq symbol transmission}
  \sum_{j=m+1}^{m_\l'} \lambdat^{2(m_\l -1/2 - \torder_\l^j)} 
  \bigabs{ \tchi(\y') \Sigma_{ \elv^j}(\y',\zb^{\l})}^2
  + \sum_{j=m+1}^{m_r'} \lambdat^{2(m_r -1/2 - \torder_r^j)} 
  \bigabs{ \tchi(\y') \Sigma_{\erv^j}(\y',\zb^{r})}^2\\
  +\sum_{j=1}^{m} \lambdat^{2(m - 1/2- \torder^j)} \bigabs{ \Sigma_{\tlv^j}(\y', \zb^\l) 
    + \Sigma_{\trv^j}(\y',\zb^r) }^2
  \gtrsim \sum_{j=0}^{m_\l-1} \lambdat^{2(m_\l  -1/2-j)} | z^\l_j|^2 
  + \sum_{j=0}^{m_r-1} \lambdat^{2(m_r  -1/2 -j)}  |z^r_j|^2, 
\end{multline}
  for $\y' \in \ovl{\U}$ and $\zb^\lr\in \C^{m_\lr}$.

We set\footnote{The introduction of $\tchi$ is made so that $\tchi
e_\varphi^k$ is defined on the whole tangential phase-space.} $\Elrv =
\Op(\tchi \elrv)$ and we define the following interface quadratic form
(see Definition~\ref{def: interface quadratic form}): 
\begin{align*}
  \G_S (u) &=  \sum_{j=1}^{m} 
  \bignorm{\Tlv^j {u_\l}\br  + \Trv^j  {u_r}\br }_{m-1/2 - \torder^j,\tau }^2\\
  &\quad 
  + \sum_{j=m+1}^{m_\l'} \bignorm{\Elv^j {u_\l}\br}_{m+ m_\l^- +1/2 - j,\tau}^2
  + \sum_{j=m+1}^{m_r'} \bignorm{\Erv^j {u_\l}\br}_{m+ m_r^- +1/2 - j,\tau}^2.
\end{align*}
Observe that with \eqref{eq: extended transmission order} we have  $m+ m_\lr^- +1/2 - j = m_\lr - 1/2 - \torder_\lr^j$ for
$j=m+1, \dots, m_\lr'$.  We thus find that $\G_S$ is of type $(m^\l-1,
m^r-1, \hf)$ and its bilinear symbol is given by 
  \begin{align*}
    \Sigma_{\G_S} (\y',\w,\tw) 
    &= \sum_{j=1}^{m} \lambdat^{2 (m - 1/2- \torder^j)} 
    \big( \Sigma_{\tlv^j}(\y', \zb^\l) + \Sigma_{\trv^j}(\y',\zb^r)\big)  
    \big( \ovl{\Sigma_{\tlv^j}}(\y', \ovl{\tzb^\l}) +
    \ovl{\Sigma_{\trv^j}}(\y',\ovl{\tzb^r} )\big) \\
    &\quad+ |\tchi (\y')|^2 \sum_{j=m+1}^{m_\l'} \lambdat^{2(m_\l -1/2 - \torder_\l^j)} 
   \Sigma_{\elv^j}(\y',\zb^{\l}) \ovl{\Sigma_{\elv^j}}(\y',\ovl{\tzb^{\l}} )\\
   &\quad + |\tchi (\y')|^2\sum_{j=m+1}^{m_r'} \lambdat^{2(m_r -1/2 - \torder_r^j)} 
   \Sigma_{\erv^j}(\y',\tzb^{r})\ovl{\Sigma_{\erv^j}}(\y',\ovl{\tzb^{r}}),
  \end{align*}
  with $\w=(\zb^\l,\zb^r), \tw=(\tzb^\l,\tzb^r)
  \in\C^{m_\l}\times\C^{m_r}$.
  Hence \eqref{eq: ineq symbol transmission} gives 
  \begin{align*}
    \Sigma_{\G_S} (\y',\w,\tw) 
    \gtrsim \sum_{j=0}^{m_\l-1} \lambdat^{2(m_\l  -1/2-j)} | z^\l_j|^2 
  + \sum_{j=0}^{m_r-1} \lambdat^{2(m_r  -1/2 -j)}  |z^r_j|^2,
  \end{align*}
  for $\y' \in \ovl{\U}$. For any $N \in \N$, by Lemma~\ref{lemma: Gaarding for
    interface forms} there exists
  $\tau_\ast\geq
  1$, $C>0$, $C_N >0$ \st  
  \begin{align}
    \label{eq: transmission 1}
    \G_S (\vv)  = \Re \G_S (\vv) 
    &\geq  C \big(
    \norm{\trace(\vv_\l)}^2_{m_\l-1,1/2,\tau}
    +\norm{\trace(\vv_r)}^2_{m_r-1,1/2,\tau}
    \big)\\
    &\quad- C_N \big(
    \norm{\trace(v_\l)}^2_{m_\l-1,1/2-N,\tau}
    +\norm{\trace(v_r)}^2_{m_r-1,1/2-N,\tau}
    \big), \notag
    \end{align}
    with $\vv = (\vv_\l, \vv_r)$ and $\vv_\lr = \Op(\chi) v_\lr$, for $\v=(v_\l,v_r)\in \big(\S(\Rpb)\big)^2$ and $\tau\geq\tau_\ast$.

\medskip
The functions $\plrv^-(\y',\xi_n)$ and $\klrv(\y',\xi_n)$ in the
symbol factorization recalled at the begining of the proof are polynomial in $\xi_n$ with
homogeneous coefficients in $\y' \in \U_0$ and the leading coefficient
of  $\plrv^-(\y',\xi_n)$ is  equal
to $1$. For $\y' \in \U_0$, their degrees are constant and equal to $m_\lr^-$ and $m_\lr
-m_\lr^-$ respectively. 
We smoothly extend $\plrv^-(\y',\xi_n)$
for $\y'$ outside of $\U_0$ keeping the leading coefficient equal to $1$ and we
denote this extension by $\tplrv^-$. 
In fact we have $\chi \plrv = \chi \klrv \plrv^- = \chi
\tchi \klrv \tplrv^-$.  We thus obtain $ \Op(\chi)\Plrv =
\Op(\tplrv^-) \Op(\chi) \Op(\tchi \klrv) + R_\lr$ with $R_\lr$ in
$\Psisc^{m,-1}$ by the last point of Proposition~\ref{prop: adjoint +
composition S m,r}. Observe that $\tchi\klrv$ is a well defined
symbol. 

Applying Lemma~\ref{lemma: elliptic estimate} to
$\Op(\tplrv^-)$ and $w_\lr = \Op(\tchi \klrv) v_\lr$ we obtain
\begin{align*}
  & \Norm{\Op(\chi)  w_\l}_{m_\l^{-},\tau}^2 + \Norm{\Op(\chi)  w_r}_{m_r^{-},\tau}^2
  + \norm{\trace(\Op(\chi) w_\l)}_{m_\l^- -1,1/2,\tau}^2
  + \norm{\trace(\Op(\chi) w_r)}_{m_r^- -1,1/2,\tau}^2\\
  &\qquad \lesssim
 \bigNorm{\Op(\tplv^-) \Op(\chi)  w_\l}_{+}^2
  + \bigNorm{\Op(\tprv^-) \Op(\chi)  w_r}_{+}^2
  + \Norm{w_\l}^2_{m_\l^-,-N,\tau}
  + \Norm{w_r}^2_{m_r^-,-N,\tau}\\
    &\qquad \quad+ \norm{\trace(w_\l)}^2_{m_\l^- -1,-N,\tau}
    + \norm{\trace(w_r)}^2_{m_r^- -1,-N,\tau}\\
    &\qquad  \lesssim 
    \bigNorm{ \Op(\chi) \Plv v_\l}_{+}^2
    + \bigNorm{ \Op(\chi)\Prv v_r}_{+}^2
    +  \Norm{v_\l}^2_{m_\l,-1,\tau}
    + \Norm{v_r}^2_{m_r,-1,\tau}
    \\
    &\qquad  \quad 
    + \Norm{v_\l}^2_{m_\l,-N,\tau}
    + \Norm{v_r}^2_{m_r,-N,\tau}
    + \norm{\trace(v_\l)}^2_{m_\l -1,-N,\tau}
    + \norm{\trace(v_r)}^2_{m_r -1,-N,\tau}\\
    &\qquad  \lesssim 
    \Norm{\Plv v_\l}_{+}^2
    + \Norm{\Prv v_r}_{+}^2
    + \Norm{v_\l}^2_{m_\l,-1,\tau}
    + \Norm{v_r}^2_{m_r,-1,\tau}
    \\
    &\qquad  \quad 
    + \norm{\trace(v_\l)}^2_{m_\l -1,-N,\tau}
    + \norm{\trace(v_r)}^2_{m_r -1,-N,\tau},
\end{align*}
yielding 
\begin{align*}
  &\sum_{j=0}^{m_\l^{-}-1}
  \bignorm{D_n^j \Op(\chi) {w_\l}\br}_{m_\l^{-}-1/2-j,\tau}^2
  + \sum_{j=0}^{m_r^{-}-1}
  \bignorm{D_n^j \Op(\chi) {w_r}\br }_{m_r^{-}-1/2-j,\tau}^2\\
  &\qquad  \lesssim
  \Norm{\Plv v_\l}_{+}^2
    + \Norm{\Prv v_r}_{+}^2
    + \Norm{v_\l}^2_{m_\l,-1,\tau}
    + \Norm{v_r}^2_{m_r,-1,\tau}
  + \norm{\trace(v_\l)}^2_{m_\l -1,-N,\tau}
    + \norm{\trace(v_r)}^2_{m_r -1,-N,\tau},
\end{align*} 
Recalling that $\elrv^{j+m+1} = \klrv \xi_n^{j}$, $j=0,\dots,m_\lr^- -1$ in
$\U_1$ we have $D_n^j \Op(\chi) \Op(\tchi\klrv) v_\lr =
\Elrv^{j+m+1} \vv_\lr + R_{\lr,j} v_\lr$ with $R_{\lr,j} \in \Psisc^{m_\lr-m_\lr^-+j,-1}$ by the
last point of Proposition~\ref{prop: adjoint + composition S m,r}.  We
then obtain, for $\tau$ chosen \suff large
\begin{align*}
  &\sum_{j=0}^{m_\l^{-}-1}
  \bignorm{\Elv^{j+m+1}  {\vv_\l}\br}_{m_\l^{-}-1/2-j,\tau}^2
  +\sum_{j=0}^{m_r^{-}-1}
  \bignorm{\Erv^{j+m+1} {\vv_r}\br}_{m_r^{-}-1/2-j,\tau}^2\\
  &\qquad \qquad  \lesssim
 \Norm{\Plv v_\l}_{+}^2
 +\Norm{\Prv v_r}_{+}^2
+\Norm{v_\l}^2_{m_\l,-1,\tau} 
 +\Norm{v_r}^2_{m_r,-1,\tau}\\
 &\qquad \qquad \quad+\norm{\trace(v_\l)}^2_{m_\l-1,-1/2,\tau}
 +\norm{\trace(v_r)}^2_{m_r-1,-1/2,\tau}, 
\end{align*}
which we write, by a shift of indices, 
\begin{align}
  \label{eq: transmission 2}
  &\sum_{j=m+1}^{m_\l'}
  \bignorm{\Elv^{j}  {\vv_\l}\br}_{m+ m_\l^{-}+1/2-j,\tau}^2
  +\sum_{j=m+1}^{m_r'}
  \bignorm{\Erv^{j} {\vv_r}\br}_{m+m_r^{-}+1/2-j,\tau}^2\\
  &\qquad \qquad  \lesssim
 \Norm{\Plv v_\l}_{+}^2
 +\Norm{\Prv v_r}_{+}^2
+\Norm{v_\l}^2_{m_\l,-1,\tau} 
 +\Norm{v_r}^2_{m_r,-1,\tau}\nonumber \\
 &\qquad \qquad \quad+\norm{\trace(v_\l)}^2_{m_\l-1,-1/2,\tau}
 +\norm{\trace(v_r)}^2_{m_r-1,-1/2,\tau}, \nonumber
\end{align}
Collecting estimates \eqref{eq: transmission 1} and \eqref{eq:
  transmission 2} we thus obtain
\begin{align*}
  \norm{\trace(\vv_\l)}^2_{m_\l-1,1/2,\tau}
    +\norm{\trace(\vv_r)}^2_{m_r-1,1/2,\tau}
   & \lesssim
  \sum_{j=1}^{m} \bignorm{\Tlv^j {\vv_\l}\br  + \Trv^j  {\vv_r}\br }_{m-1/2 - \torder^j,\tau }^2\\
 &\quad + \Norm{\Plv v_\l}_{+}^2
 +\Norm{\Prv v_r}_{+}^2
+\Norm{v_\l}^2_{m_\l,-1,\tau} 
 +\Norm{v_r}^2_{m_r,-1,\tau}\\
 &\quad +\norm{\trace(v_\l)}^2_{m_\l-1,-1/2,\tau}
 +\norm{\trace(v_r)}^2_{m_r-1,-1/2,\tau}.
\end{align*}

Writing $\Tlv^j \Op(\chi) = \Op(\chi)  \Tlv^j + [\Tlv^j , \Op(\chi) ]$
we observe that (using that $m_\l- \torder_\l^j= m_r- \torder_r^j= m- \torder^j$)
\begin{align*}
  \bignorm{\Tlv^j {\vv_\l}\br  + \Trv^j  {\vv_r}\br }_{m-1/2 - \torder^j,\tau }
  &\lesssim \bignorm{\Tlv^j {v_\l}\br  + \Trv^j  {v_r}\br }_{m-1/2 - \torder^j,\tau }\\
  &\quad + \norm{\trace(v_\l)}_{\torder_\l^j,m_\l-1/2-\torder_\l^j-1,\tau}
  + \norm{\trace(v_r)}_{\torder_r^j,m_r-1/2-\torder_r^j-1,\tau}\\
  &\lesssim \bignorm{\Tlv^j {v_\l}\br  + \Trv^j  {v_r}\br }_{m-1/2 - \torder^j,\tau }\\
  &\quad + \norm{\trace(v_\l)}_{m_\l-1,-1/2,\tau}
  + \norm{\trace(v_r)}_{m_r-1,-1/2,\tau}.
\end{align*}
This concludes the proof.
\end{proof}

\subsection{Estimate with a positive Poisson bracket on the characteristic set}
\label{sec: estimate sub-ellipticity}

If we consider the case of two symbols $a,b$ \st the Poisson bracket
$\{ a, b\}$ is positive on the characterisitic set $\{a=b=0\}$, an
estimate with the control of a volume norm can be achieved.
\begin{lemma}
  \label{lemma: positivity char}
  Let $U$ be an open set of $\ovl{V_+}$.
  Let $a\in S^{m,0}_{\tau}$ and $b\in S^{m-1,1}_{\tau}$ be real symbols
  homogeneous of degree $m$ in $(\tau,\xi)$, and set
  \begin{equation*}
    Q_{a,b}(v)=2 \Re\para{A v,iB v}_+, \quad A = a(x,D,\tau), \ B = b(x,D,\tau).
  \end{equation*}
  We assume that
  \begin{equation*}
    a(\y)=b(\y) =0  \ \ \imp \ \ \{ a,b\}>0,
    \quad \y= (x,\xi,\tau),
  \end{equation*}
  for $x \in \ovl{U}$, $(\xi,\tau) \neq (0,0)$. 
  Then there exist $C>0$, $C'>0$, and $\tau_\ast>0$ \st
  \begin{equation*}
    C \Norm{v}_{m,\tau}^2 \leq
    C' \big( \Norm{A v}_{+}^2 + \Norm{B v}_{+}^2 + \norm{\trace(v)}^2_{m-1,1/2,\tau}\big)
    + \tau \bigpara{Q_{a,b}(v)- \Re \B_{a,b}(v)},
  \end{equation*}
  for $\tau>\tau_\ast$ and for $v\in \Cinf(\Rpb)$ with $\supp(v)
  \subset U$, with $\B_{a,b}$ satisfying 
  \begin{equation*}
    \bignorm{\B_{a,b} (v)}\leq C '\bignorm{\trace(v)}^2_{m-1,1/2,\tau}.
  \end{equation*}
\end{lemma}
We refer to \cite{BLR:13} for a proof.

\subsection{A microlocal Carleman estimate}
\label{sec: microlocal Carleman}

With the results of Sections~\ref{sec: estimate with
  transmission condition} and \ref{sec: estimate sub-ellipticity},  if the transmission condition holds at one
point of the cotangent bundle at the interface and if the
sub-ellipticity property also holds we can then derive a
Carleman estimate that holds microlocally, that is, with a cut-off in
phase-space applied through a tangential pseudo-differential operator.
\begin{theorem}
  \label{theorem: microlocal Carleman}
  Let $x_0 \in S \cap V$. Assume that $\{P_\lr,\varphi_\lr\}$ satisfies
  the sub-ellipticity condition on a \nhd of $x_0$ in $\ovl{V_+}$.
  Assume moreover that $\{P_\lr,T_\lr^j,\varphi_\lr,\ j=1,\dots,m\}$ satisfies
  the transmission condition at $\y'_0=(x_0,\xi_0',\tau_0) \in
  \sphbundle(\ovl{V_+})$. Then
  there exists $\U$ a conic open \nhd of $\y'_0$ in $\ovl{V_+} \times
  \R^{n-1} \times \R_+$ \st for $\chi \in \Ssct^0$, homogeneous of
  degree 0,  with $\supp(\chi)
  \subset \U$, there exist $C>0$ and $\tau_\ast>0$ \st
  \begin{multline}
    \label{eq: microlocal Carleman}
    \Norm{\Plv v_\l}_{+}^2 + \Norm{\Prv v_r}_{+}^2 
    +\sum_{j=1}^m \bignorm{\Tlv^j {v_\l}\br + \Trv^j {v_r}\br}_{m-1/2-\torder^j,\tau}^2
    + \Norm{v_\l}_{m_\l,-1,\tau}^2 + \Norm{v_r}_{m_r,-1,\tau}^2\\
    +\norm{\trace(v_\l)}^2_{m_\l-1,-1/2,\tau}
    +\norm{\trace(v_r)}^2_{m_r-1,-1/2,\tau}
    \geq
    C \big( 
    \tau^{-1}\Norm{\Op(\chi)v_\l}_{m_\l,\tau}^2
    + \tau^{-1}\Norm{\Op(\chi)v_r}_{m_r,\tau}^2\\
    +\norm{\trace(\Op(\chi)v_\l)}_{m_\l-1,1/2,\tau}^2
    +\norm{\trace(\Op(\chi)v_r)}_{m_r-1,1/2,\tau}^2
    \big),
  \end{multline}
  for $\tau\geq\tau_\ast$, $v_\l, v_r\in \S(\Rpb)$.
\end{theorem}
Note that there are remainder terms, \viz
\begin{equation*}
  \bigNorm{v_\lr}_{m_\lr,-1,\tau}^2, \quad \norm{\trace(v_\lr)}^2_{m_\lr-1,-1/2,\tau}
\end{equation*}
that concern the unknown functions $v_\lr$ everywhere and not only in the
microlocal region $\U$ we consider here. The norms of these remainder
terms are weaker that those in the \rhs of the estimates.  When
patching microlocal estimates of the form of \eqref{eq: microlocal
  Carleman} together these remainder terms can be dealt with; see
Section~\ref{sec: proof of main theorem} below.

\begin{proof}
Let $U_0$ be a open \nhd of $x_0$ in $\ovl{V_+}$ with the
sub-ellipticity condition holding in $\ovl{U_0}$.

In the local coordinates we have chosen we have 
\begin{equation*}
  P_\lr = P_\lr(x,D)=\sum_{j=1}^m P_{\lr,j}(x,D')D_n^j,
\end{equation*}
with $P_{\lr,m} =1$ (see Section~\ref{sec: local setting}). We decompose the
conjugated operator $\Plrv = e^{\tau \varphi_\lr} P_\lr e^{-\tau \varphi_\lr}$
as
\begin{equation*}
\Plrv =P_{\lr,2} +i P_{\lr,1}, \quad P_{\lr,2} = \hf (\Plrv + \Plrv^\ast), 
\ \ P_{\lr,1} = \frac{1}{2i} (\Plrv - \Plrv^\ast).
\end{equation*}
The operators $P_{\lr,2}$ and $P_{\lr,1}$ are thus formally self-adjoint.
Their respective principal symbols $a_\lr(x,\xi,\tau)\in S^{m_\lr,0}_\tau$ and $b_\lr(x,\xi,\tau)\in
S^{m_\lr-1,1}_\tau$ are both real and homogeneous. We set 
$\plrv = a_\lr + i b_\lr$.
We then set
\begin{equation*}
  Q^\lr_{a,b}(v)=2\Re(A_\lr v_\lr,iB_\lr v_\lr)_+, \qquad A_\lr = \Op(a_\lr), \  B_\lr = \Op(b_\lr).
\end{equation*}
Note that we have 
\begin{equation}
  \label{eq: decomposition Pvarphi}
  \Plrv = A_\lr + i B_\lr + R_\lr, \qquad R_\lr \in \Psisc^{m_\lr,-1}.
\end{equation}
The sub-ellipticity condition of Definition~\ref{def: sub-ellipticity} reads
\begin{equation*}
  \plrv(x,\xi,\tau) =0 \ \ \imp \ \ \{ a_\lr, b_\lr\}(x,\xi,\tau)>0,
\end{equation*}
for $x \in U_0$ and $(\xi,\tau) \neq (0,0)$. Note that the case
$\tau=0$ is achieved because of the ellipticity of $P$ (see
Definition~\ref{def: sub-ellipticity} and Remark~\ref{remark: sub-ellipticity tau =0}).

Let now $\U$ be as given by Proposition~\ref{prop: transmission step},
possibly reduced so that $\U \subset U_0 \times \R^{n-1} \times \R_+$,
and let $\chi$ be as in the statement of the theorem.  By
Lemma~\ref{lemma: positivity char} we then have, for $\vv_\lr = \Op(\chi)
v_\lr$,
\begin{multline}
  \label{eq: Carleman 1}
     Q^\lr_{a,b}(\vv_\lr)-\Re  \B_{a_\lr,b_\lr} (\vv_\lr) 
     \geq  C \tau^{-1} \bigNorm{\vv_\lr}_{m_\lr,\tau}^2 \\
    - C' \tau^{-1}\big( \bigNorm{A_\lr \vv_\lr}_{+}^2 + \bigNorm{B_\lr \vv_\lr}_{+}^2
    + \bignorm{\trace(\vv_\lr)}^2_{m_\lr-1,1/2,\tau} \big),
  \end{multline}
with $\B_{a_\lr,b_\lr}$ satisfying 
\begin{equation*}
  \bignorm{\B_{a_\lr,b_\lr} (\vv_\lr)}\lesssim \bignorm{\trace(\vv_\lr)}^2_{m_\lr-1,1/2,\tau}.
\end{equation*}
With Proposition~\ref{prop: transmission step}, making use of the transmission
condition,  we obtain for $M$ chosen \suff large
\begin{align}
  \label{eq: Carleman 2}
  &\Re \B_{a_\l,b_\l}(\vv_\l)+\Re \B_{a_r,b_r}(\vv_r)
  +M \sum_{j=1}^m \bignorm{\Tlv^j {v_\l}\br + \Trv^j {v_r}\br}_{m-1/2-\torder^j,\tau}^2\\
  &\qquad \qquad\geq C
  \norm{\trace(\vv_\l)}_{m_\l-1,1/2,\tau}^2
  + \norm{\trace(\vv_r)}_{m_r-1,1/2,\tau}^2
  -C'\bigpara{ \Norm{v_\l}_{m_\l,-1,\tau}^2
    + \Norm{v_r}_{m_r,-1,\tau}^2\notag \\
    &\qquad \qquad\quad +\norm{\trace(v_\l)}^2_{m_\l-1,-1/2,\tau}
    +\norm{\trace(v_r)}^2_{m_r-1,-1/2,\tau}
    +\Norm{\Plv v_\l}_{+}^2
    +\Norm{\Prv v_r}_{+}^2}.\notag
\end{align}
Summing $\eqref{eq: Carleman 1}_\l$,  $\eqref{eq: Carleman 1}_r$,  and \eqref{eq: Carleman 2}  we find, by
taking $\tau$ \suff large,
\begin{multline*}
  \Norm{\Plv v_\l}_{+}^2 + \Norm{\Prv v_r}_{+}^2
  + \sum_{j=1}^m \bignorm{\Tlv^j {v_\l}\br + \Trv^j {v_r}\br}_{m-1/2-\torder^j,\tau}^2\\
  + Q^\l_{a,b}(\vv_\l) + Q^r_{a,b}(\vv_r) + \tau^{-1} (\Norm{A_\l \vv}_{+}^2 + \Norm{B_\l \vv_\l}_{+}^2
  +\Norm{A_r \vv_r}_{+}^2 + \Norm{B_r \vv_r}_{+}^2)
 \\
   + \Norm{v_\l}_{m_\l,-1,\tau}^2
  + \Norm{v_r}_{m_r,-1,\tau}^2
  +\norm{\trace(v_\l)}^2_{m_\l-1,-1/2,\tau}
  +\norm{\trace(v_r)}^2_{m_r-1,-1/2,\tau}\\
  \gtrsim
  \tau^{-1}\big( \Norm{\vv_\l}_{m_\l,\tau}^2 + \Norm{\vv_r}_{m_r,\tau}^2\big) 
  +\norm{\trace(\vv_\l)}_{m_\l-1,1/2,\tau}^2
  +\norm{\trace(\vv_r)}_{m_r-1,1/2,\tau}^2
  .
\end{multline*}
Finally, noting that 
\begin{align*}
  \tau^{-1}\big( \bigNorm{A_\lr \vv_\lr}_{+}^2 
  + \bigNorm{B_\lr \vv_\lr}_{+}^2 \big)
  + Q^\lr_{a,b}(\vv_\lr)
  &\leq \bigNorm{(A_\lr+i B_\lr) \vv_\lr}_{+}^2 \\
  &\lesssim \bigNorm{\Plrv \vv_\lr}_{+}^2 +
  \bigNorm{\vv_\lr}_{m_\lr,-1,\tau}^2\\
  &\lesssim \Norm{\Plrv v_\lr}_{+}^2 +
  \Norm{v_\lr}_{m_\lr,-1,\tau}^2,
\end{align*}
by \eqref{eq:
  decomposition Pvarphi} and pseudo-differential calculus (last point
of Proposition~\ref{prop: adjoint + composition S m,r}), we obtain the
sought microlocal estimate.
\end{proof}

\subsection{{Proof of Theorem~\ref{theorem: Carleman}}}
\label{sec: proof of main theorem}

We shall patch together estimates of the form given in
Theorem~\ref{theorem: microlocal Carleman}.

With $x_0$ as in the statement of Theorem~\ref{theorem: Carleman} the
transmission condition holds for all boundary quadruples
$\omega=(x_0,Y,\nu_0,\tau)$ with $Y \in T^*_{x_0}(\d\Omega)$, $\nu_0\in
N_{x_0}^*(S)$, and $\tau \geq 0$. In the local coordinates that we
use here this means that this property is satisfied for $\nu = d_{x_n}$ equal to
the (oriented) unit conormal to $\{x_n=0\}$ and all $\y' = (x_0, \xi',\tau)$ with
$\xi' \in \R^{n-1}$ and $\tau \geq 0$.  (See Section~\ref{sec: transmission condition in local coordinates}.)  It is fact sufficient to consider $(\xi',\tau)\in
\mathbb{S}^{n-1}_+=\set{(\xi',\tau)\in\R^n,\,\tau\geq
  0,\,\abs{(\xi',\tau)}=1}$.

By Theorem~\ref{theorem: microlocal Carleman} for all $(\xi'_0,\tau_0)
\in \mathbb{S}^{n-1}_+$ there exists a conic open\nhd $\U_{\y'_0}$ of $\y'_0=
(x_0, \xi'_0,\tau_0)$ in $\ovl{V_+} \times \R^{n-1}\times \R_+$ \st the estimate~\eqref{eq: microlocal
  Carleman}  holds. In fact by reducing $\U_{\y'_0}$ we can choose $\U_{\y'_0} =
\mathcal O_{\y'_0} \times \Gamma_{\y'_0}$ where $\mathcal O_{\y'_0}$ is an open set in
$\ovl{V_+}$ and $\Gamma_{\y'_0}$ is a conic open set in $\R^{n-1}\times \R_+$.
With the compactness of $\mathbb{S}^{n-1}_+$ we can thus find
finitely many such open sets
$\U_j = \mathcal O_j \times \Gamma_j$, $j \in J$, \st 
$ \mathbb{S}^{n-1}_+ \subset \cup_{j \in J} \Gamma_j $. 
We then set ${\mathcal O}  = \cap_{j \in J}  \mathcal O_j$ that is an open \nhd
of $x_0$ in $\ovl{V_+}$ and we set $\V_j = {\mathcal O} \times
\Gamma_j \subset \U_j$. 
We also choose an open \nhd $W$ of $x_0$ in $\R^n$ \st $W^+ = W \cap
\ovl{V_+} \Subset  {\mathcal O}$. 
 
We then choose a partition of unity,  $\chi_j \in \Ssct^0$, $j \in J$,  on $\ovl{W_+} \times
\R^{n-1}\times R_+$ subordinated by the covering by the open sets $\V_j$:
$$
\sum_{j \in J}\chi_j (\y')  = 1,  \  \text{for} \
\y'=(x,\xi',\tau)\in \ovl{W_+} \times
\R^{n-1}\times R_+\ \text{and}\ |(\xi',\tau)| \geq r_0>0, \qquad \supp(\chi_j) \subset
\V_j.
$$
The symbols $\chi_j$ are chosen homogeneous of degree $0$ for $
|(\xi',\tau)| \geq r_0>0$. We set $\udl{\chi} = 1 - \sum_{j \in J}\chi_j$
and have $\udl{\chi} \in \cap_{N \in \N} \Ssct^{-N}$. 

As $\supp(\chi_j) \subset \U_j$, we can apply the microlocal estimate of Theorem~\ref{theorem: microlocal
  Carleman}:
\begin{multline}
    \label{eq: microlocal Carleman-Uj}
    \Norm{\Plv v_\l}_{+}^2 + \Norm{\Prv v_r}_{+}^2 
    +\sum_{j=1}^m \bignorm{\Tlv^j {v_\l}\br + \Trv^j {v_r}\br}_{m-1/2-\torder^j,\tau}^2
    + \Norm{v_\l}_{m_\l,-1,\tau}^2 + \Norm{v_r}_{m_r,-1,\tau}^2\\
    +\norm{\trace(v_\l)}^2_{m_\l-1,-1/2,\tau}
    +\norm{\trace(v_r)}^2_{m_r-1,-1/2,\tau}
    \gtrsim
    \tau^{-1} \big(\Norm{v_\l}_{m_\l,\tau}^2
    + \Norm{v_r}_{m_r,\tau}^2\big)
    \\
    +\norm{\trace(\Op(\chi_j)v_\l)}_{m_\l-1,1/2,\tau}^2
    +\norm{\trace(\Op(\chi_j))v_r)}_{m_r-1,1/2,\tau}^2,
  \end{multline}
for $\tau$ chosen \suff large  and for $v_\lr = e^{\tau
  \varphi_\lr}u_\lr$ with $u_\lr = {w_\lr}_{|\Rpb}$ with $w_\lr\in
\Cinfc(W)$.
(see the statement of Theorem~\ref{theorem: Carleman}).

Observe then that, for any $N \in \N$,  
\begin{align*}
  &\Norm{v_\lr}_{m,\tau} 
  \leq \sum_{j \in J} \Norm{\Op(\chi_j)  v_\lr}_{m,\tau} 
  + \Norm{\Op(\udl{\chi})  v_\lr}_{m,\tau} 
  \lesssim \sum_{j \in J} \Norm{\Op(\chi_j)  v_\lr}_{m,\tau} 
  + \Norm{v_\lr}_{m,-N,\tau}, 
\end{align*}
and 
\begin{align*}
  \norm{\trace(v_\lr)}_{m-1,1/2,\tau}
  &\leq \sum_{j \in J} \norm{\trace(\Op(\chi_j)v_\lr)}_{m-1,1/2,\tau}
  + \norm{\trace(\Op(\udl{\chi})v_\lr)}_{m-1,1/2,\tau}\\
  &\lesssim \sum_{j \in J} \norm{\trace(\Op(\chi_j)v_\lr)}_{m-1,1/2,\tau}
  + \norm{\trace(v_\lr)}_{m-1,-N,\tau}.
\end{align*}
Summing estimates~\eqref{eq: microlocal Carleman-Uj} for each
$\chi_j$ we thus obtain
\begin{multline*}
    \Norm{\Plv v_\l}_{+}^2 + \Norm{\Prv v_r}_{+}^2 
    +\sum_{j=1}^m \bignorm{\Tlv^j {v_\l}\br + \Trv^j {v_r}\br}_{m-1/2-\torder^j,\tau}^2
    + \Norm{v_\l}_{m_\l,-1,\tau}^2 + \Norm{v_r}_{m_r,-1,\tau}^2\\
    +\norm{\trace(v_\l)}^2_{m_\l-1,-1/2,\tau}
    +\norm{\trace(v_r)}^2_{m_r-1,-1/2,\tau}
    \gtrsim
    \tau^{-1} \big(\Norm{v_\l}_{m_\l,\tau}^2
    + \Norm{v_r}_{m_r,\tau}^2\big)
    \\
    +\norm{\trace(v_\l)}_{m_\l-1,1/2,\tau}^2
    +\norm{\trace(v_r)}_{m_r-1,1/2,\tau}^2,
  \end{multline*}
Choosing now $\tau$ \suff large we obtain
\begin{multline*}
    \Norm{\Plv v_\l}_{+}^2 + \Norm{\Prv v_r}_{+}^2 
    +\sum_{j=1}^m \bignorm{\Tlv^j {v_\l}\br + \Trv^j {v_r}\br}_{m-1/2-\torder^j,\tau}^2
    \\
    \gtrsim
    \tau^{-1} \big(\Norm{v_\l}_{m_\l,\tau}^2
    + \Norm{v_r}_{m_r,\tau}^2\big)
    +\norm{\trace(v_\l)}_{m_\l-1,1/2,\tau}^2
    +\norm{\trace(v_r)}_{m_r-1,1/2,\tau}^2,
  \end{multline*}
Setting $v_\lr = e^{\tau \varphi_\lr} u_\lr$ the conclusion of the proof of
Theorem~\ref{theorem: Carleman} is then classical. \hfill
\qedsymbol \endproof

\subsection{Shifted estimates}
As in \cite{BLR:13} it may be interesting to consider shifted estimates in the Sobolev
scales. Namely we may wish to have an estimate of the following form.
\begin{corollary}
  \label{cor: Carleman shifted}
  Let $x_0 \in  S$ and let $\varphi\in \Con^0(\Omega)$ be \st
  $\varphi_k = \varphi_{|\Omega_k} \in \Cinf(\Omega_k)$ for $k=1,2$
  and \st 
  the pairs $\{P_k,\varphi_k\}$ have the sub-ellipticity property of
  Definition~\ref{def: sub-ellipticity} in a \nhd of $x_0$ in
  $\ovl{\Omega_k}$. Moreover, assume that $\big\{P_k,\varphi, T_k^j,\
    k=1,2, \ j=1,\dots,m\big\}$ satisfies the transmission condition at
  $x_0$.  Let $\ell \in \N$. Then there exist a \nhd $W$ of $x_0$ in $\R^n$ and two
  constants $C$ and $\tau_\ast>0$ \st
  \begin{multline}
    \label{eq: Carleman main result shifted}
    \sum_{k=1,2} \big( \tau^{-1}\Norm{e^{\tau\varphi_k}u_k}^{2}_{\ell+
      m_k,\tau}
    +\norm{e^{\tau\varphi_{|S}}\trace(u_k)}_{\ell+ m_k-1,1/2,\tau}^2 \big)\\
    \leq C \Big( \sum_{k=1,2}\Norm{e^{\tau\varphi_k}
        P_k(x,D)u_k}_{\ell, \tau}^2
      + \sum_{j=1}^m \big|e^{\tau\varphi_{|S}} 
      \big(T_1^j(x,D)u_1 +T_2^j(x,D)u_2\big)_{|S} \big|^2_{\ell, m-1/2-\torder^j,\tau}\Big),
  \end{multline}
  for all $u_k = {w_k}_{|\Omega_k}$ with $w_k\in \Cinfc(W)$ and $\tau\geq \tau_\ast$.
\end{corollary}
The proof of this corollary can be adapted from that of 
its counterpart at a boundary, namely Corollary 4.5 in \cite{BLR:13}.

\subsection{Interior-eigenvalue transmission problems}

Interior-eigenvalue transmission problems are very related to the
transmission problem we have considered. In fact, for $\Omega$, a
bounded open set in $\R^n$, we consider two elliptic operators $P_1$
and $P_2$ of respective orders $m_1$ and $m_2$,
as in Section~\ref{sec: introduction}, yet {\em both} defined on $\Omega$.

In addition, we consider $2 m = m_1+m_2$ boundary operators
operators 
\begin{equation}
  \label{eq: Definition Tkj-ietp}
  T_{k}^j=\sum_{\abs{\mi}\leq  \torder^j_{k}}
  t^j_{k,\mi}(x)D^\mi, 
  \quad k=1,2,\quad j=1,\dots,m,
\end{equation}
with $0\leq \torder^j_{k}< m_k$, and where the coefficients
$t^j_{k,\mi}(x)$ are $\Cinf$ complex-valued functions defined in some
neighborhood of $\d\Omega$. Setting $\torder^j = (\torder^j_1 + \torder^j_2)/2$ we assume that 
\begin{align}
  \label{eq: assumption beta-ietp}
  m_1 - \torder^j_1 = m_2 - \torder_2^j = m - \torder^j, \quad j=1, \dots, m. 
\end{align}
The interior-eigenvalue 
transmission problem consist in resolving a system of the form
\begin{equation*}
  \begin{cases}
    (P_k -\tau^{m_k} ) u_k = f_k &  \text{in} \ \Omega, \ , k=1,2\\
    T^j_1u_1+T_2^ju_2=g^j, &  \text{in} \ \Gamma,\quad j=1,\dots,m.
  \end{cases}
\end{equation*}
We refer to \cite{CPS:07,CCC:08,PS:08,CGH:10,CH:13,Robbiano:13,PV:16} and the reference therein for more details on this
very active field of research. 

In the analysis of such problems, resolvant estimates are central. In
the proof of such resolvent estimates, a  Carleman inequality at the boundary can be
a very efficient tool. Here, we provide such an estimate in a \nhd of a
point of $\d\Omega$, as the proof
is in fact given by the analysis of the previous section, in
particular, as we used the system formulation of Section~\ref{sec: system formulation}, which
yield a formulation close to that of the interior-eigenvalue
transmission problem.

Let $x_0 \in \d \Omega$ and $V$ be a \nhd  of $x_0$ where $\Omega  =
\{ x_n >0\}$. We consider two smooth weight functions $\varphi_1$ and
$\varphi_2$ in $V$ such that ${\varphi_1}\br  = {\varphi_2}\br$.

With the notation of Sections~\ref{sec: system formulation}--\ref{sec:
  symbol factorization}, with the letter $\ell$ replaced by
$1$ and the letter $r$ replaced by 2,  and moreover $P_\l$ (\resp
$P_r$) replaced by $P_1 - \tau^{m_1}$ (\resp   $P_2 - \tau^{m_2}$), 
we say that the transmission
condition holds at $x_0$  for $\{P_k - \tau^{m_k}, T_k^j,
\varphi_k,\ k=1, 2,\  j=1, \dots,
m\}$  if for all $\y' = (x_0, \xi', \tau)$  and 
for all pairs of polynomials,
  $q_1(\xi_n), q_2(\xi_n)$, there exist $U_1, U_2$,  polynomials, and
  $c_j\in\C$, $j=1,\dots,m$, such that
  \begin{align}
    \label{eq: local reformulation interior transmission 1}
    q_{1}(\xi_n)=\sum_{j=1}^{m}c_j\tlv^j(\y',\xi_n)
  +U_1(\xi_n)\kappa_1(\y',\xi_n),
 \intertext{and}
 \label{eq: local reformulation  interior transmission 2}
  q_2(\xi_n)=\sum_{j=1}^{m}c_j\trv^j(\y',\xi_n)
  +U_2(\xi_n)\kappa_2(\y',\xi_n).
  \end{align}

Then the proof of the following local Carleman estimate is the same as
that of Theorem~\ref{theorem: Carleman}. 
\begin{theorem}
  \label{theorem: Carleman-ietp}
  Let $x_0 \in  \d \Omega$ and let 
  $\varphi_k\in \Cinf(\Omega)$, $k=1,2$, as above, 
  \st the pairs $\{P_k- \tau^{m_k},\varphi_k\}$ satisfy the sub-ellipticity property of
  Definition~\ref{def: sub-ellipticity} in a \nhd of $x_0$ in
  $\ovl{\Omega}$. Moreover, assume that $\big\{P_k- \tau^{m_k}, T_k^j,
  \varphi_k,\ k=1,2, \ j=1,\dots,m\big\}$ satisfies the  transmission condition at
  $x_0$.  Then there exist a \nhd $W$ of $x_0$ in $\R^n$ and two
  constants $C$ and $\tau_\ast>0$ \st
  \begin{multline}
    \label{eq: Carleman main result-ietp}
    \sum_{k=1,2} \big( \tau^{-1}\Norm{e^{\tau\varphi_k}u_k}^{2}_{m_k,\tau}
    +\norm{e^{\tau\varphi_k}\gamma(u_k)}_{m_k-1,1/2,\tau}^2 \big)\\
    \leq C\bigpara{ \sum_{k=1,2}\Norm{e^{\tau\varphi_k}
        (P_k(x,D)- \tau^{m_k})u_k}_{L^2}^2
      + \sum_{j=1}^m |e^{\tau\varphi_{|\d \Omega}}  (T_1^j(x,D)u_1
      +T_2^j(x,D)u_2)_{|\d \Omega} |^2_{m-1/2-\torder^j,\tau}},
  \end{multline}
  for all $u_k = {w_k}_{|\Omega}$ with $w_k\in \Cinfc(W)$ and $\tau\geq \tau_\ast$.
\end{theorem}

\section{A pseudo-differential calculus with two large parameters}
\label{sec: pseudo -2p}
\setcounter{equation}{0}

The weight function we shall consider below is of the form $\varphi
(x)= \exp(\csp \psi(x))$.  The function $\psi$ is assumed to be
$\Con^0$, piecewise smooth, and to satisfy
\begin{equation*}
  0< C \leq \psi \quad \text{and}\  \bigNorm{\psi^{(k)}}_{L^\infty} <
  \infty, \ k \in \N.
\end{equation*}
We take $\csp\geq 1$. The goal of what follows is to achieve estimates as
in Theorem~\ref{theorem: Carleman} with the explicit dependency upon the additional
parameter $\csp$. This can be done by the introduction of an appropriate
pseudo-differential calculus. Assumption of the function
$\psi$ will be made in Section~\ref{sec: strong pseudo-convexity},
namely, the \spcty conditions, to obtain a Carleman estimate.

\subsection{Metric, symbols and Sobolev norms}
\label{sec: pseudo 2p}

Here, by $\y$ and $\y'$ we shall denote $\y = (x,\xi, \tau, \csp) \in \R^n\times
\R^{n}\times \R_+ \times\R_+$ and $\y' = (x,\xi', \tau, \csp) \in \R^n\times
\R^{n-1}\times \R_+ \times\R_+$.

We set $\ttau(x) = \tau \csp \varphi(x)$. Following \cite {LeRousseau:12} we consider the metrics on
phase-space
\begin{align*}
  g = \csp^2 |d x|^2 + \frac{|d \xi|^2}{\mu^2}, 
  \qquad \text{with} \ \ \mu^2 
  = \mu^2(\y)
  =|(\ttau(x),\xi)|^2 
  = \ttau(x) ^2 + |\xi|^2,
\end{align*}
and on 
tangent phase space
\begin{align*}
  \gt = \csp^2 |d x|^2 + \frac{|d \xi'|^2}{\mut^2}, 
  \qquad \text{with} \ \ \mut^2 = \mut^2(\y')
  =|(\ttau(x),\xi')|^2 
  =\ttau(x)^2 + |\xi'|^2,
\end{align*}
for $\tau \geq 1$ and $\csp \geq 1$.  Below, the explicit dependencies of
$\mu$ and $\mut$ upon $\y$ and $\y'$ are dropped to ease notation.

The metric $g$ (\resp $\gt$) along with the order function $\mu$
(\resp $\mut$) generates a (\resp tangential) Weyl-H\"ormander
pseudo-differential calculus as proven in \cite[Proposition
2.2]{LeRousseau:12}. Note that this uses the conditions $0 < C\leq \psi$ and
$\bigNorm{\psi'} < \infty$.

 For a presentation of the Weyl-H\"ormander
calculus we refer to \cite{Lerner:10}, \cite[Sections
18.4--6]{Hoermander:V3} and \cite{Hoermander:79}.

Let $a(x,\xi,\tau,\csp) \in \Cinf(\R^n\times\R^n)$, with $\tau,\csp$ as parameter
  in $[\tau_{\min},+\infty)$ and $[\csp_{\min},+\infty)$,
  $\tau_{\min}>0$, $\csp_{\min}>0$, and $m \in \R$, be \st for all multi-indices
  $\fmi, \smi \in \N^n$ we have
  \begin{equation}
      \label{eq: semi-classical symbols new}
    \abs{\d_x^\fmi \d_\xi^\smi a(\y)}
    \leq C_{\fmi,\smi} \csp^\fmi \mu^{m-\abs{\smi}},
      \quad \  \y \in \R^n \times \R^{n}\times   [\tau_{\min},+\infty)
      \times [\csp_{\min},+\infty).
  \end{equation}
  With the notation of \cite[Sections 18.4-18.6]{Hoermander:V3} we
  then have $a (\y) \in S(\mu^m, g)(\R^n \times \R^n)$, which we write $S^m (g)$ for
  simplicity\footnote{The dependence upon the metric $g$ is kept
    explicit here as we shall actually have to face two calculi
    simultaneously, associated with the weight functions on both sides
    of the interface. Interactions between the two calculi will only
    occur at the interface where they coincide. See Section~\ref{sec:
      two calculi}.}. The associated class of pseudo-differential
  operators is denoted by $ \Psi^m(g) = \Psi(\mu^m, g) (\R^n \times \R^n)$.

\medskip
Similarly we define tangential symbols and operators. Let $a(x,\xi',\tau,\csp) \in \Cinf(\Rpb\times\R^{n-1})$ and $m \in
\R$, be \st for all multi-indices $\fmi, \smi \in \N^n$ we have
  \begin{equation}
      \label{eq: semi-classical symbols tangential new}
    \abs{\d_x^\fmi \d_{\xi'}^\smi a(\y')}
    \leq C_{\fmi,\smi} \csp^\fmi \mut^{m-\abs{\smi}},
      \quad \y' \in \Rpb \times \R^{n-1}\times   [\tau_{\min},+\infty)
      \times [\csp_{\min},+\infty).
  \end{equation}
We then
have $a (\y') \in S^{m}(\gt)  = S(\mut^m, \gt)(\Rpb\times\R^{n-1})$. The associated class
of tangential pseudo-differential operators is denoted by $\Psi^m(\gt)
= \Psi(\mu^m, \gt)(\Rpb\times\R^{n-1})$.

Note that the condition $\bigNorm{\psi^{(k)}} < \infty$, $k \in \N$, is used to
prove\footnote{This condition was not written in \cite{LeRousseau:12}
  and \cite{BLR:13}. This is however made precise in \cite{LRR:15},
  including the proof of  $\ttau \in S(\ttau, g) \cap S(\ttau, \gt)$.} that $\ttau \in S(\ttau, g) \cap S(\ttau, \gt)$.

\medskip
With $\y = (x,\xi,\tau,\csp) \in \R^n \times \R^{n}\times
   \R_+\times \R_+$ (\resp $\y' = (x,\xi',\tau,\csp) \in \Rpb\times \R^{n-1}\times
   \R_+\times \R_+$) we shall associate $\ty=
   (x,\xi,\ttau(x)) \in \R^n\times \R^{n}\times \R_+$ (\resp $\ty' =
   (x,\xi',\ttau(x)) \in \Rpb\times \R^{n-1}\times \R_+$).

Note that if $\ha (x,\xi,\htau)\in \Ssc^{m}$, with
the notation of Section~\ref{sec: symbol classes}, satisfying moreover, 
for all multi-indices
  $\fmi, \smi',\smi'' \in \N^n$, with $\smi=\smi'+\smi''$,
\begin{equation}
      \label{eq: semi-classical symbols - additional derivative}
     \abs{\d_x^\fmi \d_\xi^{\smi'} \d_{\htau}^{\smi''} \ha(x,\xi,\htau)}
      \leq C_{\fmi,\smi',\smi''} \abs{(\xi,\htau)}^{m-\abs{\smi}},
      \quad x\in \R^n,\ \xi\in\R^n,\ \htau \in  [\tau_{\min},+\infty),
  \end{equation}
  \ie, differentiation \wrt $\htau$ yields the same additional decay
  as a differentiation \wrt $\xi$, then 
  \begin{equation*}
    a(x,\xi,\tau,\csp) = \ha(x,\xi,\ttau(x)) \in S^{m}(g),
  \end{equation*}
  which we shall write $a(\y) = \ha(\ty)$.  Similarly if $\ha
  (x,\xi',\htau) \in \Ssct^{m}$ with the same additional property
  regarding differentiation \wrt $\htau$ we have $a(\y') = \ha (\ty')
  \in S^{m}(\gt)$. In what follows we shall assume that symbols in $\Ssc^{m}$ and
$\Ssct^{m}$ have this additional regularity property. We then say that $a \in S^{m}(g)$ (\resp $S^{m}(\gt)$) is homogeneous
of degree $m$ with respect to $(\xi,\ttau)$ (\resp $(\xi',\ttau)$) if
we have $a(\y) = \ha (\ty))$ (\resp
$a(\y') = \ha (\ty')$) with
$\ha(x,\xi,\htau) \in \Ssc^{m}$ (\resp $\ha(x,\xi',\htau) \in
\Ssct^{m}$) homogeneous of degree $m$ in $(\xi,\htau)$ (\resp
$(\xi',\htau)$).

\bigskip
We shall also use the following classes of symbols $S(\ttau^r
\mut^m,\gt) = \ttau^r S^m(\gt)$ on $\Rpb\times\R^{n-1}$, for $r, m \in
\R$.  The associated class of tangential pseudo-differential operators
is denoted by $\ttau^r\Psi^m(\gt)  =\ttau^r\Psi(\mut^m,\gt) (\Rpb\times\R^{n-1})$. We have the following lemma
whose proof is similar to that of Lemma~2.7 in \cite{LeRousseau:12}.
\begin{lemma}
  \label{lem: regularity 2p}
  Let $r, m \in \R$ and $a \in \ttau^r   S^m(\gt)$. There exists
  $C>0$ \st for $\tau$ \suff large 
  \begin{align*}
    |\para{\Op(a) u,v}_{\d} \leq C
    \bigNorm{\Op(\ttau^{r'}\mut^{m'}) u}_{+}
    \bigNorm{\Op(\ttau^{r''}\mut^{m''}) v}_{+},
    \qquad u \in \S(\Rpb).
  \end{align*}
  for $r=r'+r''$, $m=m'+m''$.
\end{lemma}
This contains the estimate
\begin{align*}
    \bigNorm{\Op(\ttau^{s}\mut^{p}) \Op(a) u}_{+}
    \leq C
    \bigNorm{\Op(\ttau^{s+r}\mut^{p+m}) u}_{+}, 
    \qquad u\in \S(\Rpb), 
  \end{align*}
for $s, p\in \R$.
Note also that we have 
\begin{align}
  \label{eq: equivalence norms 2p}
    \bigNorm{\Op(\ttau^{r}\mut^{m}) u}_{+} \asymp
    \bigNorm{\Op(\mut^{m}) \ttau^{r} u}_{+} ,
\end{align}
for $\tau$ chosen \suff large.

\bigskip
Next we say that $a(x,\xi',\tau,\csp)\in
\ttau^r \Scl^m(\gt)$ on $\Rpb\times\R^{n-1}$ if there exists a sequence $a^{(j)} \in \csp^j \ttau^r  S^{m-j}(\gt)$, with $\csp^{-j} a^{(j)}$ homogeneous of degree $m+r-j$ in
$(\xi',\ttau)$ for $\abs{(\xi',\ttau)}\geq r_0$, with $r_0\geq 0$, \st
\begin{equation}
  a \sim \sum_{j\geq 0} a^{(j)},\quad  \text{in the sense
    that}\text\quad
  a - \sum_{j=0}^N a^{(j)} \in \csp^{N+1}  \ttau^r S^{m-N-1}(\gt).
\end{equation}
A representative of the principal part, denoted by $\sigma (a)$, is
then given by the first term in the expansion.
Then we shall say that $a(\y)\in \ttau^r
\Scl^{m,\sigma}(g)$  on $\Rpb\times\R^{n-1}$ if
\begin{equation*}
   a(\y)
   =\sum_{j=0}^m a_j(\y') \xi_n^j,
   \quad \text{with} \ a_j\in \ttau^r \Scl^{m-j+\sigma}(\gt).
\end{equation*}
The principal part is given by $\sum_{j=0}^m \sigma(a_j)(\y') \xi_n^j$.
With these symbol classes we associate classes of pseudo-differential
operators, $\ttau^r \Psicl^m(\gt)=\ttau^r \Psicl^m(\gt) (\Rpb\times\R^{n-1})$ and $\ttau^r \Psicl^{m,\sigma}(g)=\ttau^r \Psicl^{m,\sigma}(g) (\Rpb\times\R^{n-1})$,
as is done in Section~\ref{sec: PsiDO}.

\bigskip
We define the following semi-classical interior norm, for $m \in \N$, 
\begin{equation}
  \label{eq: norm sp 1}
  \Norm{u}^2_{m,\ttau}=\sum_{j=0}^m\bigNorm{\Op(\mut^{m-j}) D_n^j
    u}_{+}^2,
  \qquad u \in \S(\Rpb).
\end{equation}
We also set, for $m \in \N$ and $\sigma\in \R$, 
\begin{equation}
  \label{eq: norm sp 2}
  \Norm{u}^2_{m,\sigma,\ttau}=\bigNorm{\Op(\mut^\sigma)u}^2_{m,\ttau} 
  \sim \sum_{j=0}^m\bigNorm{\Op(\mut^{m-j+\sigma}) D_n^j
    u}_{+}^2,
  \qquad \ \ u \in \S(\Rpb).
\end{equation}

At the interface $\{x_n=0^+\}$ we define the following norms, for $m \in \N$ and $\sigma\in \R$, 
\begin{equation}
  \label{eq: norm sp 3}
\norm{\trace(u)}^2_{m,\sigma,\ttau}
=\sum_{j=0}^m \bignorm{\Op(\mut^{m-j+\sigma})\trace_j(u)}_{\d}^2,
\qquad \ \ u \in \S(\Rpb).
\end{equation}

\subsection{Transmission problem with two calculi}
\label{sec: two calculi}
In the present setting, using the system formulation of
Section~\ref{sec: system formulation} we shall in fact work in $\{ x_n
\geq 0\}$ with two weight functions, namely $\varphilr = e^{\csp
  \psilr}$. With each weight function we shall associate a
pseudo-differential calculus, classes of symbols and 
pseudo-differential operators, and Sobolev norms, as introduced in the
previous section.

We shall thus define $\ttaulr(x) = \tau \csp \varphilr(x)$,
\begin{align*}
  \mulr^2 = \ttaulr^2 + |\xi|^2, \qquad \mutlr^2 = \ttaulr^2 + |\xi|^2,
\end{align*}
the associated metrics
\begin{align*}
  \glr = \csp^2 |d x|^2 + \frac{|d \xi|^2}{\mulr ^2}, \qquad 
  \gtlr = \csp^2 |d x|^2 + \frac{|d \xi'|^2}{\mutlr^2}, 
\end{align*}
and the symbol classes $S^m(\glr)$, $S^m(\gtlr)$, $\Scl^m(\gtlr)$,
$\Scl^{m,\sigma}(\gtlr)$ and the associated operator
classes $\Psi^m(\glr)$, $\Psi^m(\gtlr)$, $\Psicl^m(\gtlr)$,
$\Psicl^{m,\sigma}(\gtlr)$. 

Accordingly for a function defined in $\{ x_n \geq 0\}$ we denote by 
$\Norm{u}_{m,\ttaulr}$ and $\Norm{u}_{m,\sigma,\ttaulr}$ the
associated norms as in \eqref{eq: norm sp 1}--\eqref{eq: norm sp 2}.

\medskip
Observe that the two calculi coincide at $x_n=0$, that is, on the
interface, since ${\psil}\br = {\psir}\br$, implying ${\varphil}\br =
{\varphir}\br$ and ${\mu_\l}\br = {\mu_r}\br$. In particular we shall
keep the notation 
\begin{equation*}
  \norm{\trace(u) }_{m,\sigma, \ttau} 
  = \norm{\trace(u) }_{m,\sigma, \ttaul} 
  = \norm{\trace(u) }_{m,\sigma,
  \ttaur}
\end{equation*}
 as in \eqref{eq: norm sp 3} for interface norms.

\subsection{Interface quadratic forms}
\label{sec: interface forms -2p}
\begin{definition}
Let $w=(w_\l,w_r)\in \big(\S (\Rpb)\big)^2$. We say that
\begin{equation*}
  \G (w)= \sum_{s=1}^N\para{A_\l ^s {w_\l}\br +A_r^s {w_r}\br, 
    B_\l^s {w_\l}\br+B_r^s{w_r}\br}_\d,
\end{equation*}
with  $A_\lr^s=a_\lr^s(x,D,\tau,\csp)$ and $B_\lr^s=b_\lr^s(x,D,\tau,\csp)$,
is an interface quadratic form of type $\para{m_\l-1,m_r-1,\sigma}$
with $\Cinf$ coefficients, if for each $s=1,\dots, N$, we have
$a_\lr^s(\y), b_\lr^s(\y)\in \Scl^{m_\lr-1,\sigma_\lr}(\gtlr)(\Rpb\times\R^n)$, with
$\sigma_\l+\sigma_r=2\sigma$, $\y=(\y',\xi_n)$ with
$\y'=(x,\xi',\tau,\csp)$.  

As in Section~\ref{sec: interface forms} we associate to $\G$ a bilinear symbol $\Sigma_{\G} (\y', \w,\tw)$.
\end{definition}

We let $\W$ be an open conic set in $\R^{n-1} \times \R^{n-1} \times \R_+$.
\begin{definition}
   Let $\G$ be an interface quadratic form of type
  $(m_\l-1,m_r-1,\sigma)$ associated with the bilinear symbol
  $\Sigma_{\G}(\y',\w,\tw)$. We say that $\G$ is positive
  definite in $\W$ if there exists $C>0$ and $R>0$ such that
  \begin{align*}
    \Re \Sigma_{\G}(\y'',x_n=0^+,\w,\w) 
    \geq C \Big(
      \sum_{j=0}^{m_\l-1}{\mutl}\br^{2(m_\l-1-j+\sigma_\l)}\bigabs{z^{\l}_j}^2
      +\sum_{j=0}^{m_r-1}{\mutr}\br^{2(m_r-1-j+\sigma_r)}\bigabs{z^{r}_j }^2
      \Big),
  \end{align*}
  for any $\w=(\zb^\l,\zb^r)$,
  $\zb^\lr=(z^{\lr}_{0},\dots,z^{\lr}_{m_\lr-1}) \in \C^{m_\lr}$, and
  $\ty''\in\W$, such that ${\mutl}\br = {\mutr}\br \geq R$,  with $\y''= (x',\xi',\tau,\csp)$
and $\ty'' = (x',\xi',\ttau(x',x_n=0^+))$.
\end{definition}

We have the following G{\aa}rding estimate.
\begin{lemma}
  \label{lemma: Gaarding for interface forms -2p}
  Let $\W$ be an open conic set in $\R^{n-1} \times \R^{n-1} \times
  \R_+$ and let $\G$ be an interface quadratic form of type
  $(0, m_\l-1,m_r-1,\sigma)$ that is positive definite in $\W$.
  Let $\hchi \in \Ssct^0$ be homogeneous of degree 0, with
  $\supp(\hchi\br)\subset \W$ and  let $N\in \N$. Then there exist $\tau_\ast\geq
  1$, $\csp_\ast \geq 1$, $C>0$, $C_N >0$ \st  
  \begin{align*}
    \Re \G (\Op(\chi) u) 
    &\geq  C \big(
    \norm{\trace(\Op(\chi_\l) u_\l)}^2_{m_\l-1,\sigma_\l,\ttau}
    +\norm{\trace(\Op(\chi_r) u_r)}^2_{m_r-1,\sigma_r,\ttau}
    \big)\\
    &\quad- C_N \big(
    \norm{\trace(u_\l)}^2_{m_\l-1,\sigma_\l-N,\ttau}
    +\norm{\trace(u_r)}^2_{m_r-1,\sigma_r-N,\ttau}
    \big)
    \end{align*}
    for $u=(u_\l,u_r)\in \big(\S(\Rpb)\big)^2$, $\tau\geq\tau_\ast$, $\csp \geq \csp_\ast$, and
  $\chi_\lr(\y') = \hchi(\ty'_\lr) \in \Ssctn^0$, with $\y'= (x,\xi',\tau,\csp)$
and $\ty'_\lr = (x,\xi',\ttau_\lr(x))$.
\end{lemma}
The proof is similar to that of Lemma~\ref{lemma: Gaarding for
  interface forms} using that the two calculi, associated with
$\psi_\l$ and $\psi_r$ respectively, coincide on the interface $S = \{
x_n =0\}$.  In particular note that ${\chi_\l}\br = {\chi_r}\br$ as
$\ttaul$ coincides with $\ttaur$ at the interface.
\section{Carleman estimate with two large parameters}
\label{sec: two parameters}
\setcounter{equation}{0}

With a weight function of the form $\varphi(x) = \exp(\csp \psi(x))$,
some condition on $\psi$ can yield $\varphi$ to fulfill the
sub-ellipticity condition of Definition~\ref{def: sub-ellipticity}.
Those are the strong pseudo-convexity conditions introduced by
L.~H\"ormander (see \cite{Hoermander:58}, \cite[Section
8.6]{Hoermander:63} and \cite[Section 28.3]{Hoermander:V4}). We shall
see that along with the transmission condition they are sufficient
to derive Carleman estimates with an explicit dependency upon the
additional parameter $\csp$. In fact the  strong pseudo-convexity
condition is also necessary if one considers a weight function of this
form; for such question we refer to \cite{LeRousseau:12}.

\subsection{Strong pseudo-convexity}
 \label{sec: strong pseudo-convexity}

 We recall the notion of \spcty and then adapt it to the geometry we
 consider.

As we restrict ourselves to elliptic operators in the
 present article,  the classical notion of \spcty then reduces to the following one (the reader can
 compare with Section 28.3 in \cite{Hoermander:V4}).
\begin{definition}[strong pseudo-convexity up to a boundary]
  \label{def:strong pseudo-convexity1}
  Let $\mathcal O$ be an open set. We say that  a smooth function $\psi$ is
  \spc at $x \in \ovl{\mathcal O}$ \wrt $p$ if $\psi'(x) \neq 0$ and if  
  for all 
  $\xi \in \R^n$ and $\htau > 0$,
  \begin{multline}
    \label{eq:strong pseudo-convexity}
    \tag{s-$\Psi$c}
    p(x,\xi + i \htau \psi'(x))=0\ \text{and}\ 
    \big\{p,\psi\big\}(x,\xi + i \htau \psi'(x))=0\\
    \ \ \imp \ \
    \frac{1}{2 i}\big\{\ovl{p}(x,\xi - i \htau \psi'(x)),
    p(x,\xi + i \htau \psi'(x))\big\}>0.
  \end{multline}
  Let $U$ be an open subset of $\mathcal O$.
  The function $\psi$ is said to be \spc \wrt $p$ in $U$
  {\em up to the boundary} if \eqref{eq:strong pseudo-convexity} is valid for all $x \in
  \ovl{U}$.
\end{definition}

\begin{definition}[strong pseudo-convexity  at an interface]
  \label{def:strong pseudo-convexity}
  Let $\Omega$, $\Omega_1$, $\Omega_2$, and $S$ be as in
  Section~\ref{sec: introduction}. 
  Let $\psi$ be a continuous function such that $\psi_k= \psi_{|\Omega_k}$
  are smooth for $k=1,2$.
  Let $U$ be an open subset of $\Omega$ that meets $S$.
The function $\psi$ is said to be \spc \wrt $P_1$ and $P_2$  in $U$ up
to the interface if both $\psi_k$, $k=1,2$, are \spc \wrt $P_k$ in $U_k = U \cap \Omega_k$
up to the boundary.

 Note in particular that  for $x \in S \cap U$
  \eqref{eq:strong pseudo-convexity} is required to hold for both
  $k=1$ and $k=2$.
\end{definition}

\subsection{Conjugated operators and transmission condition}
\label{sec: conj op transmission}

Here we use directly the notation introduced in Section~\ref{sec:
  system formulation} with the weight functions of the form
$\varphi_\lr = \exp(\csp \psi_\lr)$, which is sensible as the
transmission is a coordinate invariant property.

The principal symbol of $P_{\lr,\varphi} = e^{\tau \varphi_\lr} P_\lr e^{-\tau
  \varphi_\lr} \in \Psicl^{m,0}(\glr)$ in the present calculus is
\begin{equation*}
  p_{\lr,\varphi}(x,\xi,\tau) = p_\lr(x,\xi + i \tau \varphi_\lr'(x)) 
  = p_\lr(x,\xi + i \ttau_\lr(x) \psi_\lr'(x))
  = p_{\lr, \psi}(x,\xi,\ttau_\lr(x)) \in \Scl^{m,0}(\glr),
\end{equation*}
Similarly, the principal symbol of $T_{\lr,\varphi}^j = e^{\tau \varphi_\lr} T_\lr^j e^{-\tau
  \varphi_\lr} \in \Psicl^{\torder_k,0}(\glr)$, $j=1,\dots,m$, is 
\begin{align*}
 t_{\lr,\varphi}^j (x,\xi,\tau) = t_\lr^j(x,\xi + i \tau \varphi'(x))  
  =  t_\lr^j (x,\xi + i \ttau_\lr(x) \psi'(x))  
  = t_{\lr,\psi}^j  (x,\xi,\ttau_\lr(x)) \in \Scl^{\torder_k,0}(\glr).
\end{align*}
The dependency upon $\csp$ is hidden  either in $\varphi$ or
in $\ttau$. 

\medskip
Setting $\kappa_{\lr,\varphi} = p_{\lr,\varphi}^+  p_{\lr,\varphi}^0$
and  $\kappa_{\lr,\psi} = p_{\lr,\psi}^+  p_{\lr,\psi}^0$,
we then find 
\begin{equation*}
  \kappa_{\lr,\varphi}(x,\xi,\tau)= \kappa_{\lr,\psi}(x,\xi,\ttau_\lr(x)).
\end{equation*}

From these simple observations we thus conclude that
$\{P_\lr,T_\lr^j,\varphi,\ j = 1,\dots,m\}$ satisfies the
transmission condition at $(x_0, \xi_0', \tau_0)$, with $x_0 \in S$, 
{\em if and only if} $\{P_\lr,T_\lr^j,\psi,\ j = 1,\dots,m\}$ satisfies the
transmission condition at $(x_0, \xi_0', \ttau_0)$ with $\ttau_0 =
\ttaul(x_0) = \ttaur(x_0)$.

\subsection{Statement of the Carleman
estimate with two large parameters}

We shall prove the following theorem, counterpart of
Theorem~\ref{theorem: Carleman} in the case of a weight function of
the form $\varphi = \exp(\csp \psi)$, with an explicit dependency with
respect to the second large parameter $\csp$.  
\begin{theorem}
  \label{theorem: Carleman -2p}
  Let $x_0 \in  S$ and let $\psi\in \Con^0(\Omega)$ be \st
  $\psi_k = \psi_{|\Omega_k} \in \Cinf(\Omega_k)$ for $k=1,2$
  and \st $\psi$ has the \spcty property of Definition~\ref{def:strong pseudo-convexity}
  with respect to $P_1$ and $P_2$ in a \nhd of $x_0$ in $\Omega$. Moreover,
  assume that $\big\{P_k, T_k^j,\psi, 
    k=1,2, \ j=1,\dots,m\big\}$ satisfies the transmission condition at
  $x_0$.  Then there exist a \nhd $W$ of $x_0$ in $\R^n$ and three constants $C$,
  $\tau_\ast>0$, and $\csp_\ast>0$ \st for $\varphi_k = \exp(\csp \psi_k)$
  and $\ttau_k = \tau \csp \varphi_k$:
\begin{multline}
    \label{eq: Carleman main result -2p}
    \sum_{k=1,2} \big( \bigNorm{\ttau_k^{-1/2}  e^{\tau\varphi_k}u_k}^{2}_{m_k,\ttau_k}
    +\norm{e^{\tau\varphi_{|S}}\trace(u_k)}_{m_k-1,1/2,\ttau}^2 \big)\\
    \leq C\Big( \sum_{k=1,2}\Norm{e^{\tau\varphi_k}
        P_k(x,D)u_k}_{L^2(\Omega_k)}^2
      + \sum_{j=1}^m |e^{\tau\varphi_{|S}}  (T_1^j(x,D){u_1}
      +T_2^j(x,D){u_2})_{|S} |^2_{m-1/2-\torder^j,\ttau}
      \Big),
  \end{multline}
  for all $u_k = {w_k}_{|\Omega_k}$ with $w_k\in \Cinfc(W)$, $\tau\geq
  \tau_\ast$, and $\csp\geq \csp_\ast$.
\end{theorem}
Here norms of defined on $\Omega_k$ and $S$. They are locally equivalent to
their counterpart defined on $\{x_n>0\}$ and $\{x_n=0\}$ above. 

\subsection{Preliminary estimates}

The following lemma is the counterpart of Lemma~\ref{lemma: elliptic
  estimate}, that is, an elliptic estimate. It will be applied on both
the $\l$ and $r$ ``sides''. Hence, we formulate it for a weight
funtion $\varphi$, $\ttau$ and phase-space metric $g$  in place of
$\varphi_\lr$, $\ttau_\lr$, and $\glr$.

With $\y' = (x,\xi',\tau,\csp) \in \Rpb\times \R^{n-1}\times
   \R_+\times \R_+$ we shall associate $\ty'=
   (x,\xi',\ttau(x)) \in \Rpb\times \R^{n-1}\times \R_+$, with
   $\ttau(x) = \tau \csp \varphi(x)$. 
\begin{lemma}
   \label{lemma: elliptic estimate -2p}
   Let $h(\y)\in \Scl^{k,0}(g)$, with $\y = (x,\xi,\tau, \csp)$ and
   $k\geq 1$, be polynomial in $\xi_n$ with homogeneous coefficients
   in $(\xi',\ttau)$ and $H = h(x,D,\tau,\csp)$.  When viewed as a
   polynomial in $\xi_n$ the leading coefficient is $1$. Let $\U$ be a
   conic open subset of $\ovl{V_+} \times \R^{n-1}\times \R_+$.  We
   assume that all roots of $h(\y',\xi_n)=0$ have negative
   imaginary part for $\ty' \in \U$ .  Letting $\hchi(\hy') \in
   \Ssct^0$, $\hy'=(x,\xi',\htau)$, be homogeneous of degree 0 and  \st $\supp(\hchi) \subset
   \U$, and $N \in \N$, there exist $C>0$, $C_N>0$, $\tau_\ast>0$ and
   $\csp_\ast$, \st
  \begin{equation*}
    \Norm{\Op(\chi)w}^2_{k,\ttau}
    +\abs{\trace(\Op(\chi) w)}^2_{k-1,1/2,\ttau}
    \leq C
    \Norm{H \Op(\chi)w}_{+}^2
    +C_N\big( \Norm{w}^2_{k,-N,\ttau}
    + \norm{\trace(w)}^2_{k-1,-N,\ttau}\big),
\end{equation*}
for $w\in \S(\Rpb)$
and $\tau\geq \tau_\ast$, $\csp\geq \csp_\ast$ and $\chi(\y') =
\hchi(\ty') \in S^0(\gt)$.
\end{lemma}
We refer to~\cite{BLR:13} for a proof.

The following proposition is the counterpart of Proposition~\ref{prop: transmission step},
that is, an estimate exploiting the transmission condition,
yielding an estimate of an interface norm.
\begin{proposition}
  \label{prop: transmission step -2p}
  Assume that the transmission condition for $\big\{P_\lr,T^j_\lr,\psi_\lr,\ j=1,\dots,m\big\}$ is satisfied at
  $(x_0,\xi_0',\htau_0) \in \sphbundle(V)$ with $x_0\in S
  \cap V$.  Then there exists $\U$, a conic open \nhd of
  $(x_0,\xi_0',\htau_0)$ in $\ovl{V_+} \times \R^{n-1} \times \R_+$, \st for $\hchi \in
  \Ssct^0$, homogeneous of degree $0$,  with $\supp(\hchi) \subset
  \U$, and there exist
  $C>0$, $\tau_\ast>0$, and $\csp_\ast>0$ \st
  \begin{align*}
    &C \big(\norm{\trace(\Op(\chi_\l) v_\l)}^2_{m_\l-1,1/2,\ttau}
    + \norm{\trace(\Op(\chi_r) v_r)}^2_{m_r-1,1/2,\ttau}\big)
    \\
    &\qquad \qquad 
    \leq\sum_{j=1}^m \bignorm{\Tlv^j {v_\l}\br + \Trv^j {v_r}\br}_{m-1/2-\torder^j,\ttau}^2
    + \Norm{\Plv v_\l}_{+}^2 + \Norm{\Prv v_r}_{+}^2\\
     &\qquad \qquad \quad 
     + \csp^2\big(\Norm{v_\l}_{m_\l,-1,\ttaul}^2 + \Norm{v_r}_{m_r,-1,\ttaur}^2
    + \norm{\trace(v_\l)}^2_{m_\l-1,-1/2,\ttau}
    + \norm{\trace(v_r)}^2_{m_r-1,-1/2,\ttau}\big),
\end{align*}
for $\tau\geq\tau_\ast$, $\csp \geq \csp_\ast$, $v_\l, v_r\in \S(\Rpb)$
and $\chi_\lr(\y') = \hchi(\ty'_\lr) \in S^0(\gtlr)$, with $\y'= (x,\xi',\tau,\csp)$
and $\ty'_\lr = (x,\xi',\ttau_\lr(x))$.
\end{proposition}
\begin{proof}
  The beginning of the proof is nearly identical to that of
  Proposition~\ref{prop: transmission step}. In particular the \nhd
  $\U$ is chosen similarly.  Inequality~\eqref{eq: ineq symbol transmission} becomes
  \begin{multline*}
  \sum_{j=m+1}^{m_\l'} \lambdat^{2(m_\l -1/2 - \torder_\l^j)} 
  \bigabs{ \htchi(\hy') \Sigma_{ \elp^j}(\hy',\zb^{\l})}^2
  + \sum_{j=m+1}^{m_r'} \lambdat^{2(m_r -1/2 - \torder_r^j)} 
  \bigabs{ \htchi(\hy') \Sigma_{\erp^j}(\hy',\zb^{r})}^2\\
  +\sum_{j=1}^{m} \lambdat^{2(m - 1/2- \torder^j)} \bigabs{ \Sigma_{\tlp^j}(\hy', \zb^\l) 
    + \Sigma_{\trp^j}(\hy',\zb^r) }^2
  \geq C \big(\sum_{j=0}^{m_\l-1} \lambdat^{2(m_\l  -1/2-j)} \norm{z^\l_j}^2 
  + \sum_{j=0}^{m_r-1} \lambdat^{2(m_r  -1/2 -j)}  |z^r_j|^2\big), 
\end{multline*}
for $\zb^\lr\in \C^{m_\lr}$ and $\hy' = (x',x_n=0^+ ,\xi', \htau) \in
\ovl{\U} \cap \{ x_n=0\}$
with $\lambdat^2 = |\xi|'^2 +\htau^2$. Here, $\htchi \in \Ssct^0$ is
homogeneous of degree $0$ and 
\st $\htchi = 1$ in a \nhd of $\ovl{\U}$. We set $\tchi_\lr
(x,\xi',\tau,\csp) = \htchi(x,\xi', \ttaulr) \in S^0(\gtlr)$. 

We then obtain, taking $\hy'= \ty' = (x',x_n=0^+,\xi',\ttau(x))$,
 \begin{multline*}
  \sum_{j=m+1}^{m_\l'} {\mut}\br^{2(m_\l -1/2 - \torder_\l^j)} 
  \bigabs{ \tchi(\y'') \Sigma_{ \elv^j}(\y',\zb^{\l})}^2
  + \sum_{j=m+1}^{m_r'} {\mut}\br^{2(m_r -1/2 - \torder_r^j)} 
  \bigabs{ \tchi(\y'') \Sigma_{\erv^j}(\y',\zb^{r})}^2\\
  +\sum_{j=1}^{m} {\mut}\br^{2(m - 1/2- \torder^j)} \bigabs{ \Sigma_{\tlv^j}(\y', \zb^\l) 
    + \Sigma_{\trv^j}(\y',\zb^r) }^2
  \geq C \big(\sum_{j=0}^{m_\l-1} {\mut}\br^{2(m_\l  -1/2-j)} | z^\l_j|^2 
  + \sum_{j=0}^{m_r-1} {\mut}\br^{2(m_r  -1/2 -j)}  |z^r_j|^2\big), 
\end{multline*}
for all $\zb^\lr\in \C^{m_\lr}$ and $\y' = (x',x_n=0+,\xi',\tau,\csp)$
and $\y'' = (x',\xi',\tau,\csp)$ \st $\ty' \in \ovl{\U} \cap \{ x_n=0\}$.  Here $\tchi(\y'') =
\htchi(\ty'',x_n=0^+)$ with $\ty'' = (x',\xi', \ttau(x',x_n=0^+))$. 
We have $\tchi(\y'')  = \tchi_\l (\y')\br = \tchi_r (\y')\br$.
We set 
\begin{align*}
  \G_S (u) &=  \sum_{j=1}^{m} 
  \bignorm{\Tlv^j {u_\l}\br  + \Trv^j  {u_r}\br }_{m-1/2 - \torder^j,\tau }^2\\
  &\quad 
  + \sum_{j=m+1}^{m_\l'} \bignorm{\Elv^j {u_\l}\br}_{m+ m_\l^- +1/2 - j,\tau}^2
  + \sum_{j=m+1}^{m_r'} \bignorm{\Erv^j {u_\l}\br}_{m+ m_r^- +1/2 - j,\tau}^2.
\end{align*}
with $\Elrv =
\Op(\tchi_\lr \elrv)$.
 Then, according to  the G{\aa}rding inequality of
Lemma~\ref{lemma: Gaarding for interface forms -2p} for interface  quadratic forms
of type $(m_\l-1,m_r-1, 1/2)$, there exists $\tau_\ast>0$, $\csp_\ast>0$,
$C>0$, and $C_N >0$ \st
\begin{align}
    \label{eq: transmission 1 -2p}
    \G_S (\vv)  = \Re \G_S (\vv) 
    &\geq  C \big(
    \norm{\trace(\vv_\l)}^2_{m_\l-1,1/2,\ttau}
    +\norm{\trace(\vv_r)}^2_{m_r-1,1/2,\ttau}
    \big)\\
    &\quad- C_N \big(
    \norm{\trace(v_\l)}^2_{m_\l-1,1/2-N,\ttau}
    +\norm{\trace(v_r)}^2_{m_r-1,1/2-N,\ttau}
    \big), \notag
    \end{align}
    with $\vv = (\vv_\l, \vv_r)$ and $\vv_\lr = \Op(\chi_\lr) v_\lr$, for
    $\v=(v_\l,v_r)\in \big(\S(\Rpb)\big)^2$, $\tau\geq\tau_\ast$, and 
    $\csp\geq \csp_\ast$.

    \bigskip
    Now, arguing as in the proof of Proposition~\ref{prop:
      transmission step} we write $\chi_\lr \plrv = \chi_\lr \klrv \plrv^- = \chi_\lr
\tchi_\lr \klrv \tplrv^-$, where $\tplrv^-$ denotes an extension of  $\plrv^-$ to the
whole phase space. Then $$\Op(\chi_\lr)\Plrv =
\Op(\tplrv^-) \Op(\chi_\lr) \Op(\tchi_\lr \klrv) + R_\lr,$$ with $R_\lr$ in
$\csp \Psicl^{m,-1}(\glr)$. Applying Lemma~\ref{lemma: elliptic estimate -2p} to
$\Op(\tplrv^-)$ and $w_\lr = \Op(\tchi_\lr \klrv) v_\lr$ we find
\begin{align*}
  & \Norm{\Op(\chi_\l)  w_\l}_{m_\l^{-},\ttaul}^2 + \Norm{\Op(\chi_r)  w_r}_{m_r^{-},\ttaur}^2
  + \norm{\trace(\Op(\chi_\l) w_\l)}_{m_\l^- -1,1/2,\ttau}^2
  + \norm{\trace(\Op(\chi_r) w_r)}_{m_r^- -1,1/2,\ttau}^2\\
    &\qquad  \lesssim 
    \Norm{\Plv v_\l}_{+}^2
    + \Norm{\Prv v_r}_{+}^2
    + \csp^2 \big(\Norm{v_\l}^2_{m_\l,-1,\ttaul}
    + \Norm{v_r}^2_{m_r,-1,\ttaur}\big)
    \\
    &\qquad  \quad 
    + \norm{\trace(v_\l)}^2_{m_\l -1,-N,\ttau}
    + \norm{\trace(v_r)}^2_{m_r -1,-N,\ttau},
\end{align*}
yielding 
\begin{align*}
  &\sum_{j=0}^{m_\l^{-}-1}
  \bignorm{D_n^j \Op(\chi_\l) {w_\l}\br}_{m_\l^{-}-1/2-j,\ttau}^2
  + \sum_{j=0}^{m_r^{-}-1}
  \bignorm{D_n^j \Op(\chi_r) {w_r}\br }_{m_r^{-}-1/2-j,\ttau}^2\\
   &\qquad  \lesssim 
    \Norm{\Plv v_\l}_{+}^2
    + \Norm{\Prv v_r}_{+}^2
    + \csp^2 \big(\Norm{v_\l}^2_{m_\l,-1,\ttaul}
    + \Norm{v_r}^2_{m_r,-1,\ttaur}\big)
    \\
    &\qquad  \quad 
    + \norm{\trace(v_\l)}^2_{m_\l -1,-N,\ttau}
    + \norm{\trace(v_r)}^2_{m_r -1,-N,\ttau},
\end{align*}
Recalling that $\elrv^{j+m+1} = \klrv \xi_n^{j}$, $j=0,\dots,m_\lr^- -1$ in
a \nhd of $\U$ we have 
$$D_n^j \Op(\chi_\lr) \Op(\tchi_\lr \klrv) v_\lr =
\Elrv^{j+m+1} \vv_\lr + R_{\lr,j} v_\lr,$$ with $R_{\lr,j} \in \csp
\Psicl^{m_\lr-m_\lr^-+j,-1}(\glr)$. We thus obtain 
\begin{align}
  \label{eq: transmission 2 -2p}
  &\sum_{j=0}^{m_\l^{-}-1}
  \bignorm{{\Elv^{j+m+1} \vv_\l}\br}_{m_\l^{-}-1/2-j,\ttau}^2
  + \sum_{j=0}^{m_r^{-}-1}
  \bignorm{{\Erv^{j+m+1} \vv_r}\br }_{m_r^{-}-1/2-j,\ttau}^2\\
   &\qquad  \lesssim 
    \Norm{\Plv v_\l}_{+}^2
    + \Norm{\Prv v_r}_{+}^2
    + \csp^2 \Big( 
    \Norm{v_\l}^2_{m_\l,-1,\ttaul}
    + \Norm{v_r}^2_{m_r,-1,\ttaur}\big)\notag
    \\
    &\qquad  \quad 
    + \norm{\trace(v_\l)}^2_{m_\l -1,-1/2,\ttau}
    + \norm{\trace(v_r)}^2_{m_r -1,-1/2,\ttau}\Big).\notag
\end{align}
Collecting \eqref{eq: transmission 1 -2p} and \eqref{eq: transmission 2 -2p} we
obtain the result of Proposition~\ref{prop: transmission step -2p}, for $\tau$
and $\csp$ chosen \suff large, with an additional commutator argument
as in the end of the proof of Proposition~\ref{prop: transmission step}.
\end{proof}

\subsection{Proof of the Carleman estimate with two-large parameters}
\label{sec: proof Carleman -2p}
We prove a microlocal result, counterpart of that of
Theorem~\ref{theorem: microlocal Carleman}. Patching microlocal
estimates of this type, arguing as in Section~\ref{sec: proof of main
  theorem} we can then obtain the local Carleman estimate of
Theorem~\ref{theorem: Carleman -2p}. The proof is left to the
reader. 
\begin{theorem}
  \label{theorem: microlocal Carleman -2p}
  Let $x_0 \in S \cap V$ and let $\psi\in \Con^0(V)$ be \st $\psi_\lr
  \in \Cinf(\ovl{V^+})$  has the \spcty
  property of Definition~\ref{def:strong pseudo-convexity1} with
  respect to $P_\lr$ in a \nhd of $x_0$ in $\ovl{V_+}$.  Moreover,
  assume that $\big\{P_\lr,\psi_\lr, T_\lr^j,\ j=1,\dots,m\big\}$
  satisfies the transmission condition at $(x_0,\xi_0',\htau_0) \in
  \sphbundle(\ovl{V_+})$. Then there exists $\U$ a conic open \nhd of
  $(x_0,\xi_0',\htau_0)$ in $\ovl{V_+} \times \R^{n-1} \times \R_+$
  \st for $\hchi \in \Ssct^0$, homogeneous of degree $0$,  with $\supp(\hchi) \subset \U$, there
  exist $C>0$, $\tau_\ast>0$, and $\csp_\ast>0$ \st, for $\varphi_\lr =
  \exp(\csp \psi_\lr)$ and $\ttau_\lr = \tau \csp \varphi_\lr$,
\begin{multline}
    \label{eq: microlocal Carleman -2p}
    \Norm{\Plv v_\l}_{+}^2
    + \Norm{\Prv v_r}_{+}^2
    +\sum_{j=1}^m  \bignorm{ \Tlv^j  {v_\l}\br + \Trv^j {v_r}\br}_{m-\torder^j-1/2,\ttau}^2\\
    + \csp^2  \big( \Norm{v_\l}_{m_\l,-1,\ttaul}^2
    + \Norm{v_r}_{m_r,-1,\ttaur}^2
    +  \norm{\trace(v_\l)}_{m_\l-1,-1/2,\ttau}^2  
    +  \norm{\trace(v_r)}_{m_r-1,-1/2,\ttau}^2\big)\\
    \geq C \big(
    \bigNorm{\ttaul^{-\hf} \Op(\chi_\l) v_\l}_{m_\l,\ttaul}^2
    + \bigNorm{\ttaur^{-\hf} \Op(\chi_r) v_r}_{m_r,\ttaur}^2\\
    +\norm{\trace(\Op(\chi_\l) v_\l)}_{m_\l-1,1/2,\ttau}^2
    +\norm{\trace(\Op(\chi_r) v_r)}_{m_r-1,1/2,\ttau}^2\big),
  \end{multline}
  for all $v_\l,v_r \in \S(\Rpb)$, $\tau\geq
  \tau_\ast$, $\csp\geq \csp_\ast$, and 
 $\chi_\lr(\y') = \hchi(\ty'_\lr) \in S^0(\gtlr)$, with $\ty'_\lr =
 (x,\xi', \ttaulr(x))$ for $\y' =
 (x,\xi', \tau, \csp)$. 
\end{theorem}

\begin{proof}
  Applying Lemma~6.8 in \cite{BLR:13} we obtain that there exists $\U$
  a conic open \nhd of $(x_0,\xi_0',\htau_0)$ in $\ovl{V_+} \times
  \R^{n-1} \times \R_+$ \st for $\hchi \in \Ssct^0$, homogeneous of
  degree $0$,  with $\supp(\hchi)
  \subset \U$, there exist $C>0$, $\tau_0>0$, and $\csp_0>0$ \st
  \begin{multline}
    \label{eq: pre-carleman microlocal-2p}
    C \bigNorm{  \Plrv v_\lr}_{+}^2 - \Re\B_{a_\lr,b_\lr}( \Op(\chi_\lr) v_\lr)
    \geq C'\bigNorm{\ttaulr^{-\hf} \Op(\chi_\lr) v_\lr}_{m_\lr,\ttaulr}^2 \\
    - C''\big(
    \csp^2 \Norm{v_\lr}_{m_\lr,-1,\ttaulr}^2
    + \bignorm{\ttaulr^{-\hf} \trace(\Op(\chi_\lr)  v_\lr)}_{m_\lr-1,1/2,\ttaulr}^2
    + \csp \norm{\trace(\Op(\chi_\lr) v_\lr)}_{m_\lr-1,0,\ttaulr}^2
    \big),
  \end{multline}
  for $\tau\geq \tau_0$, $\csp\geq \csp_0$, and $\chi_\lr(\y') =
  \hchi(\ty'_\lr) \in S^0(\gtlr)$, where $\B_{a_\lr,b_\lr}$ satisfies
 \begin{equation*}
   \norm{\B_{a_\lr,b_\lr} (\Op(\chi_\lr) v_\lr)}\lesssim
   \norm{\trace(\Op(\chi_\lr) v_\lr)}^2_{m_\lr-1,1/2,\ttau}.
 \end{equation*}
 With Proposition~\ref{prop: transmission step -2p}, making use of the transmission condition, we  obtain
for $M$ chosen \suff large, in a possibly reduced \nhd $\U$, 
\begin{align}
  \label{eq: estimate with transmission condition -2p}
  &\Re \B_{a_\l,b_\l} (\Op(\chi_\l) v_\l) + \Re \B_{a_r,b_r} (\Op(\chi_r) v_r)
  +M \sum_{j=1}^m \bignorm{\Tlv^j {v_\l}\br + \Trv^j {v_r}\br}_{m-1/2-\torder^j,\ttau}^2\\
    &\qquad \geq C \big(\norm{\trace(\Op(\chi_\l) v_\l)}^2_{m_\l-1,1/2,\ttau}
    + \norm{\trace(\Op(\chi_r) v_r)}^2_{m_r-1,1/2,\ttau}\big)
    - C' \big(
     \Norm{\Plv v_\l}_{+}^2 + \Norm{\Prv v_r}_{+}^2  \notag\\
     &\qquad \quad 
     + \csp^2\big(\Norm{v_\l}_{m_\l,-1,\ttaul}^2 + \Norm{v_r}_{m_r,-1,\ttaur}^2
    + \norm{\trace(v_\l)}^2_{m_\l-1,-1/2,\ttau}
    + \norm{\trace(v_r)}^2_{m_r-1,-1/2,\ttau}\big), \notag
\end{align}
Summing $\eqref{eq: pre-carleman microlocal-2p}_\l$, $\eqref{eq:
  pre-carleman microlocal-2p}_r$, and \eqref{eq: estimate with
  transmission condition -2p} we obtain the result of
Theorem~\ref{theorem: microlocal Carleman -2p}, by taking $\tau$ and $\csp$
\suff large. 
\end{proof}

\subsection{Estimate with the simple characterisitic property}
As in \cite{LeRousseau:12} and \cite{BLR:13} a  stronger estimate with two parameters can be achieved if one assumes that the operator
$P$ and the weight function $\psi$ fulfills the so-called simple
characterisitic property.

We introduce the map 
\begin{align}
  \label{eq: simp char map}
  \begin{array}{rl}
  \rho_{x,\xi}: \R^+     &\to \C, \\
  \htau &\mapsto p(x,\xi + i \htau \psi'(x)), 
  \end{array}
\end{align}
where $x \in \ovl{\Omega}$ and $\xi \in \R^n$. 

\begin{definition}
  \label{def: simple characteristics}
  Let $U$ be an open subset of $\Omega$.
  Given a weight function $\psi$ and an operator $P$ we say that 
   the simple-characteristic property is satisfied in $\ovl{U}$ if, for
  all $x \in \ovl{U}$, we have $\xi = 0$ and $\htau=0$ when  the map
  $\rho_{x,\xi}$ has a double root. 
\end{definition}

\begin{remark}
  In fact the simple-characteristic property implies the property of
  \spcty.
We refer the reader to \cite{LeRousseau:12} and \cite{BLR:13}. 
\end{remark}

We have the following result.
\begin{theorem}
  \label{theorem: Carleman -simple char}
  Let $x_0 \in S$ and let $\psi\in \Con^0(\Omega)$ be \st $\psi_k =
  \psi_{|\Omega_k} \in \Cinf(\Omega_k)$ for $k=1,2$ and \st $\psi_k$
  and $P_k$ have the simple characteristic property of
  Definition~\ref{def: simple characteristics} in a \nhd of $x_0$ in
  $\ovl{\Omega_k}$. Moreover, assume that $\big\{P_k, T_k^j,\psi, \
  k=1,2, \ j=1,\dots,m\big\}$ satisfies the transmission condition
  at $x_0$.  Then there exist a \nhd $W$ of $x_0$ in $\R^n$ and three
  constants $C$, $\tau_\ast>0$, and $\csp_\ast>0$ \st, for $\varphi_k =
  \exp(\csp \psi_k)$ and $\ttau_k = \tau \csp \varphi_k$,
\begin{multline}
    \label{eq: Carleman main result -simple char}
    \sum_{k=1,2} \big(  \csp \bigNorm{\ttau_k^{-\hf}  e^{\tau\varphi_k}u_k}^{2}_{m_k,\ttau_k}
    +\norm{e^{\tau\varphi_{|S}} \trace(u_k)}_{m_k-1,1/2,\ttau}^2 \big)\\
    \leq C\bigpara{ \sum_{k=1,2}\Norm{e^{\tau\varphi_k}
        P_k(x,D)u_k}_{L^2(\Omega_k)}^2
      + \sum_{j=1}^m \bignorm{e^{\tau\varphi_{|S}}  (T_1^j(x,D){u_1} +T_2^j(x,D){u_2})_{|S}}^2_{m-1/2-\torder^j,\ttau}},
  \end{multline}
  for all $u_k = {w_k}_{|\Omega_k}$ with $w_k\in \Cinfc(W)$, $\tau\geq
  \tau_\ast$ and $\csp\geq \csp_\ast$.
\end{theorem}
We prove the following microlocal result and the result of
Theorem~\ref{theorem: Carleman -simple char} can be deduced, arguing
as in Section~\ref{sec: proof of main theorem}.

\begin{theorem}
  \label{theorem: microlocal Carleman -simple char}
  Let $x_0 \in S \cap V$ and let $\psi\in \Con^0(V)$ be \st $\psi_\lr
  \in \Cinf(\ovl{V^+})$  and \st $\psi_\lr$ and $P_\lr$ have the simple characteristic property of
  Definition~\ref{def: simple characteristics} in a \nhd of $x_0$ in
  $\ovl{V^+}$. Moreover,
  assume that $\big\{P_\lr, T_\lr^j, \psi_\lr,\ j=1,\dots,m\big\}$
  satisfies the transmission condition at $(x_0,\xi_0',\htau_0) \in
  \sphbundle(\ovl{V_+})$. Then there exists $\U$ a conic open \nhd of
  $(x_0,\xi_0',\htau_0)$ in $\ovl{V_+} \times \R^{n-1} \times \R_+$
  \st for $\hchi \in \Ssct^0$, homogeneous of degree $0$,  with $\supp(\hchi) \subset \U$, there
  exist $C>0$, $\tau_\ast>0$, and $\csp_\ast>0$ \st, for $\varphi_\lr =
  \exp(\csp \psi_\lr)$ and $\ttau_\lr = \tau \csp \varphi_\lr$,
  \begin{multline}
    \label{eq: microlocal Carleman -simple char}
    \Norm{\Plv v_\l}_{+}^2
    + \Norm{\Prv v_r}_{+}^2
    +\sum_{j=1}^m  \bignorm{ \Tlv^j  {v_\l}\br + \Trv^j {v_r}\br}_{m-\torder^j-1/2,\ttau}^2\\
    + \csp^2  \big( \Norm{v_\l}_{m_\l,-1,\ttaul}^2
    + \Norm{v_r}_{m_r,-1,\ttaur}^2
    +  \norm{\trace(v_\l)}_{m_\l-1,-1/2,\ttau}^2  
    +  \norm{\trace(v_r)}_{m_r-1,-1/2,\ttau}^2\big)\\
    \geq C \big(
    \csp \bigNorm{\ttaul^{-\hf} \Op(\chi_\l) v_\l}_{m_\l,\ttaul}^2
    + \csp \bigNorm{\ttaur^{-\hf} \Op(\chi_r) v_r}_{m_r,\ttaur}^2\\
    +\norm{\trace(\Op(\chi_\l) v_\l)}_{m_\l-1,1/2,\ttau}^2
    +\norm{\trace(\Op(\chi_r) v_r)}_{m_r-1,1/2,\ttau}^2\big),
  \end{multline}
  for all $v_\l,v_r \in \S(\Rpb)$, $\tau\geq \tau_\ast$, $\csp\geq
  \csp_\ast$, and $\chi_\lr(\y') = \hchi(\ty'_\lr) \in S^0(\gtlr)$, with
  $\ty'_\lr = (x,\xi', \ttaulr(x))$ for $\y' = (x,\xi', \tau,
  \csp)$.
\end{theorem}
\begin{proof}
  Applying Lemma~6.13 in \cite{BLR:13} 
  we obtain that there exists $\U$
  a conic open \nhd of $(x_0,\xi_0',\htau_0)$ in $\ovl{V_+} \times
  \R^{n-1} \times \R_+$ \st for $\hchi \in \Ssct^0$, homogeneous of
  degree $0$, with $\supp(\hchi)
  \subset \U$, there exist $C>0$, $\tau_0>0$, and $\csp_0>0$ \st
  \begin{multline}
    \label{eq: pre-carleman microlocal -simple char}
    C \bigNorm{  \Plrv v_\lr}_{+}^2 - \Re\B_{a_\lr,b_\lr}( \Op(\chi_\lr) v_\lr)
    \geq C' \csp \bigNorm{\ttaulr^{-\hf} \Op(\chi_\lr) v_\lr}_{m_\lr,\ttaulr}^2 \\
    - C''\big(
    \csp^2 \Norm{v_\lr}_{m_\lr,-1,\ttaulr}^2
    + \bignorm{\ttaulr^{-\hf} \trace(\Op(\chi_\lr)  v_\lr)}_{m_\lr-1,1/2,\ttau}^2
    + \csp \norm{\trace(\Op(\chi_\lr) v_\lr)}_{m_\lr-1,0,\ttau}^2
    \big),
  \end{multline}
  for $\tau\geq \tau_0$, $\csp\geq \csp_0$, and $\chi_\lr(\y') =
  \hchi(\ty'_\lr) \in S^0(\gtlr)$, where $\B_{a_\lr,b_\lr}$ satisfies
 \begin{equation*}
   \norm{\B_{a_\lr,b_\lr} (\Op(\chi_\lr) v_\lr)}\lesssim
   \norm{\trace(\Op(\chi_\lr) v_\lr)}^2_{m_\lr-1,1/2,\ttau}.
 \end{equation*}
 Summing $\eqref{eq: pre-carleman microlocal -simple char}_\l$, $\eqref{eq: pre-carleman microlocal -simple char}_r$, and \eqref{eq: estimate with
  transmission condition -2p} we obtain the result of
Theorem~\ref{theorem: microlocal Carleman -simple char} by taking $\tau$ and $\csp$
\suff large. 
\end{proof}

\subsection{Shifted estimates}
\label{sec: shifted estimates 2p}

As in \cite{BLR:13} it may be interesting to consider shifted estimates in the Sobolev
scales. Namely we may wish to have an estimate of the following form.
\begin{corollary}
  \label{cor: Carleman shifted-2p}
  Let $x_0 \in S$ and let $\psi\in \Con^0(\Omega)$ be \st $\psi_k =
  \psi_{|\Omega_k} \in \Cinf(\Omega_k)$ for $k=1,2$ and \st $\psi$ has
  the \spcty property of Definition~\ref{def:strong pseudo-convexity}
  with respect to $P_1$ and $P_2$ in a \nhd of $x_0$ in
  $\Omega$. Moreover, assume that $\big\{P_k, T_k^j,\psi, k=1,2, \
  j=1,\dots,m\big\}$ satisfies the transmission condition at $x_0$.
  Let $\ell \in \N$ and $s \in \R$. Then there exist a \nhd $W$ of
  $x_0$ in $\R^n$ and three constants $C$, $\tau_\ast>0$, and
  $\csp_\ast>0$ \st for $\varphi_k = \exp(\csp \psi_k)$ and $\ttau_k =
  \tau \csp \varphi_k$:
\begin{multline}
    \label{eq: Carleman main result shifted-2p}
    \sum_{k=1,2} \big( \bigNorm{\ttau_k^{s-1/2}  e^{\tau\varphi_k}u_k}^{2}_{\ell+m_k,\ttau_k}
    +\norm{\ttau_k^{s} e^{\tau\varphi_{|S}}\trace(u_k)}_{\ell+m_k-1,1/2,\ttau}^2 \big)\\
    \leq C\Big( \sum_{k=1,2}\Norm{\ttau_k^{s} e^{\tau\varphi_k}
        P_k(x,D)u_k}_{\ell, \ttau_k}^2
      + \sum_{j=1}^m |\ttau_k^{s} e^{\tau\varphi_{|S}}  (T_1^j(x,D){u_1}
      +T_2^j(x,D){u_2})_{|S}|^2_{\ell, m-1/2-\torder^j,\ttau}
      \Big),
  \end{multline}
  for all $u_k = {w_k}_{|\Omega_k}$ with $w_k\in \Cinfc(W)$, $\tau\geq
  \tau_\ast$, and $\csp\geq \csp_\ast$.
\end{corollary}

\begin{corollary}
  \label{cor: Carleman shifted -simple char}
  Let $x_0 \in S$ and let $\psi\in \Con^0(\Omega)$ be \st $\psi_k =
  \psi_{|\Omega_k} \in \Cinf(\Omega_k)$ for $k=1,2$ and \st $\psi_k$
  and $P_k$ have the simple characteristic property of
  Definition~\ref{def: simple characteristics} in a \nhd of $x_0$ in
  $\ovl{\Omega_k}$. Moreover, assume that $\big\{P_k, T_k^j,\psi, \
  k=1,2, \ j=1,\dots,m\big\}$ satisfies the transmission condition
  at $x_0$.   Let $\ell \in \N$ and $s \in \R$. Then there exist a \nhd $W$ of $x_0$ in $\R^n$ and three
  constants $C$, $\tau_\ast>0$, and $\csp_\ast>0$ \st, for $\varphi_k =
  \exp(\csp \psi_k)$ and $\ttau_k = \tau \csp \varphi_k$,
\begin{multline}
    \label{eq: Carleman main result shifted -simple char}
    \sum_{k=1,2} \big(  \csp \bigNorm{\ttau_k^{s-\hf}  e^{\tau\varphi_k}u_k}^{2}_{\ell+m_k,\ttau_k}
    +\norm{\ttau_k^{s} e^{\tau\varphi_{|S}}\trace(u_k)}_{\ell+m_k-1,1/2,\ttau}^2 \big)\\
    \leq C\bigpara{ \sum_{k=1,2}\Norm{\ttau_k^{s} e^{\tau\varphi_k}
        P_k(x,D)u_k}_{\ell,\ttau_k}^2
      + \sum_{j=1}^m \bignorm{\ttau_k^{s} e^{\tau\varphi_{|S}}  
        (T_1^j(x,D){u_1} +T_2^j(x,D){u_2})_{|S}}^2_{\ell, m-1/2-\torder^j,\ttau}},
  \end{multline}
  for all $u_k = {w_k}_{|\Omega_k}$ with $w_k\in \Cinfc(W)$, $\tau\geq
  \tau_\ast$ and $\csp\geq \csp_\ast$.
\end{corollary}
The proofs of these two corollaries can be adapted from the proofs of
their counterpart at a boundary, namely Corollaries~6.14 and 6.15 in \cite{BLR:13}.

\section{Application to unique continuation}
\label{sec: unique continuation}
With the Carleman estimates we have derived here we can obtain unique
continuation results near an interface for high-order elliptic
operators with the transmission condition, if we make a geometrical
assumption, namely the \spcty condition.
Result for a product of two operators can be obtained if additionaly the simple
characteristic propery holds for one  of them.

\subsection{Uniqueness under \spcty and transmission condition}

\begin{theorem}
  \label{theorem: unique continuation1}
  Let $P_k$ and $T^j_k$, $j=1, \dots, m$ be given as in Section~\ref{sec: introduction}.
  Let $x_0 \in S$,  $f \in \Con^0(\Omega)$, and $V$ be a
  \nhd of $x_0$, be such that $f$ has 
  the \spcty property of Definition~\ref{def:strong pseudo-convexity}
  with respect to $P_1$ and $P_2$ in $V$ . 
   Moreover,
  assume that $\big\{P_k,f, T^j_k,\ k=1,2, \ j=1,\dots,m\big\}$ satisfies
  the transmission condition at $x_0$. 
  Assume that $u$ is such that $u_k = u_{|\Omega_k} \in H^{m_k}(\Omega_k)$ and
  satisfies
  \begin{itemize}
  \item 
  \begin{align}
    \label{eq: unique continuation - inequality}
    |P_k u_k (x)| \leq C
    \sum_{|\mi| \leq m-1}
   |D^\mi u_k(x)|, \ \ 
    \ \text{a.e. in}\ V_k = V \cap \Omega_k, \ k=1,2;
  \end{align}
  \item for $j = 1 , \dots, m$ and $|\mi|\leq m - \torder^j$, with
    $\mi \in \N^{n-1}$, 
    \begin{align}
      \label{eq: unique continuation - inequality2}
      |D_\T^{\mi}  \big( T_1^j u_1 (x) + T_2^j u_2(x)\big)|  
      \leq C \sum_{k=1,2}\sum_{|\mi'| \leq |\mi| \atop+ \torder_k^j -1} |D^{\mi'} u_k(x)|, 
      \ \ \text{a.e. in} \ V \cap S;
    \end{align}
  \item 
    and $u$ vanishes in $\{ x \in V; \ f(x) \geq f(x_0)\}$.
  \end{itemize}
  Then $u$ vanishes in a \nhd of $x_0$.
\end{theorem}
Here $D_\T^\mi$ denotes a familly of differential operators that act
tangentially to the interface $S$ and, in local coordinates near
$x_0$, where
$S= \{ x_n=0\}$,  their principal symbol is  $\xi'^\mi$. 

\begin{proof}
  Strong pseudo-convexity is a stable notion in $\Con^2$ (see
  Proposition~28.3.2 in \cite{Hoermander:V4}). Here the function $f$
  is continuous and  piecewise smooth. The argument of
  \cite{Hoermander:V4} applies on both sides of the interface.   For $\eps$ chosen \suff
  small, there exists a \nhd $V'$ of $x_0$ such that the function
  $\psi(x) = f(x) - \eps |x-x_0|^2$ has the \spcty property of
  Definition~\ref{def:strong pseudo-convexity} with respect to $P_1$
  and $P_2$  in
  $V'$.  Similary we saw in Section~\ref{sec: transmission condition in local coordinates} for the
  proof of Proposition~\ref{prop: stability transmission}, that the
  transmission condition  (or rather property~\eqref{eq: rank formulation
  transmission condition} ) is robust upon perturbation of
  the weight function.  Hence if $\eps$ is chosen
  \suff small $\set{P,\psi, B^k,\ k=1,\dots,\mu}$ will also satisfy
  this condition.

  We set $\varphi = \exp (\csp \psi)$. As shown in Proposition~28.3.3 in \cite{Hoermander:V4}
  the \spcty of the function $\psi$ with respect to $P_1$ and $P_2$ 
  implies the sub-ellipticity condition for $\{ P_k, \varphi_k\}$ for $\csp
 $ chosen \suff large for both $k=1, 2$ with $\varphi_k =
 \varphi_{|\Omega_k}$. 
  Moreover, as seen in Section~\ref{sec: conj op transmission}
  $\set{P_k,T^j_k,\varphi; \ k=1,2, \ j = 1,\dots,\mu}$ also satisfies the
  transmission  condition at $x_0$. 
 
The geometrical situation we describre is illustrated in
 Figures~\ref{fig:geometry} and ~\ref{fig:geometry2}. The two figures
 show different unique continuation configuration: across an
 hypersuface that is not related to the interface $S$, or across the
 interface $S$.
  We call $W$ the region $\{ x \in V;\ f(x) \geq f(x_0)\}$ (region
beneath $\{f(x)= f(x_0)\}$ in Figure~\ref{fig:geometry}) where $u$
vanishes by assumption.
  We choose $V''$ a \nhd of $x_0$ such that  $V'' \Subset V'$. 

  \begin{figure}
  \begin{center}
   \begin{picture}(0,0)%
\includegraphics{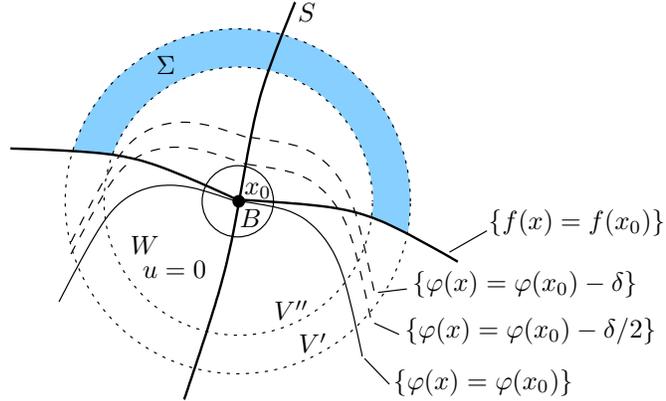}%
\end{picture}%
\setlength{\unitlength}{3947sp}%
\begingroup\makeatletter\ifx\SetFigFont\undefined%
\gdef\SetFigFont#1#2#3#4#5{%
  \reset@font\fontsize{#1}{#2pt}%
  \fontfamily{#3}\fontseries{#4}\fontshape{#5}%
  \selectfont}%
\fi\endgroup%
\begin{picture}(3040,2542)(-4376,-2047)
\put(-2699,-1524){\makebox(0,0)[lb]{\smash{{\SetFigFont{10}{12.0}{\rmdefault}{\mddefault}{\updefault}{\color[rgb]{0,0,0}$V''$}%
}}}}
\put(-2545,-1730){\makebox(0,0)[lb]{\smash{{\SetFigFont{10}{12.0}{\rmdefault}{\mddefault}{\updefault}{\color[rgb]{0,0,0}$V'$}%
}}}}
\put(-3437, 22){\makebox(0,0)[lb]{\smash{{\SetFigFont{10}{12.0}{\rmdefault}{\mddefault}{\updefault}{\color[rgb]{0,0,0}$\Sigma$}%
}}}}
\put(-1828,-1341){\makebox(0,0)[lb]{\smash{{\SetFigFont{10}{12.0}{\rmdefault}{\mddefault}{\updefault}{\color[rgb]{0,0,0}$\{\varphi(x) = \varphi(x_0)-\de\}$}%
}}}}
\put(-1877,-1636){\makebox(0,0)[lb]{\smash{{\SetFigFont{10}{12.0}{\rmdefault}{\mddefault}{\updefault}{\color[rgb]{0,0,0}$\{\varphi(x) = \varphi(x_0)-\de/2\}$}%
}}}}
\put(-1948,-1957){\makebox(0,0)[lb]{\smash{{\SetFigFont{10}{12.0}{\rmdefault}{\mddefault}{\updefault}{\color[rgb]{0,0,0}$\{\varphi(x) = \varphi(x_0)\}$}%
}}}}
\put(-1351,-953){\makebox(0,0)[lb]{\smash{{\SetFigFont{10}{12.0}{\rmdefault}{\mddefault}{\updefault}{\color[rgb]{0,0,0}$\{f(x) = f(x_0)\}$}%
}}}}
\put(-2914,-938){\makebox(0,0)[lb]{\smash{{\SetFigFont{10}{12.0}{\rmdefault}{\mddefault}{\updefault}{\color[rgb]{0,0,0}$B$}%
}}}}
\put(-3599,-1111){\makebox(0,0)[lb]{\smash{{\SetFigFont{10}{12.0}{\rmdefault}{\mddefault}{\updefault}{\color[rgb]{0,0,0}$W$}%
}}}}
\put(-3524,-1261){\makebox(0,0)[lb]{\smash{{\SetFigFont{10}{12.0}{\rmdefault}{\mddefault}{\updefault}{\color[rgb]{0,0,0}$u=0$}%
}}}}
\put(-2549,355){\makebox(0,0)[lb]{\smash{{\SetFigFont{10}{12.0}{\rmdefault}{\mddefault}{\updefault}{\color[rgb]{0,0,0}$S$}%
}}}}
\put(-2884,-736){\makebox(0,0)[lb]{\smash{{\SetFigFont{10}{12.0}{\rmdefault}{\mddefault}{\updefault}{\color[rgb]{0,0,0}$x_0$}%
}}}}
\end{picture}%
 
    \caption{Local geometry for the unique continuation problem.
      The shaded region $\Sigma$ contains the  supports of $[P,\chi]
      u$ and $[\chi,D_{x_k}^j] u$.
    In the figure the funtions $f$ and $\varphi$ are continuous and yet
    only piecewise smooth. Here
      unique continuation is performed across a surface that is not
      related to $S$.}
  \label{fig:geometry}
  \end{center}
\end{figure}

 \begin{figure}
  \begin{center}
    \begin{picture}(0,0)%
\includegraphics{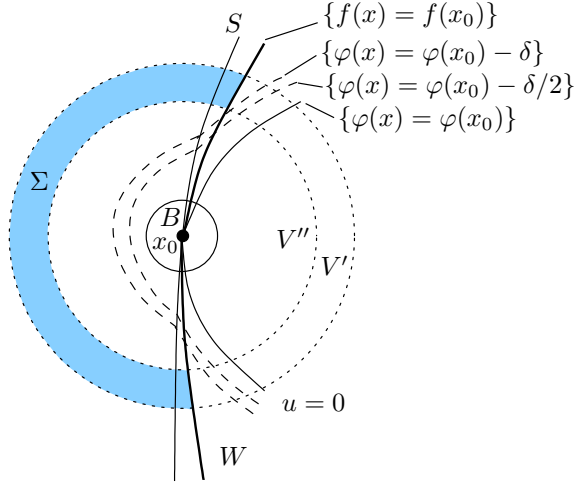}%
\end{picture}%
\setlength{\unitlength}{3947sp}%
\begingroup\makeatletter\ifx\SetFigFont\undefined%
\gdef\SetFigFont#1#2#3#4#5{%
  \reset@font\fontsize{#1}{#2pt}%
  \fontfamily{#3}\fontseries{#4}\fontshape{#5}%
  \selectfont}%
\fi\endgroup%
\begin{picture}(2184,3046)(2132,-2335)
\put(3494,503){\makebox(0,0)[lb]{\smash{{\SetFigFont{10}{12.0}{\rmdefault}{\mddefault}{\updefault}{\color[rgb]{0,0,0}$S$}%
}}}}
\put(3852,-1879){\makebox(0,0)[lb]{\smash{{\SetFigFont{10}{12.0}{\rmdefault}{\mddefault}{\updefault}{\color[rgb]{0,0,0}$u=0$}%
}}}}
\put(3458,-2219){\makebox(0,0)[lb]{\smash{{\SetFigFont{10}{12.0}{\rmdefault}{\mddefault}{\updefault}{\color[rgb]{0,0,0}$W$}%
}}}}
\put(4105,564){\makebox(0,0)[lb]{\smash{{\SetFigFont{10}{12.0}{\rmdefault}{\mddefault}{\updefault}{\color[rgb]{0,0,0}$\{f(x) = f(x_0)\}$}%
}}}}
\put(4204,-91){\makebox(0,0)[lb]{\smash{{\SetFigFont{10}{12.0}{\rmdefault}{\mddefault}{\updefault}{\color[rgb]{0,0,0}$\{\varphi(x) = \varphi(x_0)\}$}%
}}}}
\put(4118,119){\makebox(0,0)[lb]{\smash{{\SetFigFont{10}{12.0}{\rmdefault}{\mddefault}{\updefault}{\color[rgb]{0,0,0}$\{\varphi(x) = \varphi(x_0)-\de/2\}$}%
}}}}
\put(3085,-718){\makebox(0,0)[lb]{\smash{{\SetFigFont{10}{12.0}{\rmdefault}{\mddefault}{\updefault}{\color[rgb]{0,0,0}$B$}%
}}}}
\put(2268,-493){\makebox(0,0)[lb]{\smash{{\SetFigFont{10}{12.0}{\rmdefault}{\mddefault}{\updefault}{\color[rgb]{0,0,0}$\Sigma$}%
}}}}
\put(4093,-1026){\makebox(0,0)[lb]{\smash{{\SetFigFont{10}{12.0}{\rmdefault}{\mddefault}{\updefault}{\color[rgb]{0,0,0}$V'$}%
}}}}
\put(3815,-854){\makebox(0,0)[lb]{\smash{{\SetFigFont{10}{12.0}{\rmdefault}{\mddefault}{\updefault}{\color[rgb]{0,0,0}$V''$}%
}}}}
\put(3036,-862){\makebox(0,0)[lb]{\smash{{\SetFigFont{10}{12.0}{\rmdefault}{\mddefault}{\updefault}{\color[rgb]{0,0,0}$x_0$}%
}}}}
\put(4094,324){\makebox(0,0)[lb]{\smash{{\SetFigFont{10}{12.0}{\rmdefault}{\mddefault}{\updefault}{\color[rgb]{0,0,0}$\{\varphi(x) = \varphi(x_0)-\de\}$}%
}}}}
\end{picture}%

    \caption{Local geometry for the unique continuation problem. Here,
      unique continuation is performed across the interface $S$.}
  \label{fig:geometry2}
  \end{center}
\end{figure}

We pick a function $\chi \in \Cinfc(\R^n)$ such that $\chi
=1$ in $V''$ and $\supp (\chi) \cap V \subset V'$. 
We observe that the Carleman estimate of 
Theorem~\ref{theorem: Carleman} applies to  $\chi u$
by density (possibly by reducing the \nhds $V$ and $V'$ of $x_0$):
 \begin{multline}
   \label{eq: Carleman UC}
    \sum_{k=1,2} \big( \tau^{-1/2}\Norm{e^{\tau\varphi_k} \chi u_k}_{m_k,\tau}
    +\norm{e^{\tau\varphi_{|S}}\trace( \chi u_k)}_{m_k-1,1/2,\tau} \big)\\
    \lesssim \sum_{k=1,2}\Norm{e^{\tau\varphi_k}
        P_k (\chi u_k)}_{L^2(\Omega_k)}
      + \sum_{j=1}^m \bignorm{e^{\tau\varphi_{|S}} 
      \big(T_1^j (\chi) u_1 +T_2^j (\chi u_2)\big)_{|S}}_{m-1/2-\torder^j,\tau},
  \end{multline}
for $\tau\geq \tau_0$.  

 We have
$P_k(\chi u_k) = \chi  P_k u_k + [P_k,\chi] u_k$, 
where the commutator is a differential operator of order $m-1$.
 With \eqref{eq: unique continuation - inequality} we have
\begin{align*}
  \Norm{e^{\tau\varphi_k}P_k (\chi u_k)}_{L^2(\Omega_k)}
  &\lesssim \sum_{|\mi|\leq m-1} 
   \Norm{e^{\tau\varphi_k} \chi D^\mi u_k}_{L^2(\Omega_k)}
   + \Norm{e^{\tau\varphi_k} [P_k,\chi]u_k}_{L^2(\Omega_k)}\\
   &
   \lesssim \sum_{|\mi|\leq m-1} 
   \Norm{e^{\tau\varphi_k} D^\mi (\chi u_k)}_{L^2(\Omega_k)}
   + \sum_{i \in I_k}\Norm{e^{\tau\varphi_k} M_{k,i} u_k}_{L^2(\Omega_k)},
\end{align*}
where $I_k$ is finite and the operators $M_{k,i}$ are commutators fo $\chi$
and differential operators. They are of order $m_k-1$ at most.

We also write
\begin{align*}
  &\bignorm{e^{\tau\varphi_{|S}} 
    \big(T_1^j(\chi u_1)  +T_2^j (\chi u_2)\big)_{|S}}_{m-1/2-\torder^j,\tau}\\
  &\qquad\leq 
  \bignorm{e^{\tau\varphi_{|S}} 
    \big(T_1^j (\chi u_1) +T_2^j (\chi u_2)\big)_{|S}}_{m-\torder^j,\tau}\\
  &\qquad = 
  \sum_{r + |\mi|\atop \leq m-\torder^j}
  \bignorm{\tau^r e^{\tau\varphi_{|S}} 
    D_\T^{\mi} \big(T_1^j (\chi u_1) +T_2^j (\chi
    u_2)\big)_{|S}}_{L^2(S)}
\end{align*}
We write $D_\T^{\mi}  T_k^j (\chi u_k) = \chi D_\T^{\mi}  T_k^j u_k  +
[D_\T^{\mi}  T_k^j, \chi] u_k$ and have
\begin{align*}
  &\bignorm{e^{\tau\varphi_{|S}} 
    \big(T_1^j(\chi u_1)  +T_2^j (\chi u_2)\big)_{|S}}_{m-1/2-\torder^j,\tau}\\
  &\qquad \lesssim
  \sum_{r + |\mi|\atop \leq m-\torder^j}
  \Big( \bignorm{\tau^r  e^{\tau\varphi_{|S}} 
    \chi  \big( D_\T^{\mi}T_1^j u_1+D_\T^{\mi}T_2^j 
    u_2\big)_{|S}}_{L^2(S)}\\
  &\qquad\qquad 
  + 
  \bignorm{\tau^r  e^{\tau\varphi_{|S}} 
    \big( [D_\T^{\mi}T_1^j , \chi] u_1+ [D_\T^{\mi} T_2^j , \chi]
    u_2\big)_{|S}}_{L^2(S)}\Big) \\
   &\qquad \lesssim
 \sum_{k=1,2} \sum_{r + |\mi|\atop \leq m-\torder^j}
  \sum_{|\mi'| \leq |\mi| \atop +\torder^j_k-1 }  
  \bignorm{\tau^r e^{\tau\varphi_{|S}} 
    \chi  (D^{\mi'}u_k)_{|S}}_{L^2(S)}\\
  &\qquad\qquad 
  +\sum_{k=1,2}  \sum_{r + |\mi|\atop \leq m-\torder^j}
  \bignorm{\tau^r  e^{\tau\varphi_{|S}} 
    \big( [D_\T^{\mi}T_k^j , \chi] u_k\big)_{|S}}_{L^2(S)}, 
\end{align*}
Using commutators once more we write
\begin{align*}
  \bignorm{\tau^r  e^{\tau\varphi_{|S}} 
    \chi  (D^{\mi'}u_k)_{|S}}_{L^2(S)}
  \leq \bignorm{
    \tau^r D^{\mi'}( e^{\tau\varphi_{|S}} \chi  u_k)_{|S}}_{L^2(S)}
  + \bignorm{
    \tau^r ([e^{\tau\varphi_{|S}} \chi,  D^{\mi'}] u_k)_{|S}}_{L^2(S)}. 
\end{align*}
We thus find
\begin{align*}
  &\sum_{j=1}^m \bignorm{e^{\tau\varphi_{|S}} 
    \big(T_1^j(\chi u_1)  +T_2^j (\chi u_2)\big)_{|S}}_{m-1/2-\torder^j,\tau}\\
  &\qquad \lesssim
  \sum_{j=1}^m\sum_{k=1,2} \sum_{r + |\mi|\atop \leq m-\torder^j}
  \sum_{|\mi'| \leq |\mi| \atop +\torder^j_k-1 } \bignorm{
    \tau^r \big(D^{\mi'}( e^{\tau\varphi_{|S}} \chi  u_k)\big)_{|S}}_{L^2(S)}
  + \sum_{k=1,2} \sum_{i \in J_{k}}
  \bignorm{(\tilde{M}_{k,i} u_k)_{|S}}_{L^2(S)}\\
   &\qquad \lesssim
  \sum_{k=1,2} \bignorm{
    \trace( e^{\tau\varphi} \chi  u_k)}_{m_k-1,0,\tau}
  + \sum_{k=1,2} \sum_{i \in J_{k}}
  \bignorm{\tilde{M}_{k,i} u_k}_{L^2(S)},
\end{align*}
where  $J_{k}$ is finite and $\tilde{M}_{k,i}$ is the commutator of $\chi$ with a
differential operator. The operator   $\tilde{M}_{k,i}$ is of order at most $m_k-1$
(\wrt to $\tau$ and $\xi$).

As 
\begin{align*}
  \norm{e^{\tau\varphi_{|S}}\trace(\chi u)}_{m_k-1,1/2,\tau}
  \geq \tau^\hf \norm{e^{\tau\varphi_{|S}}\trace(\chi u_k)}_{m_k-1,0,\tau}
  \geq \tau^\hf \norm{\trace(e^{\tau\varphi_k}\chi u_k)}_{m_k-1,0,\tau},
\end{align*}
for $\tau$ chosen \suff large, from \eqref{eq: Carleman UC} we thus obtain 
\begin{multline*}
    \sum_{k=1,2} \big( \tau^{-1/2}\Norm{e^{\tau\varphi_k} \chi u_k}_{m_k,\tau}
    +\norm{e^{\tau\varphi_k}\trace( \chi u_k)}_{m_k-1,1/2,\tau} \big)\\
    \lesssim \sum_{k=1,2}\Big( \sum_{i \in I_k}\bigNorm{e^{\tau\varphi_k}
        M_{k,i}u_k}_{L^2(\Omega_k)}
      +\sum_{i \in J_k}  \bignorm{(\tilde{M}_{k,i}
        u_k)_{|S}}_{L^2(S)}\Big),
  \end{multline*}

We set $\Sigma:= V'\setminus(V''\cup W)$ (see the shaded region in
Figure~\ref{fig:geometry}). We have 
\begin{align*}
  \supp( M_{k,i} u_k) \subset \Sigma, \ \ i \in I_k\quad \text{and}\ \ 
  \supp(\tilde{M}_{k,j} u_k) \subset \Sigma, \ \ i \in J_k
 ,
\end{align*}
 as they are confined in the region where
$\chi$ varies and $u$ does not vanish.

We thus obtain
\begin{multline*}
    \sum_{k=1,2} \big( \tau^{-1/2}\Norm{e^{\tau\varphi_k} \chi u_k}_{m_k,\tau}
    +\norm{e^{\tau\varphi_k}\trace( \chi u_k)}_{m_k-1,1/2,\tau} \big)\\
    \lesssim \sum_{k=1,2} \Big( \sum_{|\mi| \leq m_k-1} \bigNorm{e^{\tau\varphi_k}
        D^\mi u_k}_{L^2(\Sigma)}
      +\sum_{r + |\mi| \leq m_k-1}   
      \bignorm{\tau^r (
        D^\mi u_k)_{|S}}_{L^2(\Sigma \cap S)}\Big).
  \end{multline*}

 For all $\de >0$, we set
$V_\de = \{ x \in
V; \ \varphi(x) \leq \varphi(x_0) - \de\}$.
There exists $\de>0$ \st $\Sigma \Subset V_\de$.  We then choose
$B$ a \nhd  of  $x_0$ \st $\ovl{B} \subset V''
\setminus V_{\de/2}$ and obtain, as $\chi \equiv 1$ on $B$, 
\begin{align*}
    e^{\tau \inf_B \varphi}\sum_{k=1,2} \Norm{ u_k}_{H^m_k(B)}
    \lesssim e^{\tau (\sup_{\Sigma}  \varphi+\de/2)}  \sum_{k=1,2}
    \Big( \Norm{ u_k}_{H^m_k(\Sigma)}
    + \sum_{|\mi|\leq m_k-1}  \norm{D^\mi  u_{|S}}_{L^2(\Sigma \cap S}
    \Big),\quad \tau \geq \tau_1
  \end{align*}
for some $\tau_1>0$.
Since $\inf_{B} \varphi > \sup_{\Sigma} \varphi+\de/2$, letting $\tau$ go to
$\infty$, we obtain $u=0$ in $B$. 
\end{proof}

\subsection{Uniqueness for products of operators}
We now consider two sets of elliptic operators:
  $P_1$  and $Q_1$ defined on $\Omega_1$ and $P_2$  and $Q_2$ defined
  on $\Omega_2$.
We denote their respective orders by $m^p_1$, $m^q_1$, $m^p_2$, and
$m^q_2$.  We assume that $m^p_1 = m^p_2 = m^p$. 

We also consider interface operators $T^{p,j}_k$, $k=1,2$, $j=1,
\dots, m^p$ of order $\torder^{p,j}_k$ and  $T^{q,j}_k$, $k=1,2$, $j=1,
\dots, m^q$ of order $\torder^{q,j}_k$  with $m^q = (m^q_1 +
m^q_2)/2$.
We assume that $m^p - \torder^{p,j}_1 = m^p - \torder^{p,j}_2 = m^p -
\torder^{p,j}$
and $m^q_1 - \torder^{q,j}_1 = m^q_2 - \torder^{q,j}_2 = m^q -
\torder^{q,j}$.
 Then the operators $P_1$, $P_2$, $T^{p,j}_k$, $k=1,2$, $j=1,
\dots, m^p$, allow one to define an elliptic transmission problem as presented
in Section~\ref{sec: introduction}. The same is valid for 
$Q_1$, $Q_2$, $T^{q,j}_k$, $k=1,2$, $j=1,
\dots, m^q$.

We observe that $P_1Q_1$, $P_2 Q_2$, and the intreface operators $T^{p,j}_kQ_k$, $k=1,2$, $j=1,
\dots, m^p$, and $T^{q,j}_k$, $k=1,2$, $j=1,
\dots, m^q$ also allow to define an elliptic  transmission problem. In
fact, 
the operators $P_1Q_1$ and $P_2Q_2$  are of respective order $m_1 = m^p +
m^q_1$ and $m_2 = m^p +
m^q_2$. We set $m = m^p + m^q = (m_1 + m_2)/2$. 
The interface operator $T^{q,j}_k$ is of order $\torder^{q,j}_k$ and
we have $m_1 - \torder^{q,j}_1 = m_2 - \torder^{q,j}_2 = m -
\torder^{q,j}$. The interface operator  $T^{p,j}_kQ_k$ is of order
$\torder^{p,j}_k + m^q_k$ and
we have $m_1 - (\torder^{p,j}_1+m^q_1)  = m_2 - (\torder^{p,j}_2 + m^q_2)$.

One may possibly
wonder about unique continuation for this product transmission
problem, in particuler in the case  no Carleman estimate of the type
derived here can be achieved.
Let us however  assume that for a function $\psi$ and the weigth function
$\varphi = \exp(\csp \psi)$ we can derive Carleman estimates for the
transmission problems associated with $P_1, P_2$ and $Q_1, Q_2$.
More precisely we assume that the first problem satisfies the simple
characteristic property while the second one only satisfies the \spcty condition.
\begin{theorem}
  \label{theorem: unique continuation2}
  Let $P_k$, $T^{p,j}_k$, $k=1,2$, $j=1, \dots, m^p$, 
  and $Q_k$,   $T^{q,j}_k$, $k=1,2$, $j=1, \dots, m^q$, be given as above.
  Let $x_0 \in S$,  $f \in \Con^0(\Omega)$, with $f_k = f_{|\Omega_k}$
  and $V$ a \nhd of $x_0$, be such that 
  \begin{enumerate}
  \item $f_k$ and $P_k$ have the simple characteristic property of
  Definition~\ref{def: simple characteristics} in $V \cap
  \ovl{\Omega_k}$;
  \item 
  $f$ has  the \spcty property of Definition~\ref{def:strong pseudo-convexity}
  with respect to $Q_1$ and $Q_2$ in $V$;
  \item 
    $\big\{P_k,f, T^{p,j}_k,\ k=1,2, \ j=1,\dots,m^p\big\}$ satisfies
  the transmission condition at $x_0$;
  \item 
    $\big\{Q_k,f, T^{q,j}_k,\ k=1,2, \ j=1,\dots,m^q\big\}$ satisfies
  the transmission condition at $x_0$. 
  \end{enumerate}

  Assume that $u$ is such that $u_k = u_{|\Omega_k} \in H^{m^p+m^q_k}(\Omega_k)$ and
  satisfies
  \begin{itemize}
  \item 
  \begin{align}
    \label{eq: unique continuation - inequality3}
    |P_k Q_k u_k (x)| \leq C
    \sum_{|\mi| \leq m^p+m^q_k-1}
   |D^\mi u_k(x)|, \ \ 
    \ \text{a.e. in}\ V_k = V \cap \Omega_k,\ k=1,2;
  \end{align}
  \item for $j = 1 , \dots, m^p$, $|\mi|\leq m^p- \torder^{p,j}$, with
      $\mi \in \N^{n-1}$, 
    \begin{align}
      \label{eq: unique continuation - inequality4}
      |D_\T^{\mi}  \big( T_1^{p,j} Q_1 u_1 (x) + T_2^{p,j} Q_2 u_2(x)\big)|  
      \leq C \sum_{k=1,2}\sum_{|\mi'| \leq |\mi| + m^q_k \atop+ \torder_k^{p,j} -1} |D^{\mi'} u_k(x)|, 
      \ \ \text{a.e. in} \ V \cap S;
    \end{align}
  \item for $j = 1 , \dots, m^q$, $|\mi_1|\leq m^p$,  $|\mi_2|\leq m^q - \torder^{q,j}$, with
     $\mi_1 \in \N^{n}$ and  $\mi_2 \in \N^{n-1}$, 
    \begin{align}
      \label{eq: unique continuation - inequality5}
      |D^{\mi_1}D_\T^{\mi_2}  \big( T_1^{q,j} u_1 (x) + T_2^{q,j} u_2(x)\big)|  
      \leq C \sum_{k=1,2}\sum_{|\mi'| \leq |\mi_1| + |\mi_2| \atop+ \torder_k^{q,j} -1} |D^{\mi'} u_k(x)|, 
      \ \ \text{a.e. in} \ V \cap S;
    \end{align}
  \item 
    and $u$ vanishes in $\{ x \in V; \ f(x) \geq f(x_0)\}$.
  \end{itemize}
  Then $u$ vanishes in a \nhd of $x_0$.
\end{theorem}
Here $D_\T^\mi$ denotes a familly of differential operators that act
tangentially to the interface $S$ and, in local coordinates near
$x_0$, where
$S= \{ x_n=0\}$,  their principal symbol is  $\xi'^\mi$. 

\begin{remark}
Here we have assume that $m^p_1  = m^p_2$. It would be interesting to
know if such an assumption can be removed.  This assumption is
connected to the shifted estimate of Corollary~\ref{cor: Carleman
  shifted-2p}, where the shift is the same on both sides of the
interface.
Having different Sobolev-scale shifts from one side to the other leads
to technical difficulties with the transmission terms on the
interface. It is not clear whether such estimates can be achieved
without modifying the properties of the transmission operators.  
\end{remark}
\begin{proof}
  The proof follows that of Theorem~\ref{theorem: unique
    continuation1}.  We set $\psi(x) = f(x) - \eps |x-x_0|^2$ and
  conditions (1)-(4) in the statement of the theorem are also
  satisfied by $\psi$ for $\eps$ chosen \suff small in a \nhd
  $V'\subset V$ of $x_0$.
  We then set $\varphi = \exp (\csp \psi)$.
   
  We derive an estimate for $P_kQ_k$. We first write an estimate for
  $P_k$. By Theorem~\ref{theorem: Carleman -simple char}, there exists
  $V_1$ \nhd of $x_0$ in $\R^n$ such that $V_1\subset V'$ and 
  \begin{multline*}
    \sum_{k=1,2} \big(  \csp^{1/2} \bigNorm{\ttau_k^{-\hf}  e^{\tau\varphi_k}v_k}_{m^p,\ttau_k}
    +\norm{e^{\tau\varphi_{|S}}\trace(v_k)}_{m^p-1,1/2,\ttau} \big)\\
    \lesssim \sum_{k=1,2}\Norm{e^{\tau\varphi_k}
        P_kv_k}_{L^2(\Omega_k)}
      + \sum_{j=1}^{m^p} \bignorm{e^{\tau\varphi_{|S}} 
        (T_1^{p,j}{v_1} +T_2^{p,j}{v_2})_{|S}}_{m^p-1/2-\torder^{p,j},\ttau},
  \end{multline*}
  for all $v_k = {w_k}_{|\Omega_k}$ with $w_k\in \Cinfc(V_1)$, $\tau\geq
  \tau_1$ and $\csp\geq \csp_1$, for $\tau_1$ and $\csp_1$
  chosen \suff large.
  
  For $Q_k$, by Corollary~\ref{cor: Carleman shifted-2p}, there
  exists $V_2$ \nhd of $x_0$ in $\R^n$ such that $V_2 \subset V'$ and 
  \begin{multline*}
    \sum_{k=1,2} \big( \bigNorm{\ttau_k^{-1}  e^{\tau\varphi_k}v_k}_{m^p+m^q_k,\ttau_k}
    +\bignorm{\ttau^{-1/2} e^{\tau\varphi_{|S}}\trace(v_k)}_{m^p+m^q_k-1,1/2,\ttau} \big)\\
    \lesssim \sum_{k=1,2}\bigNorm{\ttau_k^{-1/2} e^{\tau\varphi_k}
        Q_kv_k}_{m^p, \ttau_k}
      + \sum_{j=1}^{m^q} \bignorm{\ttau^{-1/2} e^{\tau\varphi_{|S}} 
        (T_1^{q,j}{v_1}+T_2^{q,j}{v_2})_{|S}}_{m^p, m^q-1/2-\torder^{q,j},\ttau}
      ,
  \end{multline*}
   for all $v_k = {w_k}_{|\Omega_k}$ with $w_k\in \Cinfc(V_2)$, $\tau\geq
  \tau_2$ and $\csp\geq \csp_2$, for $\tau_2$ and $\csp_2$
  chosen \suff large.

  Letting $V_3 = V_1 \cap V_2$ with the previous two estimates we
  obtain
  \begin{multline}
    \label{eq: Carleman-uc}
    \csp^{1/2} \sum_{k=1,2} \big(  \bigNorm{\ttau_k^{-1}
      e^{\tau\varphi_k}v_k}_{m^p+ m^q_k,\ttau_k}
    +\bignorm{\ttau^{-1/2}  e^{\tau\varphi_{|S}}\trace(v_k)}_{m^p+m^q_k-1,1/2,\ttau} \big)\\
    \lesssim \sum_{k=1,2}\Norm{e^{\tau\varphi_k}
        P_kQ_kv_k}_{L^2(\Omega_k)}
      + \sum_{j=1}^{m^p} \bignorm{e^{\tau\varphi_{|S}} 
        (T_1^{p,j}{Q_1v_1} +T_2^{p,j}{Q_2
          v_2})_{|S}}_{m^p-1/2-\torder^{p,j},\ttau}\\
      +\csp^{1/2} \sum_{j=1}^{m^q} \bignorm{\ttau^{-1/2} e^{\tau\varphi_{|S}} 
        (T_1^{q,j}{v_1}+T_2^{q,j}{v_2})_{|S}}_{m^p, m^q-1/2-\torder^{q,j},\ttau},
  \end{multline}
 for all $v_k = {w_k}_{|\Omega_k}$ with $w_k\in \Cinfc(V_3)$, $\tau\geq
  \tau_3$ and $\csp\geq \csp_3$, for $\tau_3$ and $\csp_3$
  chosen \suff large.

  We choose $\chi$ as in the proof of Theorem~\ref{theorem: unique
    continuation1}
  and we apply estimate~\eqref{eq: Carleman-uc}
  to $v_k=\chi u_k$ as can be done by a density argument. 
  We now sketch how the remainder of the proof can be carried out.

  We have $P_k Q_k (\chi u) = \chi P_k Q_k u + [P_k Q_k,\chi] u$. The term $[P_k Q_k,\chi] u$
  is supported in the set $\Sigma$ introduced in the proof of
  Theorem~\ref{theorem: unique continuation1} and can be handle as it
  is done there. 
  For the first term,  with \eqref{eq: unique continuation - inequality3} we have 
  \begin{align*}
    \Norm{e^{\tau\varphi} \chi P_k Q_k u}_{L^2(\Omega_k)}
    &\lesssim  \sum_{|\mi|\leq m^p + m^q_k -1} 
   \Norm{e^{\tau\varphi} \chi D^\mi u}_{L^2(\Omega_k)} \\
   &\lesssim
   \sum_{|\mi|\leq m^p + m^q_k -1} 
   \Norm{e^{\tau\varphi} D^\mi (\chi u)}_{L^2(\Omega_k)}
   + \sum_{|\mi|\leq m^p + m^q_k -1}\Norm{e^{\tau\varphi} [\chi, D^\mi] u}_{L^2(\Omega_k)}.
  \end{align*}
The second term in the \rhs concerns functions with support located in $\Sigma$ and
their treatment is done as in the proof of
  Theorem~\ref{theorem: unique continuation1}. For the first term we
  have 
  \begin{align*}
   \sum_{|\mi|\leq m^p + m^q_k -1} 
   \Norm{e^{\tau\varphi} D^\mi (\chi u)}_{L^2(\Omega_k)}
   \lesssim 
    \Norm{e^{\tau\varphi} \chi u}_{m^p + m^q_k -1,\ttau_k}
    \lesssim \bigNorm{\ttau_k^{-1}  e^{\tau\varphi} \chi u}_{m^p + m^q_k,\ttau_k},
  \end{align*}
which can be absorbed by the first term in \eqref{eq: Carleman-uc}
by chosing $\csp$ \suff large. 

Next we have
\begin{align*}
  \bignorm{e^{\tau\varphi_{|S}} 
        (T_1^{p,j}{Q_1v_1} +T_2^{p,j}{Q_2 v_2})_{|S}}_{m^p-1/2-\torder^{p,j},\ttau}
  &\leq 
  \bignorm{e^{\tau\varphi_{|S}} 
        (T_1^{p,j}{Q_1v_1} +T_2^{p,j}{Q_2 v_2})_{|S}}_{m^p-\torder^{p,j},\ttau}\\
   &= \sum_{r + |\mi| \atop \leq m^p-\torder^{p,j}}
    \bignorm{ \ttau^r D_\T^\mi  e^{\tau\varphi_{|S}} 
        (T_1^{p,j}{Q_1v_1} +T_2^{p,j}{Q_2 v_2})_{|S}}_{L^2(S)}\\
   &\lesssim
   \sum_{r + |\mi| \atop \leq m^p-\torder^{p,j}}
    \bignorm{ \ttau^r  e^{\tau\varphi_{|S}} D_\T^\mi
        (T_1^{p,j}{Q_1v_1} +T_2^{p,j}{Q_2 v_2})_{|S}}_{L^2(S)}
   .
  \end{align*}
Writing $D_\T^\mi  T_k^{p,j} Q_kv_k = \chi D_\T^\mi  T_k^{p,j} Q_k u_k +
[D_\T^\mi  T_k^{p,j} Q_k, \chi] u$ we have 
\begin{align*}
   &\bignorm{e^{\tau\varphi_{|S}} 
        (T_1^{p,j}{Q_1v_1} +T_2^{p,j}{Q_2 v_2})_{|S}}_{m^p-1/2-\torder^{p,j},\ttau}\\
   &\qquad \quad\lesssim 
   \sum_{r + |\mi| \atop \leq m^p-\torder^{p,j}}
    \bignorm{ \ttau^r  e^{\tau\varphi_{|S}} \chi D_\T^\mi
        (T_1^{p,j}{Q_1u_1} +T_2^{p,j}{Q_2 u_2})_{|S}}_{L^2(S)}\\
   &\qquad \qquad +
   \sum_{r + |\mi| \atop \leq m^p-\torder^{p,j}}
    \bignorm{ \ttau^r  e^{\tau\varphi_{|S}} ([D_\T^\mi
        T_1^{p,j}Q_1, \chi] u_k +[D_\T^\mi T_2^{p,j}Q_2, \chi]  v_2)_{|S}}_{L^2(S)}.
 \end{align*}
The terms $[D_\T^\mi  T_k^{p,j} Q_k, \chi] u$ are supported in the set $\Sigma$ and can
be treated as in the proof of Theorem~\ref{theorem: unique
  continuation1}.  For the first term, with \eqref{eq: unique
  continuation - inequality4} we have
\begin{align*}
    &\bignorm{ \ttau^r  e^{\tau\varphi_{|S}}\chi D_\T^\mi
        (T_1^{p,j}{Q_1u_1} +T_2^{p,j}{Q_2 u_2})_{|S} }_{L^2(S)}\\
    &\qquad \lesssim 
      \sum_{k=1,2} 
      \sum_{|\mi'| \leq |\mi| + m^q_k \atop+
        \torder_k^{p,j} -1}
      \bignorm{ \ttau^r  e^{\tau\varphi_{|S}} (\chi D^{\mi'}
        u_k)_{|S}}_{L^2(S)}\\
       &\qquad \lesssim 
       \sum_{k=1,2} 
      \sum_{|\mi'| \leq |\mi| + m^q_k \atop+
        \torder_k^{p,j} -1}
      \bignorm{ \ttau^r  D^{\mi'}
       (e^{\tau\varphi_k}\chi u_k)_{|S}}_{L^2(S)}+ 
      \sum_{k=1,2} 
      \sum_{|\mi'| \leq |\mi| + m^q_k \atop+
        \torder_k^{p,j} -1}
      \bignorm{ \ttau^r  ([e^{\tau\varphi_{|S}} \chi, D^{\mi'}]
        u_k)_{|S}}_{L^2(S)}
 \end{align*}
 The second term in the \rhs concerns functions with support located in $\Sigma$ and
their treatment is done as in the proof of
  Theorem~\ref{theorem: unique continuation1}.
  For the first term we have $r + |\mi| \leq m^p-\torder^{p,j}$ and $|\mi'| \leq |\mi| + m^q_k +
        \torder_k^{p,j} -1 $
and thus we write 
\begin{align*}
\bignorm{
      \ttau^r  D^{\mi'}(e^{\tau\varphi_k}  \chi u_k)_{|S}}_{L^2(S)}
    \lesssim \bignorm{\trace(e^{\tau\varphi_k}  \chi u_k)}_{m^p + m^q_k-1,0,\ttau}.
\end{align*}
As 
\begin{align*}
  \csp^{1/2}\bignorm{\ttau^{-1/2}
    e^{\tau\varphi_{|S}}\trace( \chi u_k)}_{m^p+m^q_k-1,1/2,\ttau}
  &\geq 
  \csp^{1/2}\bignorm{
    e^{\tau\varphi_{|S}}\trace(\chi u_k)}_{m^p+m^q_k-1,0,\ttau}\\
  &\gtrsim 
  \csp^{1/2}\bignorm{
    \trace(e^{\tau\varphi_{k}}  \chi u_k)}_{m^p+m^q_k-1,0,\ttau}
\end{align*}
we see that the above terms can be absorbed by the second term in
\eqref{eq: Carleman-uc}
by chosing $\csp$ \suff large.

We finally write
\begin{align*}
  &\sum_{j=1}^{m^q} \bignorm{\ttau^{-1/2} e^{\tau\varphi_{|S}} 
        (T_1^{q,j}{v_1}+T_2^{q,j}{v_2})_{|S}}_{m^p,
        m^q-1/2-\torder^{q,j},\ttau}\\
      &\qquad \leq 
       \sum_{j=1}^{m^q} \bignorm{\ttau^{-1/2} e^{\tau\varphi_{|S}} 
        (T_1^{q,j}{v_1}+T_2^{q,j}{v_2})_{|S}}_{m^p,
        m^q-\torder^{q,j},\ttau}\\
      &\qquad \lesssim
      \sum_{ {{|\mi_1|\leq m^p} \atop {|\mi_2|\leq m^q -
        \torder^{q,j}}}\atop {r + |\mi_1|+ |\mi_2| \leq  m^p + m^q -
      \torder^{q,j}}}
  \bignorm{\ttau^{r-1/2}  D^{\mi_1} D_\T^{\mi_2}  e^{\tau\varphi_{|S}} 
    (T_1^{q,j}{v_1}+T_2^{q,j}{v_2})_{|S}},
\end{align*}
 which can be treated as above by using \eqref{eq: unique
  continuation - inequality5}.
We then conclude the proof as 
in that of Theorem~\ref{theorem: unique continuation1}.
\end{proof}


\providecommand{\bysame}{\leavevmode\hbox to3em{\hrulefill}\thinspace}
\providecommand{\MR}{\relax\ifhmode\unskip\space\fi MR }
\providecommand{\MRhref}[2]{%
  \href{http://www.ams.org/mathscinet-getitem?mr=#1}{#2}
}
\providecommand{\href}[2]{#2}
\begin{thebibliography}{}

\end{thebibliography}


\begin{thebibliography}{10}

\bibitem{Barbu:00}
V.~Barbu, \emph{Exact controllability of the superlinear heat equation}, Appl.
  Math. Optim. \textbf{42} (2000), 73--89.

\bibitem{BLR:92}
C.~Bardos, G.~Lebeau, and J.~Rauch, \emph{Sharp sufficient conditions for the
  observation, control, and stabilization of waves from the boundary}, SIAM J.
  Control Optim. \textbf{30} (1992), 1024--1065.

\bibitem{Bellassoued:03}
M.~Bellassoued, \emph{Carleman estimates and distribution of resonances for the
  transparent obstacle and application to the stabilization}, Asymptotic Anal.
  \textbf{35} (2003), 257--279.

\bibitem{BLR:13}
M.~Bellassoued and J.~Le~Rousseau, \emph{Carleman estimates for elliptic
  operators with complex coefficients. {P}art {I}: boundary value problems.},
  J. Math. Pures Appl. \textbf{104} (2015), 657--728.

\bibitem{BY:10}
M.~Bellassoued and M.~Yamamoto, \emph{Carleman estimate with second large
  parameter for second order hyperbolic operators in a {R}iemannian manifold
  and applications in thermoelasticity cases}, Appl. Anal. \textbf{91} (2012),
  no.~1, 35--67.

\bibitem{BK:81}
A.~L. Bukhgeim and M.~V. Klibanov, \emph{Global uniqueness of class of
  multidimensional inverse problems}, Soviet Math. Dokl. \textbf{24} (1981),
  244--247.

\bibitem{CCC:08}
Fioralba Cakoni, Mehmet {\c{C}}ay{\"o}ren, and David Colton, \emph{Transmission
  eigenvalues and the nondestructive testing of dielectrics}, Inverse Problems
  \textbf{24} (2008), no.~6, 065016, 15.

\bibitem{CGH:10}
Fioralba Cakoni, Drossos Gintides, and Houssem Haddar, \emph{The existence of
  an infinite discrete set of transmission eigenvalues}, SIAM J. Math. Anal.
  \textbf{42} (2010), no.~1, 237--255.

\bibitem{CH:13}
Fioralba Cakoni and Houssem Haddar, \emph{Transmission eigenvalues in inverse
  scattering theory}, Inverse problems and applications: inside out. {II},
  Math. Sci. Res. Inst. Publ., vol.~60, Cambridge Univ. Press, Cambridge, 2013,
  pp.~529--580.

\bibitem{Calderon:58}
A.-P. Calder{\'o}n, \emph{Uniqueness in the {C}auchy problem for partial
  differential equations.}, Amer. J. Math. \textbf{80} (1958), 16--36.

\bibitem{Carleman:39}
T.~Carleman, \emph{Sur une probl\`eme d'unicit\'e pour les syst\`emes
  d'\'equations aux d\'eriv\'ees partielles \`a deux variables
  ind\'ependantes}, Ark. Mat. Astr. Fys. \textbf{26B} (1939), no.~17, 1--9.

\bibitem{CPS:07}
David Colton, Lassi P{\"a}iv{\"a}rinta, and John Sylvester, \emph{The interior
  transmission problem}, Inverse Probl. Imaging \textbf{1} (2007), no.~1,
  13--28.

\bibitem{DKSU:09}
D.~Dos Santos~Ferreira, C.~E. Kenig, M.~Salo, and G.~Uhlmann, \emph{Limiting
  carleman weights and anisotropic inverse problems}, Invent. Math.
  \textbf{178} (2009), 119--171.

\bibitem{Eller:00}
M.~Eller, \emph{Carleman estimates with a second large parameter}, Journal of
  Mathematical Analysis and Applications \textbf{249} (2000), 491--514.

\bibitem{EI:00}
M.~Eller and V.~Isakov, \emph{Carleman estimates with two large parameters and
  applications}, Contemporary Math. \textbf{268} (2000), 117--136.

\bibitem{FZ:00}
E.~Fern{\'a}ndez-Cara and E.~Zuazua, \emph{Null and approximate controllability
  for weakly blowing up semilinear heat equations}, Ann. Inst. H. Poincar\'e,
  Analyse non lin. \textbf{17} (2000), 583--616.

\bibitem{FI:96}
A.~Fursikov and O.~Yu. Imanuvilov, \emph{Controllability of evolution
  equations}, vol.~34, Seoul National University, Korea, 1996, Lecture notes.

\bibitem{Hoermander:58}
L.~H{\"o}rmander, \emph{On the uniqueness of the {C}auchy problem}, Math.
  Scand. \textbf{6} (1958), 213--225.

\bibitem{Hoermander:63}
\bysame, \emph{Linear partial differential operators}, Die Grundlehren der
  mathematischen Wissenschaften, Bd. 116, Academic Press Inc., Publishers, New
  York, 1963.

\bibitem{Hoermander:79}
\bysame, \emph{The {W}eyl calculus of pseudo-differential operators}, Comm.
  Pure Appl. Math. \textbf{32} (1979), 359--443.

\bibitem{Hoermander:V3}
\bysame, \emph{The analysis of linear partial differential operators. {III}},
  Grundlehren der Mathematischen Wissenschaften, vol. 274, Springer-Verlag,
  Berlin, 1985, Pseudodifferential operators.

\bibitem{Hoermander:V4}
\bysame, \emph{The analysis of linear partial differential operators. {IV}},
  Grundlehren der Mathematischen Wissenschaften, vol. 275, Springer-Verlag,
  Berlin, 1994, Fourier integral operators.

\bibitem{IIY:03}
O.~Yu. Imanuvilov, V.~Isakov, and M.~Yamamoto, \emph{An inverse problem for the
  dynamical {L}am\'e system with two sets of boundary data}, Comm. Pure Appl.
  Math. \textbf{56} (2003), 1366--1382.

\bibitem{Isakov:98}
V.~Isakov, \emph{Inverse problems for partial differential equations},
  Springer-Verlag, Berlin, 1998.

\bibitem{IK:08}
V.~Isakov and N.~Kim, \emph{Carleman estimates with second large parameter and
  applications to elasticity with residual stress}, Applicationes Mathematicae
  \textbf{35} (2008), 447--465.

\bibitem{JK:85}
D.~Jerison and C.~E. Kenig, \emph{Unique continuation and absence of positive
  eigenvalues for {S}chr\"odinger operators}, Ann. of Math. (2) \textbf{121}
  (1985), no.~3, 463--494, With an appendix by E. M. Stein.

\bibitem{KSU:07}
C.~E. Kenig, J.~Sj{\"o}strand, and G.~Uhlmann, \emph{The {C}alder{\'o}n problem
  with partial data}, Ann. of Math. \textbf{165} (2007), 567--591.

\bibitem{KT:01}
H.~Koch and D.~Tataru, \emph{Carleman estimates and unique continuation for
  second-order elliptic equations with nonsmooth coefficients}, Comm. Pure
  Appl. Math. \textbf{54} (2001), no.~3, 339--360.

\bibitem{KT:02}
\bysame, \emph{Sharp counterexamples in unique continuation for second order
  elliptic equations}, J. Reine Angew. Math. \textbf{542} (2002), 133--146.

\bibitem{Kubo:00}
M.~Kubo, \emph{Uniqueness in inverse hyperbolic problems---{C}arleman estimate
  for boundary value problems}, J. Math. Kyoto Univ. \textbf{40} (2000), no.~3,
  451--473.

\bibitem{LeRousseau:12}
J.~{Le~Rousseau}, \emph{On {C}arleman estimates with two large parameters},
  Indiana Univ. Math. J., to appear \textbf{64} (2015), 55--113.

\bibitem{LRL:12}
J.~Le~Rousseau and G.~Lebeau, \emph{On carleman estimates for elliptic and
  parabolic operators. applications to unique continuation and control of
  parabolic equations}, ESAIM: Control, Optimisation and Calculus of Variations
  \textbf{18} (2012), 712--747.

\bibitem{LR-L:13}
J.~{Le~Rousseau} and Nicolas Lerner, \emph{Carleman estimates for anisotropic
  elliptic operators with jumps at an interface}, Anal. PDE \textbf{6} (2013),
  1601--1648.

\bibitem{LRR:10}
J.~{Le~Rousseau} and L.~Robbiano, \emph{Carleman estimate for elliptic
  operators with coefficents with jumps at an interface in arbitrary dimension
  and application to the null controllability of linear parabolic equations},
  Arch. Rational Mech. Anal. \textbf{105} (2010), 953--990.

\bibitem{LRR:11}
J.~Le~Rousseau and L.~Robbiano, \emph{Local and global {C}arleman estimates for
  parabolic operators with coefficients with jumps at interfaces}, Inventiones
  Math. \textbf{183} (2011), 245--336.

\bibitem{LRR:15}
\bysame, \emph{A spectral inequality for the bi-laplace operator}, preprint
  (2015), 48 pages.

\bibitem{LR:95}
G.~Lebeau and L.~Robbiano, \emph{Contr\^ole exact de l'\'equation de la
  chaleur}, Comm. Partial Differential Equations \textbf{20} (1995), 335--356.

\bibitem{LR:97}
\bysame, \emph{Stabilisation de l'\'equation des ondes par le bord}, Duke Math.
  J. \textbf{86} (1997), 465--491.

\bibitem{Lerner:10}
N.~Lerner, \emph{Metrics on the {P}hase {S}pace and {N}on-{S}elfadjoint
  {P}seudo-{D}ifferential {O}perators}, Pseudo-Differential Operators, Vol.~3,
  Birkh\"auser, Basel, 2010.

\bibitem{LM:68}
J.-L. Lions and E.~Magenes, \emph{Probl\`emes aux {L}imites {N}on
  {H}omog\`enes}, vol.~1, Dunod, 1968.

\bibitem{Miyatake:75}
S.~Miyatake, \emph{Mixed problems for hyperbolic equations of second order with
  first order complex boundary operators}, Japan. J. Math. (N.S.) \textbf{1}
  (1975/76), no.~1, 109--158.

\bibitem{PS:08}
Lassi P{\"a}iv{\"a}rinta and John Sylvester, \emph{Transmission eigenvalues},
  SIAM J. Math. Anal. \textbf{40} (2008), no.~2, 738--753.

\bibitem{PV:16}
V.~Petkov and G.~Vodev, \emph{Asymptotics of the number of interior
  transmission eigenvalues}, Journal of Spectral Theorry (to appear).

\bibitem{Robbiano:13}
Luc Robbiano, \emph{Spectral analysis of the interior transmission eigenvalue
  problem}, Inverse Problems \textbf{29} (2013), no.~10, 104001, 28.

\bibitem{Sakamoto:70}
R.~Sakamoto, \emph{Mixed problems for hyperbolic equations. {I}. {E}nergy
  inequalities}, J. Math. Kyoto Univ. \textbf{10} (1970), 349--373.

\bibitem{Sakamoto:82}
\bysame, \emph{Hyperbolic boundary value problems}, Cambridge University Press,
  Cambridge, 1982, Translated from the Japanese by K. Miyahara.

\bibitem{Schechter:60}
Martin Schechter, \emph{A generalization of the problem of transmission}, Ann.
  Scuola Norm. Sup. Pisa (3) \textbf{14} (1960), 207--236.

\bibitem{Sogge:89}
C.~D. Sogge, \emph{Oscillatory integrals and unique continuation for second
  order elliptic differential equations}, J. Amer. Math. Soc. \textbf{2}
  (1989), no.~3, 491--515.

\bibitem{Tataru:96}
D.~Tataru, \emph{Carleman estimates and unique continuation for solutions to
  boundary value problems}, J. Math. Pures Appl. \textbf{75} (1996), 367--408.

\bibitem{Zuily:83}
C.~Zuily, \emph{Uniqueness and {N}on {U}niqueness in the {C}auchy {P}roblem},
  Birkhauser, Progress in mathematics, 1983.

\end{thebibliography}
\bibliographystyle{amsplain}

\providecommand{\bysame}{\leavevmode\hbox to3em{\hrulefill}\thinspace}
\providecommand{\MR}{\relax\ifhmode\unskip\space\fi MR }
\providecommand{\MRhref}[2]{%
  \href{http://www.ams.org/mathscinet-getitem?mr=#1}{#2}
}
\providecommand{\href}[2]{#2}

\end{document}